\documentclass[12pt]{article}

\usepackage[english]{babel}
\usepackage[T1]{fontenc}
\usepackage{amsfonts}
\usepackage{amsmath}
\usepackage{amssymb}
\usepackage{amsthm}
\usepackage{stmaryrd}
\usepackage{dsfont}
\usepackage{graphicx}
\usepackage{hyperref}
\usepackage{titlesec}
\usepackage{mathabx}
\usepackage{enumitem,tocloft}
\usepackage[top=2.5cm, bottom=2.5cm, left=3cm, right=3cm]{geometry}
\usepackage{xcolor}
\usepackage{fourier-orns}
\usepackage[utf8]{inputenc}
\usepackage{bm}
\usepackage{framed}
\usepackage{supertabular}
\hypersetup{linktocpage,breaklinks=true}
\usepackage{mathrsfs}

\usepackage{csquotes}

\newcommand{\Z}{\mathbb{Z}}
\newcommand{\N}{\mathbb{N}}
\newcommand{\R}{\mathbb{R}}
\newcommand{\C}{\mathbb{C}}
\newcommand{\T}{\mathbb{T}}
\newcommand{\A}{\mathcal{A}}
\newcommand{\leb}{\mathrm{Leb}}
\newcommand{\ld}{\mathrm{L}}
\newcommand{\esp}{(X,\mu)}
\newcommand{\espalg}{(X,\A,\mu)}
\newcommand{\espy}{(Y,\nu)}

\newcommand{\esplebalg}{([0,1],\mathcal{B}([0,1]),\leb)}
\newcommand{\aut}{\mathrm{Aut}(X,\mu)}
\newcommand{\autalg}{\mathrm{Aut}(X,\A,\mu)}
\newcommand{\auty}{\mathrm{Aut}(Y,\nu)}
\newcommand{\orb}{\mathrm{Orb}}
\newcommand{\sn}{\zeta}
\newcommand{\fc}{\mathcal{F}_{\mathcal{C}}}

\newtheorem{theorem}{Theorem}[section]
\newtheorem*{theorem*}{Theorem}
\newtheorem{corollary}[theorem]{Corollary}
\newtheorem*{corollary*}{Corollary}

\newtheorem{proposition}[theorem]{Proposition}
\newtheorem{lemma}[theorem]{Lemma}

\newtheorem*{claim*}{Claim}

\newtheorem{theoremletter}{Theorem}

\theoremstyle{definition}
\newtheorem{definition}[theorem]{Definition}
\newtheorem{remark}[theorem]{Remark}
\newtheorem{question}[theorem]{Question}

\newtheorem{manualremarknumber}{Remark}
\newenvironment{manualremark}[1]{%
	\IfBlankTF{#1}
	{}
	{\renewcommand\themanualremarknumber{#1}}%
	\manualremarknumber
}{\endmanualremarknumber}

\title{Rank-one systems, flexible classes and Shannon orbit equivalence}
\author{Corentin Correia}
\date{October 19, 2024}

\begin{document}
	
	\maketitle
	
	\begin{abstract}
		We build a Shannon orbit equivalence between the universal odometer and a variety of rank-one systems. This is done in a unified manner, using what we call flexible classes of rank-one transformations. Our main result is that every flexible class contains an element which is Shannon orbit equivalent to the universal odometer. Since a typical example of flexible class is $\{T\}$ when $T$ is an odometer, our work generalizes a recent result by Kerr and Li, stating that every odometer is Shannon orbit equivalent to the universal odometer.\par
		When the flexible class is a singleton, the rank-one transformation given by the main result is explicit. This applies to odometers and Chacon's map. We also prove that strongly mixing systems, systems with a given eigenvalue, or irrational rotations whose angle belongs to any fixed nonempty open subset of the real line form flexible classes. In particular, strong mixing, rationality or irrationality of the eigenvalues are not preserved under Shannon orbit equivalence.
	\end{abstract}
	
	\setcounter{tocdepth}{2}
	\tableofcontents
	
	\section{Introduction}
	
	At the level of ergodic probability measure-preserving bijections, \emph{quantitative orbit equivalence} aims at bridging the gap between the well-studied but very complicated relation of \emph{conjugacy}, and the trivial relation of \emph{orbit equivalence}, which is equality of orbits up to conjugacy.\par
	To be more precise, given two ergodic probability measure-preserving bijections $S$ and $T$ on a standard atomless probability space $\espalg$, if $S$ and some system $\Psi^{-1}T\Psi$ conjugate to $T$ have the same orbits, then $S$ and $T$ are said to be \textit{orbit equivalent} and the probability measure-preserving bijection $\Psi\colon X\to X$ is called an \textit{orbit equivalence} between $T$ and $S$. Dye's theorem~\cite{dyeGroupsMeasurePreserving1959} states that, if $S$ and $T$ are ergodic, then they are orbit equivalent.\par
	To get an interesting theory, let us define the \textit{cocycles} associated to $\Psi$, these are the integer-valued functions $c_S$ and $c_T$ defined by $Sx=\Psi^{-1}T^{c_S(x)}\Psi(x)$ and $Tx=\Psi S^{c_T(x)}\Psi^{-1}(x)$. \emph{Shannon orbit equivalence} requires that there exists an orbit equivalence whose cocycles are Shannon, meaning that the partitions associated to $c_S$ and $c_T$ are both of finite entropy. For $\varphi$\emph{-integrable orbit equivalence} we ask that both integrals $\int_X{\varphi(|c_S(x)|)\mathrm{d}\mu(x)}$ and $\int_X{\varphi(|c_T(x)|)\mathrm{d}\mu(x)}$ are finite. In the particular case of a linear map $\varphi$, $\varphi$-integrable orbit equivalence exactly requires the integrability of the cocycles, and is simply called integrable orbit equivalence.\par
	Belinskaya's theorem \cite{belinskayaPartitionsLebesgueSpace1969} implies that integrable orbit equivalence is exactly flip-conjugacy ($S$ and $T$ are flip-conjugate if $S$ is conjugate to $T$ or $T^{-1}$). In fact it only requires that one of the two cocycles is integrable. Carderi, Joseph, Le Maître and Tessera \cite{carderiBelinskayaTheoremOptimal2023} proved that this result is optimal, meaning that $\varphi$-integrable orbit equivalence never implies flip-conjugacy for a sublinear map $\varphi$. Moreover, $\varphi$-integrable orbit equivalence implies Shannon orbit equivalence when $\varphi$ is asymptotically greater than $\log$. An impressive result of Kerr and Li \cite{kerrEntropyVirtualAbelianness2024} guarantees that these relations are not trivial: entropy is preserved under Shannon orbit equivalence (and this is the only invariant that we know of). As a consequence, two transformations with different entropies can neither be Shannon orbit equivalent nor $\varphi$-integrably orbit equivalent for any $\varphi$ greater than $\log$.
	
	\begin{figure}[ht]
		\centering
		\includegraphics[width=0.9\linewidth]{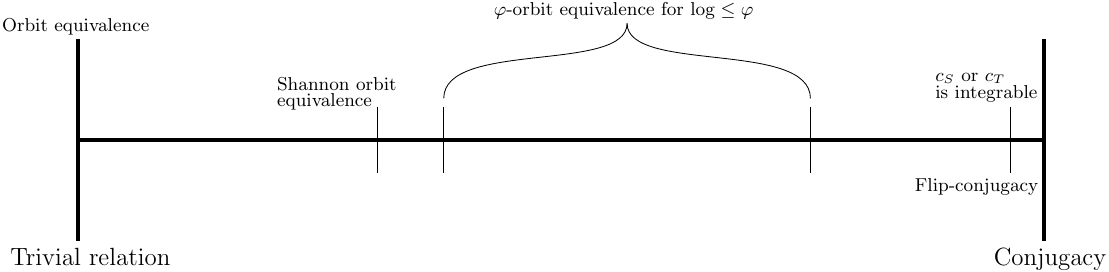}
		\caption{Here is a schematic view of the interplay between the relations on ergodic bijections we have seen so far.}
	\end{figure}
	
	Historically, the question of preservation of entropy in quantitative orbit equivalence was asked in the more general setting of group actions. We do not give any definition in this setting, as the paper is only about probability measure-preserving bijections $S$, which can be seen as $\mathbb{Z}$-actions via $(n,x)\in \Z\times X\mapsto S^nx$. Austin \cite{austinBehaviourEntropyBounded2016a} showed that integrable orbit equivalence between actions of infinite finitely generated amenable groups preserves entropy. Kerr and Li \cite{kerrEntropyShannonOrbit2021, kerrEntropyVirtualAbelianness2024} then generalized this result, replacing integrable orbit equivalence by Shannon orbit equivalence, and going beyond the amenable case using sofic entropy.
	
	\paragraph{The universal odometer and a theorem of Kerr and Li \cite{kerrEntropyVirtualAbelianness2024}.}
	In \cite{carderiBelinskayaTheoremOptimal2023}, the statement about $\varphi$-integrable orbit equivalence in the sublinear case is the following. This gives a result on Shannon orbit equivalence since this is implied by $\varphi$-integrable orbit equivalence for $\varphi$ greater than $\log$.
	
	\begin{theorem*}[Carderi, Joseph, Le Maître, Tessera \cite{carderiBelinskayaTheoremOptimal2023}]\label{cjlt}
		Let $\varphi\colon\R_+\to\R_+$ be a sublinear function. Let $S$ be an ergodic probability measure-preserving transformation and assume that $S^n$ is ergodic
		for some $n\geq 2$. Then there is another ergodic probability measure-preserving transformation
		$T$ such that $S$ and $T$ are $\varphi$-integrably orbit equivalent but not flip-conjugate.
	\end{theorem*}
	
	\begin{corollary*}[Carderi, Joseph, Le Maître, Tessera \cite{carderiBelinskayaTheoremOptimal2023}]\label{cjlt2}
		Let $S$ be an ergodic probability measure-preserving transformation and assume that $S^n$ is ergodic
		for some $n\geq 2$. Then there is another ergodic probability measure-preserving transformation
		$T$ such that $S$ and $T$ are Shannon orbit equivalent but not flip-conjugate.
	\end{corollary*}
	
	The proof is constructive and the resulting transformation $T$ is built so that $T^n$ is not ergodic. It is natural to wonder whether this statement holds for systems $T$ without ergodic non-trivial powers. A well-known example of such a system is the universal odometer.
	
	\begin{question}
		Which systems are Shannon orbit equivalent to the universal odometer?
	\end{question}
	
	A first answer is given by Kerr and Li.
	
	\begin{theorem*}[Kerr, Li \cite{kerrEntropyVirtualAbelianness2024}]
		Every odometer is Shannon orbit equivalent to the universal odometer.
	\end{theorem*}
	
	Odometers are exactly probability measure-preserving bijections admitting a nested sequence of partitions of the space, each of them being a Rokhlin tower, and increasing to the $\sigma$-algebra $\A$, see Figure~\ref{copie1} (we refer the reader to the end of Section~\ref{cutsta} for concrete examples with adding machines). Kerr and Li use this combinatorial specificity of these bijections to build an orbit equivalence between them.
	
	\begin{figure}[ht]
		\centering
		\includegraphics[width=0.6\linewidth]{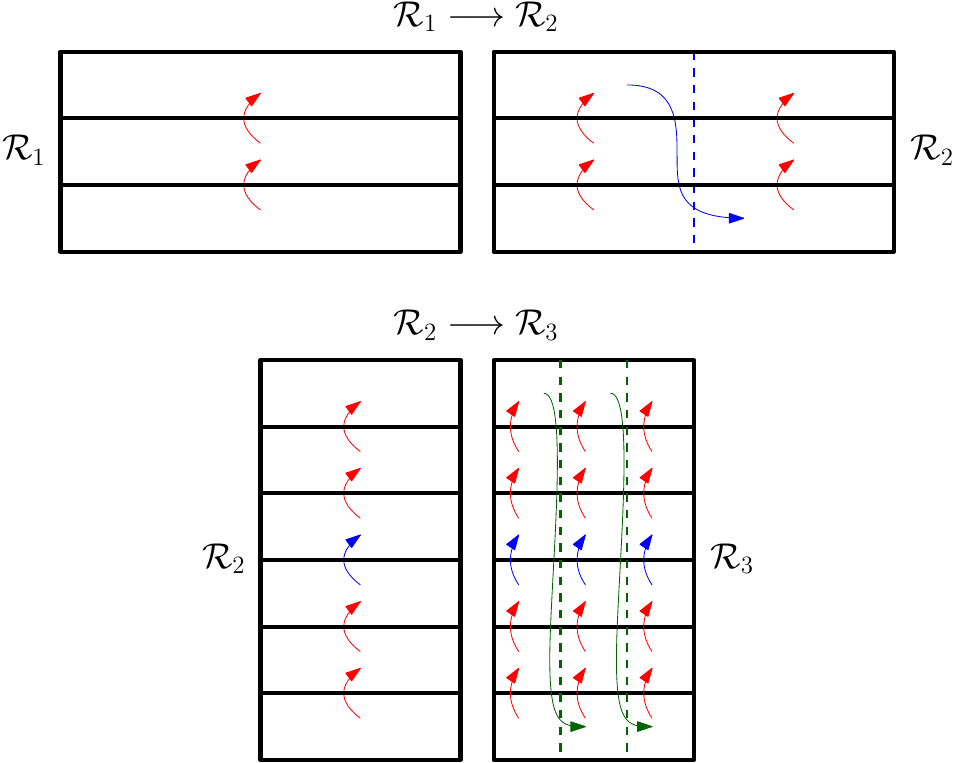}
		\caption{In this example, $(\mathcal{R}_n)$ denotes the nested sequence of Rokhlin towers defining an odometer. Dividing $\mathcal{R}_1$ in two sub-towers and stacking them, this gives the next tower $\mathcal{R}_2$. From $\mathcal{R}_2$, $\mathcal{R}_3$ is defined by dividing in three sub-towers and stacking them.}
		\label{copie1}
	\end{figure}
	
	\paragraph{Rank-one systems.}
	The aim of the paper is to extend Kerr and Li's result to \textit{rank-one bijections}. These are more general transformations admitting a nested sequence of Rokhlin towers increasing to the $\sigma$-algebra $\A$ but the towers do not necessarily partition the space. This means that from a tower to the next one, we need to add some parts of the space which are not covered by the previous tower, called \textit{spacers}, so that the measure of the subset covered by the $n$-th tower tends to $1$ as $n$ goes to $+\infty$. As illustrated in Figure~\ref{copie2}, to get the next tower, the current one is subdivided in sub-towers which are stacked with optional spacers between them. The number of sub-towers is called the cutting parameter and the number of consecutive spacers between these sub-towers are the spacing parameters (see Definition~\ref{defr1}). For example, an odometer admits a cutting-and-stacking construction with spacing parameters equal to zero at each step.
	
	\begin{figure}[ht]
		\centering
		\includegraphics[width=0.5\linewidth]{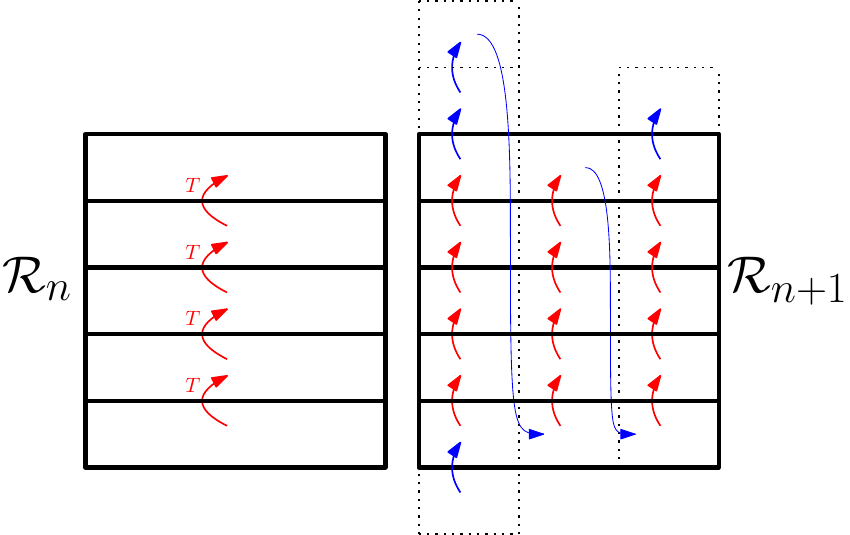}
		\caption{In this example, there are four spacers and the cutting parameter is three.}
		\label{copie2}
	\end{figure}
	
	Rank-one systems all have entropy zero. They include systems with \textit{discrete spectrum} (\cite{juncoTransformationsDiscreteSpectrum1976}), also called \textit{compact} systems. Such systems are not weakly mixing and are completely classified up to conjugacy by their point spectrum (\cite{halmosOperatorMethodsClassical1942}). Examples include odometers and irrational rotations.\par
	The family of rank-one systems is much richer than its subclass of discrete spectrum systems. Indeed, the latter are not weakly mixing whereas there exist strongly mixing systems of rank one, and also rank-one systems which are weakly mixing but not strongly mixing (Chacon's map was the first example of such a system and opened the study of rank-one systems). Rank-one systems can have irrational eigenvalues (i.e.~of the form $\exp{(2i\pi\theta)}$ with irrational numbers $\theta$), it is the case of irrational rotations, whereas odometers only have rational eigenvalues. The reader may refer to the complete survey of Ferenczi \cite{ferencziSystemsFiniteRank1997} about rank-one systems and more generally systems of finite rank.\par
	The combinatorial structure of a general rank-one system does not differ too much from the structure of an odometer but the systems can have completely different properties, thus this class may extend the result of Kerr and Li and provide interesting flexibility results about Shannon orbit equivalence.
	
	\paragraph{A first extension of Kerr and Li's theorem.}
	The construction of an orbit equivalence between the universal odometer $S$ and any rank-one system $T$ is a natural generalization of Kerr and Li's method for the universal odometer and any odometer (see Remark~\ref{oe}). The difficulty is to quantify the cocycles.\par
	At the beginning of our work, we first proved that the Shannon orbit equivalence established by Kerr and Li in \cite{kerrEntropyVirtualAbelianness2024} is actually a $\varphi$-integrable orbit equivalence for any $\varphi\colon\R_+\to\R_+$ with $\varphi(t)=o(t^{1/3})$. We then generalized this to rank-one systems called BSP, for "bounded-spacing-parameter", see Definition~\ref{cspbsp}. This notion of BSP systems was already introduced by Gao and Ziegler in \cite{gaoTopologicalMixingProperties2019a}, using the symbolic definition of rank-one systems (in this paper we will only consider the cutting-and-stacking definition of rank-one systems, which is often more appropriate for constructions in a measure-theoretic setting).
	
	\begin{theoremletter}\label{thetheorem}
		Every BSP rank-one system is $\varphi$-integrably orbit equivalent to the universal odometer for any $\varphi\colon\R_+\to \R_+$ satisfying \mbox{$\varphi(t)\underset{t\to +\infty}{=}o\left (t^{1/3}\right )$}.
	\end{theoremletter}
	
	Therefore $\varphi$-integrable orbit equivalence, for a $\varphi$ as in the above theorem, and Shannon orbit equivalence do not preserve weak mixing since Chacon's map is a BSP rank-one system.\par
	Now the goal is to get a result for systems of rank one outside the class of BSP systems. For this purpose, we find a more general framework with the notion of \textit{flexible classes}, and a general statement (Theorem~\ref{thfc0}) implying Theorem~\ref{thetheorem} and other flexibility results (Theorems~\ref{thetheoremB},~\ref{thetheoremC},~\ref{thetheoremD}). Theorem~\ref{uncountable} is a refinement of Theorem~\ref{thetheoremB}.
	
	\paragraph{A modified strategy.}
	We first have to understand why the quantification of the cocycles is more difficult to determine for general rank-one systems than for odometers (or even for BSP systems in Theorem~\ref{thetheorem}). In \cite{kerrEntropyVirtualAbelianness2024}, the quantification of the cocycles relies on a series whose terms vanish to zero as the cutting parameters get larger and larger. The key is then to get quickly increasing cutting parameters for the series to converge. In order to do so, it suffices to skip steps in the cutting-and-stacking process, i.e.~from the $n$-th Rokhlin tower, we can directly build the $(n+k)$-th Rokhlin for $k$ so big that the new cutting parameter is large enough. In other words, we can recursively choose the cutting parameters so that they increase quickly enough.\par
	When the rank-one system is not an odometer, we need an asymptotic control on the spacing parameters (recall that they are zero for an odometer) for the cocycles to be well quantified. When skipping steps in the cutting-and-stacking method, the spacing parameters may increase too quickly, preventing us from quantifying the cocycles. As we will see in Lemma~\ref{skip}, we do not have this problem with BSP rank-one systems.\par
	When the rank-one system is not BSP, skipping steps in the cutting-and-stacking construction is not relevant as it may improperly change the spacing parameters. In Section~\ref{towardsfc} (see Lemma~\ref{dependence}), we will notice that the construction of Kerr and Li enables us to build the universal odometer $S$ while we are building the rank-one system $T$, focusing only on the combinatorics behind the systems, whereas for Kerr and Li $T$ and its cutting-and-stacking settings are fixed and $S$ is built from these data. This new strategy will enable us to have a result for systems of rank one outside the class of BSP systems, with the notion of flexible class.
	
	\paragraph{Flexible classes.}
	A flexible class (see Definition~\ref{deffc}) is basically a class of rank-one systems satisfying a common property (e.g.~the set of strongly mixing rank-one systems), with the following two requirements. We first ask for a sufficient condition, given by a set $\fc$, on the first $n$ cutting and spacing parameters (for all integers $n\geq 0$) for the underlying rank-one system to be in this class. Secondly, given a sequence of $n$ cutting and spacing parameters in $\fc$ (they will be the first $n$ parameters of a cutting-and-stacking construction), we require that it can be completed in a sequence of $n+1$ parameters in $\fc$, with infinitely many choices for the $(n+1)$-th cutting parameters, and with the appropriate asymptotic control on the $(n+1)$-th spacing parameters.\par
	The idea is to inductively choose the parameters so that the cutting parameters increase fastly enough, with the appropriate asymptotics on the spacing parameters, and the underlying rank-one system has the desired property, namely the system is in the flexible class that we consider.\par
	The general statement on flexible classes is the following.
	
	\begin{theoremletter}[see Theorem~\ref{thfc}]\label{thfc0}
		Let $\varphi\colon\R_+\to \R_+$ be a map satisfying $\varphi(t)\underset{t\to +\infty}{=}o\left (t^{1/3}\right )$. If $\mathcal{C}$ is a flexible class, then there exists $T$ in $\mathcal{C}$ which is $\varphi$-integrably orbit equivalent to the universal odometer.
	\end{theoremletter}
	
	A very interesting phenomenon is when a rank-one system $T$ is flexible, meaning that $\{T\}$ is a flexible class. This first means that given the parameters of a cutting-and-stacking construction of $T$, it is possible to change the $(n+1)$-th parameters so that they have the desired asymptotic control, and to inductively do so for every $n$ so that the underlying rank-one system is again $T$. We do not know if every rank-one system is flexible. Secondly, Theorem~\ref{thfc0} is an existence result and when a flexible class is a singleton $\{T\}$, this statement provides a concrete example of rank-one system which is $\varphi$-integrably orbit equivalent to the universal odometer.\par
	The following proposition gives examples of flexible classes.
	
	\begin{proposition}[see Proposition~\ref{fc}]\label{fc0}
		\begin{enumerate}
			\item Every BSP rank-one system is flexible.
			\item For every nonempty open subset $\mathcal{V}$ of $\R$, the set $\{R_{\theta}\mid\theta\in\mathcal{V}\cap (\R\setminus\mathbb{Q})\}$ is a flexible class.
			\item For every irrational number $\theta$, the class of rank-one systems which have $e^{2i\pi\theta}$ as an eigenvalue is flexible.
			\item The class of strongly mixing rank-one systems is flexible.
		\end{enumerate}
	\end{proposition}
	
	Proving that a BSP system is flexible is not difficult and we rely on the fact that bounded spacing parameters already have the desired asymptotics even though we skip steps in the cutting-and-stacking process for the cutting parameters to increase quickly enough (see Section~\ref{proofBSP}). We use a construction by Drillick, Espinosa-Dominguez, Jones-Baro, Leng, Mandelshtam and Silva \cite{drillickNonrigidRankoneInfinite2023} to prove Proposition~\ref{fc0} for irrational rotations (see Section~\ref{proofirra}). We also consider a construction by Danilenko and Vieprik~\cite{danilenkoExplicitRank1Constructions2023} for the rank-one systems with a given eigenvalue (see Section~\ref{proofeigenvalue}). Finally, Ornstein~\cite{ornsteinRootProblemErgodic1972} gives the first example of strongly mixing rank-one systems and the fact that these systems form a flexible class follows from his construction (see Section~\ref{proofstrmix}).\par
	Combined with Proposition~\ref{fc0}, Theorem~\ref{thfc0} provides four flexibility results. The first one is Theorem~\ref{thetheorem} stated above, this is a generalization of Kerr and Li's theorem. The second one is another result with almost explicit examples of systems which are $\varphi$-integrably orbit equivalent to the universal odometer.
	
	\begin{theoremletter}\label{thetheoremB}
		Let $\varphi\colon\R_+\to \R_+$ be a map satisfying \mbox{$\varphi(t)\underset{t\to +\infty}{=}o\left (t^{1/3}\right )$}. The set of irrational numbers $\theta$ whose associated irrational rotation is $\varphi$-integrably orbit equivalent to the universal odometer is dense in $\R$.
	\end{theoremletter}
	
	The point spectrum of $R_{\theta}$ is exactly the circle subgroup generated by $\exp{(2i \pi\theta)}$ and the eigenvalues of the universal odometer are rational, so Theorem~\ref{thetheoremB} implies that there exist two Shannon orbit equivalent systems (more specifically $\varphi$-integrably orbit equivalent with \mbox{$\varphi(t)\underset{t\to +\infty}{=}o\left (t^{1/3}\right )$}), with non-trivial point spectrums and such that $1$ is the only common eigenvalue.\par
	The way we prove Theorem~\ref{thfc0} will enable us to get the following refinement, its proof is written at the end of the paper.
	
	\begin{theoremletter}\label{uncountable}
		For every map $\varphi\colon\R_+\to \R_+$ satisfying \mbox{$\varphi(t)\underset{t\to +\infty}{=}o\left (t^{1/3}\right )$}, and for every nonempty open subset $\mathcal{V}$ of $\R$, the set of irrational numbers $\theta\in\mathcal{V}$ whose associated irrational rotation is $\varphi$-integrably orbit equivalent to the universal odometer is uncountable.
	\end{theoremletter}
	
	\begin{question}
		Let us consider the set of irrational numbers $\theta$ whose associated irrational rotation is $\varphi$-integrably orbit equivalent to the universal odometer. Is this set conull with respect to the Lebesgue measure? equal to the set of irrational numbers?
	\end{question}
	
	Finally we get the following corollaries, providing implicit examples.
	
	\begin{theoremletter}\label{thetheoremC} For every map $\varphi\colon\R_+\to \R_+$ satisfying \mbox{$\varphi(t)\underset{t\to +\infty}{=}o\left (t^{1/3}\right )$}, and for every irrational number $\theta$, there exists a rank-one system which has $e^{2i\pi\theta}$ as an eigenvalue and which is $\varphi$-integrably orbit equivalent to the universal odometer.
	\end{theoremletter}
	
	\begin{theoremletter}\label{thetheoremD} For every map $\varphi\colon\R_+\to \R_+$ satisfying \mbox{$\varphi(t)\underset{t\to +\infty}{=}o\left (t^{1/3}\right )$}, there exists a strongly mixing rank-one system which is $\varphi$-integrably orbit equivalent to the universal odometer.
	\end{theoremletter}
	
	As $\exp{(2i\pi\theta)}$ is an eigenvalue of the irrational rotation of angle $\theta$, and as we do not know if Theorem~\ref{thetheoremB} holds for every irrational number $\theta$, Theorem~\ref{thetheoremC} then completes this statement with a weaker result for the remaining $\theta$.\par
	Theorem~\ref{thetheoremD} implies that $\varphi$-integrable orbit equivalence, with \mbox{$\varphi(t)\underset{t\to +\infty}{=}o\left (t^{1/3}\right )$}, and Shannon orbit equivalence do not preserve strong mixing. This is also a consequence of the result from~\cite{carderiBelinskayaTheoremOptimal2023}. Indeed if $S$ is strongly mixing, then all its non-trivial powers are ergodic and the statements give some $T$ with a non-trivial power which is not ergodic, so $T$ is not strongly mixing. Here Theorem~\ref{thetheoremD} gives an example starting from a very non-strongly mixing system $S$ (the universal odometer). Finally, note that strongly mixing systems are not BSP. This is a consequence of Theorem~1.3 in \cite{gaoTopologicalMixingProperties2019a}: BSP rank-one systems are not topologically mixing, therefore they are not measure-theoretically strongly mixing.
	
	\paragraph{Further comments.}
	As they both preserve entropy, we may wonder whether there is a connection between Shannon orbit equivalence (or more generally $\varphi$-integrable orbit equivalence for $\varphi$ greater than $\log$) and \textit{even Kakutani equivalence}. Two ergodic probability measure-preserving bijections $S$ and $T$, respectively acting on $\esp$ and $\espy$, are evenly Kakutani equivalent if there exist measurable subsets $A\subseteq X$ and $B\subseteq Y$ with equal measure, i.e.~$\mu(A)=\nu(B)$, such that the induced maps $S_A$ and $T_B$ are conjugate. Even Kakutani equivalence is an equivalence relation, contrarily to Shannon orbit equivalence and $\varphi$-integrable orbit equivalence a priori (except for linear maps $\varphi$, by Belinskaya's theorem). The theory of Ornstein, Rudolph and Weiss \cite{ornsteinEquivalenceMeasurePreserving1982} gives a complete classification up to even Kakutani equivalence among loosely Bernoulli (LB) systems and entropy is a complete invariant. Moreover the class of LB systems is closed by even Kakutani equivalence, meaning that if $S$ is LB and equivalent to $T$, then $T$ is also LB.\par
	Rank-one systems are zero-entropy and LB, and by Theorems~\ref{thetheorem},~\ref{thetheoremB},~\ref{thetheoremC} and~\ref{thetheoremD}, some of them are Shannon orbit equivalent to the universal odometer.
	
	\begin{question}
		Does even Kakutani equivalence imply Shannon orbit equivalence or $\varphi$-integrable orbit equivalence for some $\varphi$?
	\end{question}
	
	In a forthcoming paper we will provide a new construction of orbit equivalence in order to prove that the converse is false: for every $\varepsilon>0$, there exists a non-LB system which is $(x\mapsto x^{\frac{1}{2}-\varepsilon})$-integrably orbit equivalent to the dyadic odometer. So $(x\mapsto x^{\frac{1}{2}-\varepsilon})$-integrable orbit equivalence and Shannon orbit equivalence do not imply even Kakutani equivalence.
	
	\paragraph{Outline of the paper.}
	
	After a few preliminaries in Section~\ref{preliminaries}, rank-one systems are defined in Section~\ref{sectionr1} using the cutting-and-stacking method. We also define the central notion of flexible classes of rank-one transformations. In Section~\ref{sectionproofprop}, we prove Proposition~\ref{fc0} (Proposition~\ref{fc} in Section~\ref{sectionr1}), i.e.~we show that the classes mentionned in Theorem~\ref{thfc0} (Theorem~\ref{thfc} in Section~\ref{sectionr1}) are flexible. It remains to show that every flexible class admits an element which is $\varphi$-integrably orbit equivalent to the universal odometer (Theorem~\ref{thfc}). In Section~\ref{sectionconstr}, we will describe the construction of Kerr and Li, generalized to rank-one systems, and establish that this is an orbit equivalence with some important properties preparing for the proof of Theorem~\ref{thfc}. Theorems~\ref{thetheorem},~\ref{thetheoremB},~\ref{thetheoremC} and~\ref{thetheoremD} directly follows from Proposition~\ref{fc} and Theorem~\ref{thfc}. We prove Theorem~\ref{uncountable} at the end of the paper.
	
	\paragraph{Acknowledgements.}

	I am very grateful to my advisors François Le Maître and Romain Tessera for their support, fruitful discussions and valuable advice. I also wish to thank Mathieu Da Silva, Victor Dubach and Fabien Hoareau for their useful comments on the paper. Finally, I am very grateful to the referee for their careful reading, for many corrections and for pointing out mistakes in the proof of an earlier version of Proposition~\ref{fc0}.
	
	\section{Preliminaries}\label{preliminaries}
	
	\paragraph{Basics of ergodic theory.}
	
	The probability space $\espalg$ is assumed to be standard and atomless. Such a space is isomorphic to $\esplebalg$, i.e.~there exists a bimeasurable bijection $\Psi\colon X\to [0,1]$ (defined almost everywhere) such that $\Psi_{\star}\mu=\mathrm{Leb}$, where $\Psi_{\star}\mu$ is defined by $\Psi_{\star}\mu(A)=\mu(\Psi^{-1}(A))$ for every measurable set $A$. We consider maps $T\colon X\to X$ acting on this space and which are bijective, bimeasurable and \textbf{probability measure-preserving} (\textbf{p.m.p.}), meaning that $\mu(T^{-1}(A))=\mu(A)$ for all measurable sets $A\subseteq X$, and the set of these transformations is denoted by $\autalg$, or simply $\aut$, two such maps being identified if they coincide on a measurable set of full measure. In this paper, elements of $\aut$ are called \textbf{transformations} or (\textbf{dynamical}) \textbf{systems}.\par
	A measurable set $A\subseteq X$ is $T$\textbf{-invariant} if $\mu(T^{-1}(A)\Delta A)=0$, where $\Delta$ denotes the symmetric difference. A transformation $T\in\aut$ is said to be \textbf{ergodic} if every $T$-invariant set is of measure $0$ or $1$. If $T$ is ergodic, then $T$ is \textbf{aperiodic}, i.e.~$T^n(x)\not =x$ for almost every $x\in X$ and for every $n\in\Z\setminus\{0\}$, or equivalently the $T$\textbf{-orbit} of $x$, denoted by $\orb_T(x)\coloneq \{T^n(x)\mid n\in\Z\}$, is infinite for almost every $x\in X$.\par
	$T$ is \textbf{weakly mixing} if
	$$\frac{1}{n}\sum_{k=0}^{n}{\left |\mu(A\cap T^{-n}(B))-\mu(A)\mu(B)\right |}\underset{n\to +\infty}{\to}0$$
	for every measurable sets $A,B$.
	$T$ is \textbf{strongly mixing} if
	$$\left |\mu(A\cap T^{-n}(B))-\mu(A)\mu(B)\right |\underset{n\to +\infty}{\to}0$$
	for every measurable sets $A,B$. It is not difficult to prove that strong mixing implies weak mixing and that the latter implies ergodicity.\par
	The notions of weak mixing and ergodicity can be translated in terms of eigenvalues. Denoting by $\ld^2\espalg$ the space of complex-valued and square-integrable functions defined on $X$, a complex number $\lambda$ is an \textbf{eigenvalue} of $T$ if there exists $f\in\ld^2\espalg\setminus\{0\}$ such that $f\circ T=\lambda f$ almost everywhere ($f$ is then called an eigenfunction). An eigenvalue $\lambda$ is automatically an element of the unit circle $\T\coloneq \{z\in\C\mid |z|=1\}$. The \textbf{point spectrum} of $T$ is then the set of all its eigenvalues. Notice that $\lambda=1$ is always an eigenvalue since the constant functions are in its eigenspace. Finally $T$ is ergodic if and only if the constant functions are the only eigenfunctions with eigenvalue one, in other words the eigenspace of $\lambda=1$ is the line of constant functions (we say that it is a simple eigenvalue). If $T$ is ergodic, it is weakly mixing if and only if the only eigenvalue of $T$ is $1$. For a complete survey on spectral theory for dynamical systems, the reader may refer to \cite{vianaFoundationsErgodicTheory2016}.\par
	All the properties that we have introduced are preserved under conjugacy. Two transformations $S\in\aut$ and $T\in\mathrm{Aut}(Y,\nu)$ are \textbf{conjugate} if there exists a bimeasurable bijection $\Psi\colon X\to Y$ such that $\Psi_{\star}\mu=\nu$ and $\Psi\circ S=T\circ\Psi$ almost everywhere. Some classes of transformations have been classified up to conjugacy, the two examples to keep in mind are the following. By Ornstein \cite{ornsteinBernoulliShiftsSame1970}, entropy is a total invariant of conjugacy among Bernoulli shifts, and Ornstein and Weiss \cite{ornsteinEntropyIsomorphismTheorems1987} generalized this result for Bernoulli shifts of amenable groups. For more details about entropy, see \cite{downarowiczEntropyDynamicalSystems2011} for non necessarily invertible transformations $T\colon X\to X$, and \cite{kerrErgodicTheoryIndependence2016} more generally for actions of amenable groups. Finally Halmos and von Neumann \cite{halmosOperatorMethodsClassical1942} showed that two ergodic systems with discrete spectrums are conjugate if and only if they have equal point spectrums (a system has discrete spectrum if the span of all its eigenfunctions is dense in $\ld^2\espalg$).
	
	\paragraph{Quantitative orbit equivalence.}
	
	The conjugacy problem in full generality is very complicated (see \cite{foremanConjugacyProblemErgodic2011}). We now give the formal definition of orbit equivalence, which is a weakening of the conjugacy problem.
	
	\begin{definition}
		Two aperiodic transformations $S\in \aut$ and $T\in \auty$ are \textbf{orbit equivalent} if there exist a bimeasurable bijection $\Psi\colon X\to Y$ satisfying $\Psi_{\star}\mu=\nu$, such that $\mathrm{Orb}_S(x)=\mathrm{Orb}_{\Psi^{-1}T\Psi}(x)$ for almost every $x\in X$. The map $\Psi$ is called an \textbf{orbit equivalence} between $S$ and $T$.\par
		We can then define the \textbf{cocycles} associated to this orbit equivalence. These are measurable functions $c_S\colon X\to\mathbb{Z}$ and $c_T\colon Y\to\mathbb{Z}$ defined almost everywhere by
		$$Sx=\Psi^{-1}T^{c_S(x)}\Psi(x)\text{ and }Ty=\Psi S^{c_T(y)}\Psi^{-1}(y)$$
		($c_S(x)$ and $c_T(y)$ are uniquely defined by aperiodicity).
	\end{definition}
	
	Given a function $\varphi\colon\R_+\to\R_+$, a measurable function $f\colon X\to\Z$ is said to be $\varphi$\textbf{-integrable} if
	$$\int_{X}{\varphi(|f(x)|)\mathrm{d}\mu}<+\infty.$$
	For example integrability is exactly $\varphi$-integrability when $\varphi$ is non-zero and linear. Then a weaker quantification on cocycles is the notion of $\varphi$-integrability for a \textit{sublinear} map $\varphi$, meaning that $\lim_{t\to +\infty}{\varphi(t)/t}=0$. Two transformations in $\aut$ are said to be $\varphi$\textbf{-integrably orbit equivalent} if there exists an orbit equivalence between them whose associated cocycles are $\varphi$-integrable. Another form of quantitative orbit equivalence is Shannon orbit equivalence. We say that a measurable function $f\colon X\to\Z$ is \textbf{Shannon} if the associated partition $\{f^{-1}(n)\mid n\in\Z\}$ of $X$ has finite entropy. Two transformations in $\aut$ are \textbf{Shannon orbit equivalent} if there exists an orbit equivalence between them whose associated cocycles are Shannon.
	
	\section{Rank-one systems}\label{sectionr1}
	
	\subsection{The cutting-and-stacking method}\label{cutsta}
	
	Before the definition of a rank-one system (Definition~\ref{defr1}), and for the definition of flexible classes (Definition~\ref{deffc}), we need to define sequences of integers which will provide the combinatorial data of a rank-one system, namely the cutting and spacing parameters.
	
	\begin{definition}\label{defparam}
		By a \textbf{cutting and spacing parameter}, we mean a tuple of the form
		$$(q,(\sigma_{.,0},\ldots,\sigma_{.,q}))$$
		with an integer $q\geq 2$ (the \textbf{cutting parameter}) and non-negative integers $\sigma_{.,0},\ldots,\sigma_{.,q}$ (the \textbf{spacing parameters}), and we denote by $\mathcal{P}$ the set of all cutting and spacing parameters. We also define the set of finite sequences of cutting and spacing parameters:
		$$\mathcal{P}^{*}\coloneq \bigcup_{n\in\N}{\mathcal{P}^n}.$$
		Given a sequence of cutting and spacing parameters $\bm{p}=(q_k,(\sigma_{k,0},\ldots,\sigma_{k,q_k}))_{k\geq 0}\in\mathcal{P}^{\N}$ and an integer $n\geq 0$, the tuple $(q_n,(\sigma_{n,0},\ldots,\sigma_{n,q_n}))$ in $\mathcal{P}$ is the $n$\textbf{-th cutting and spacing parameter} of $\bm{p}$, and the tuple $(q_k,(\sigma_{k,0},\ldots,\sigma_{k,q_k}))_{0\leq k\leq n}$ is the \textbf{projection} of $\bm{p}$ on $\mathcal{P}^{n+1}$ (it gives the first $n+1$ cutting and spacing parameters). From $\bm{p}$, we also define three sequences:
		\begin{itemize}
			\item $(h_n)_{n\geq 0}$ the \textbf{height sequence} of $\bm{p}$, inductively defined by
			$\left\{\begin{array}{l}
				h_0=1,\\
				h_{n+1}=q_nh_n+\sigma_{n}
			\end{array}\right.$,
			$h_n$ is called the \textbf{height of the }$n$\textbf{-th tower};
			\item $(\sigma_n)_{n\geq 0}$, with $\sigma_n\coloneq \sum_{i=0}^{q_n}{\sigma_{n,i}}$ (the number of new spacers at step $n$);
			\item $(Z_n)_{n\geq 0}$, with $Z_n\coloneq \max{\{\sigma_{j,i}\mid 0\leq j\leq n, 0\leq i\leq q_{j}\}}$,
		\end{itemize}
		and it is also possible to consider the finite sequences $(h_k)_{0\leq k\leq n+1}$, $(\sigma_k)_{0\leq k\leq n}$ and $(Z_k)_{0\leq k\leq n}$ associated to a finite sequence of cutting and spacing parameters in $\mathcal{P}^{n+1}$.
	\end{definition}
	
	The terminology "cutting", "spacing", "tower", "height", etc, is justified by Definition~\ref{defr1} and Figure~\ref{figrg1}. There are many definitions of rank-one systems (see \cite{ferencziSystemsFiniteRank1997} for a complete survey and various facts in this section). In this paper the goal is to use the combinatorial structure given by the cutting-and-stacking method (see Figure~\ref{figrg1}).
	
	\begin{definition}\label{defr1}
		A transformation $T\in\aut$ is of \textbf{rank one} if there exist
		\begin{enumerate}
			\item a sequence of cutting and spacing parameters $\bm{p}=(q_n,(\sigma_{n,0},\ldots,\sigma_{n,q_n}))_{n\geq 0}\in\mathcal{P}^{\N}$ satisfying
			\begin{equation}\tag{F}\label{finite}
				\displaystyle\sum_{n=0}^{+\infty}{\frac{\sigma_n}{h_{n+1}}}<+\infty,
			\end{equation}
			where $(h_n)$ and $(\sigma_n)$ are the sequences associated to $\bm{p}$, as described in Definition~\ref{defparam};
			\item measurable subsets of $X$, denoted by $B_n$ (for every $n\geq 0$), $B_{n,i}$ (for every $n\geq 0$ and $0\leq i\leq q_n-1$), and $\Sigma_{n,i,j}$ (for every $n\geq 0$, $0\leq i\leq q_n$ and $1\leq j\leq \sigma_{n,i}$; if $\sigma_{n,i}=0$, then there are no $\Sigma_{n,i,j}$) such that for all $n\geq 0$
			\begin{enumerate}
				\item $B_n,\ldots ,T^{h_n-1}(B_n)$ are pairwise disjoint;
				\item $(B_{n,0}, B_{n,1}, \ldots, B_{n,q_n-1})$ is a partition of $B_n$;
				\item $T^{h_n}(B_{n,i})=
				\left\{\begin{array}{ll}
					\Sigma_{n,i+1,1}&\text{ if }\sigma_{n,i}>0\\
					B_{n,i+1}&\text{  if }\sigma_{n,i}=0\text{  and }i<q_n-1
				\end{array}\right.$;
				\item if $\sigma_{n,i}>0$, then $T(\Sigma_{n,i,j})=
				\left\{\begin{array}{ll}
					\Sigma_{n,i,j+1}&\text{ if }j<\sigma_{n,i}\\
					B_{n,i}&\text{ if }j=\sigma_{n,i}\text{ and }i\leq q_n-1
				\end{array}\right.$;
				\item $B_{n+1}=
				\left\{\begin{array}{ll}
					\Sigma_{n,0,1}&\text{  if }\sigma_{n,0}>0\\
					B_{n,1}&\text{  if }\sigma_{n,0}=0
				\end{array}\right.$;
			\end{enumerate}
		\end{enumerate}
		and if the Rokhlin towers $\mathcal{R}_n\coloneq (T^k(B_n))_{0\leq k\leq h_n-1}$
		are increasing to the $\sigma$-algebra $\A$, meaning that the $\sigma$-algebra generated by $\{T^k(B_n)\mid n\in\N,\ 0\leq k\leq h_n-1\}$ is $\A$ up to null sets (since $\A$ is standard, this also means that $\{T^k(B_n)\mid n\in\N,\ 0\leq k\leq h_n-1\}$ separates the points). Note that $\mathcal{R}_0$ is the tower with only one level $B_0$. The sets $\Sigma_{n,i,j}$ are called the \textbf{spacers}. In this paper we will usually write
		\begin{itemize}
			\item $X_n\coloneq B_n\sqcup\ldots\sqcup T^{h_n-1}(B_n)$ the subset covered by the $n$-th tower $\mathcal{R}_n$;
			\item $\varepsilon_n\coloneq \mu((X_n)^c)$ where $(X_n)^c$ denotes the complement of the subset $X_n$ of $X$.
		\end{itemize}
	\end{definition}
	
	\begin{figure}[ht]
		\centering
		\includegraphics[width=1\linewidth]{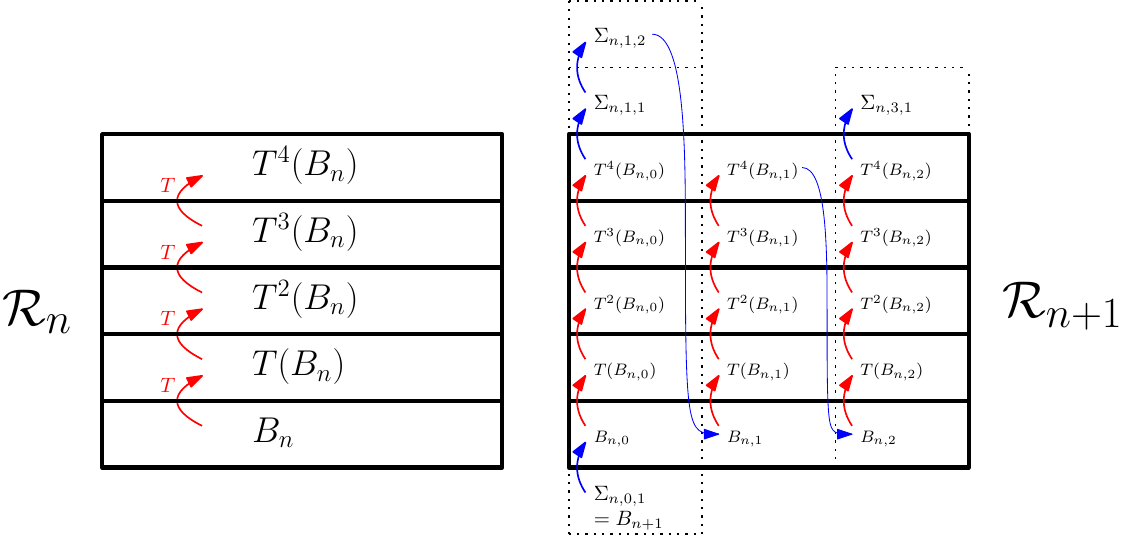}
		\caption{An example of cutting-and-stacking construction with $h_n=5$, $q_n=3$, $\sigma_{n,0}=1$, $\sigma_{n,1}=2$, $\sigma_{n,2}=0$, $\sigma_{n,3}=1$. We then have $h_{n+1}=19$.}
		\label{figrg1}
	\end{figure}
	
	Since $X_n$ is increasing and $\mathcal{R}_n$ increases to the atomless $\sigma$-algebra $\A$, we have \mbox{$\mu(X_n)\underset{n\to +\infty}{\to}1$}. In other words $\varepsilon_n$ tends to $0$.\par
	Before giving examples, the following lemmas give some easy properties on these systems in order to understand their combinatorial structure and the hypotheses required in the definition.
	
	\begin{lemma}\label{etagerg1}
		Let $(h_n)$ and $(\sigma_n)$ be the sequences associated to $(q_n,(\sigma_{n,0},\ldots,\sigma_{n,q_n}))_n\in\mathcal{P}^{\N}$ (see Definition~\ref{defparam}). The following assertions are equivalent:
		\begin{enumerate}
			\item the series $\displaystyle \sum{\frac{\sigma_n}{h_{n+1}}}$ converges (condition~\eqref{finite} in Definition~\ref{defr1});
			\item the series $\displaystyle \sum{\frac{\sigma_n}{q_0\ldots q_n}}$ converges;
			\item there exists a constant $M_0\leq 1$ such that $\displaystyle h_{n+1}\underset{n\to +\infty}{\sim}\frac{q_0\ldots q_n}{M_0}$,
		\end{enumerate}
		and if one of these equivalent assertions is true, then $\sum_{n\geq 0}{\frac{\sigma_n}{q_0\ldots q_n}}=\frac{1}{M_0}-1$.
	\end{lemma}
	
	\begin{proof}[Proof of Lemma~\ref{etagerg1}]
		If the series $\sum{\frac{\sigma_n}{q_0\ldots q_n}}$ converges, so does the series $\sum{\frac{\sigma_n}{h_{n+1}}}$ since $h_{n+1}$ is greater or equal to $q_0\ldots q_n$. Now assume that the series $\sum{\frac{\sigma_n}{h_{n+1}}}$ converges. Notice that we have
		$$\frac{\sigma_n}{h_{n+1}}=\frac{h_{n+1}-q_nh_{n}}{h_{n+1}}=1-q_n\frac{h_n}{h_{n+1}}$$
		and since the series is convergent, the product $\prod{q_n\frac{h_n}{h_{n+1}}}$ converges to some $M_0>0$, i.e.~$q_0\ldots q_n/h_{n+1}\to M_0$. The constant $M_0$ is less than or equal to $1$ since we have $h_{n+1}\geq q_nh_n$ for every $n\geq 0$. Finally let us assume $q_0\ldots q_n/h_{n+1}\to M_0$. Notice that we have
		$$\frac{\sigma_n}{q_0\ldots q_n}=\frac{h_{n+1}-q_nh_n}{q_0\ldots q_n}=\frac{h_{n+1}}{q_0\ldots q_n}-\frac{h_n}{q_0\ldots q_{n-1}},$$
		so by telescoping consecutive terms, we get $\sum_{n\geq 0}{\frac{\sigma_n}{q_0\ldots q_n}}=\lim_{n\to\infty}{\frac{h_{n+1}}{q_0\ldots q_n}}-h_0=\frac{1}{M_0}-1$ and we are done for the equivalence between the three assumptions.
	\end{proof}
	
	\begin{lemma}\label{etagerg1bis}
		Let $T\colon X\to X$ be a bimeasurable bijection. Assume that $T$ preserves a non-zero measure $\mu$ and it admits a sequence of Rokhlin towers as in Definition~\ref{defr1}. The following hold:
		\begin{enumerate}
			\item the levels $T^k(B_n)$ of the $n$-th Rokhlin tower $\mathcal{R}_n$ have $\mu$-measure $\displaystyle \frac{\mu(B_0)}{q_0\ldots q_{n-1}}$;
			\item $\mu$ is finite if and only if the condition~\eqref{finite} is satisfied. Furthermore, if $\mu$ is a probability measure (this implies that $T$ is a rank-one system), then $\mu(B_0)=M_0$ and $\displaystyle h_{n+1}\leq\frac{q_0\ldots q_n}{M_0}$, where $M_0$ is given by Lemma~\ref{etagerg1}.
		\end{enumerate}
	\end{lemma}
	
	\begin{proof}[Proof of Lemma~\ref{etagerg1bis}]
		For a fixed $n$, the levels of $\mathcal{R}_n$ have the same measure by $T$-invariance of the measure $\mu$. Moreover the first level $B_n$ is a disjoint union of $q_n$ levels $B_{n,0}, \ldots, B_{n,q_n-1}$ of $\mathcal{R}_{n+1}$. Then it is clear by induction that a level of $\mathcal{R}_n$ has measure $\frac{\mu(B_0)}{q_0\ldots q_{n-1}}$. Since the sequence $(X_n)_{n\geq 0}$ is increasing to $X$, and $X_{n+1}$ is obtained from $X_n$ by adding $\sigma_n$ spacers, which are levels of $\mathcal{R}_{n+1}$, we get
		\begin{equation}\label{finitemeasure}
			\mu(X)=\mu(X_0)+\sum_{n\geq 0}{\mu(X_{n+1}\setminus X_n)}=\mu(B_0)+\sum_{n\geq 0}{\frac{\mu(B_0)\sigma_n}{q_0\ldots q_n}},
		\end{equation}
		so $\mu(B_0)$ is non-zero, and $\mu(X)$ is finite if and only if the sum $\sum_{n\geq 0}{\frac{\sigma_n}{q_0\ldots q_n}}$ is finite. Finally, let us assume that $\mu$ is a probability measure. This implies $\sum_{n\geq 0}{\frac{\sigma_n}{q_0\ldots q_n}}=\frac{1}{M_0}-1$ and, using \eqref{finitemeasure}, we get $\mu(B_0)=M_0$. The measurable set $X_n$ is the disjoint union of $h_n$ levels of $\mathcal{R}_n$, so the inequality $h_n\leq \frac{q_0\ldots q_{n-1}}{M_1}$ follows from the fact that $\mu$ is a probability measure.
	\end{proof}
	
	It is possible to build a finite measure-preserving transformation $T$ of rank one with a given combinatorial setting $(q_n,(\sigma_{n,0},\ldots,\sigma_{n,q_n}))_{n\geq 0}\in\mathcal{P}^{\N}$ satisfying the hypothesis~\eqref{finite}. For instance it suffices to build $(X_n)$ as an increasing sequence of intervals of $\R_+$, with $B_{n,i}$ and $\Sigma_{n,i,j}$ being subintervals of equal length and disjoint (for a fixed $n$), each on which $T$ is defined as an affine map, and with $B_{0}=[0,M_0]$. The convergence of the series $\sum{\frac{\sigma_n}{h_{n+1}}}$ and Lemma~\ref{etagerg1} ensure that $X\coloneq \bigcup{X_n}$ is equal to $[0,1]$ (up to a null set), so the Lebesgue measure on $[0,1]$ is a probability measure preserved by $T$. Notice that if the series is divergent, we can set $B_0=[0,1]$ and this defines $T$ on the set of positive real numbers endowed with the Lebesgue measure, so this is an infinite measure-preserving transformation.\par
	Therefore for every $(q_n,(\sigma_{n,0},\ldots,\sigma_{n,q_n}))_{n\geq 0}\in\mathcal{P}^{\N}$ satisfying the condition~\eqref{finite}, there exists a rank-one system having a cutting-and-stacking construction with these cutting and spacing parameters, this fact will be used in this paper since it enables us to only take into account the combinatorics behind the systems.\par
	The hypothesis on the Rokhlin towers $\mathcal{R}_n$ aims not only to have $\varepsilon_n\to 0$ but also to define two isomorphic systems when they admit cutting-and-stacking constructions with the same cutting and spacing parameters. Moreover if $T$ admits such a construction with Rokhlin towers increasing to a sub-$\sigma$-algebra $\mathcal{B}$ of $\A$, then $T$, seen as an element of $\autalg$, is not necessarily a rank-one system but admits a rank-one system ($T$ on the sub-$\sigma$-algebra $\mathcal{B}$) as a factor.\par
	Two different families of cutting and spacing parameters do not necessarily define non-isomorphic systems. Indeed in a construction of a rank-one system with parameters $q_n$ and $\sigma_{n,i}$, one can decide to only consider a subsequence $\mathcal{R}_{n_k}$ of Rokhlin towers. For example, the new cutting parameters will be $q_{n_k}q_{n_k+1}\ldots q_{n_{k+1}-1}$ for $k\geq 0$.
	\newline
	
	The rank-one systems form a class of ergodic and zero entropy systems. The easiest examples of rank-one systems are the \textbf{irrational rotations}
	$$R_{\theta}\colon z\in\T\mapsto e^{2i\pi\theta}z\in\T$$
	for every irrational numbers $\theta$, where $\T$ is the unit circle endowed with its Haar measure. These systems are not weakly mixing. Moreover they have discrete spectrum and the point spectrum of $R_{\theta}$ is $\{e^{in\theta}\mid n\in\Z\}$, so by the Halmos-von Neumann Theorem \cite{halmosOperatorMethodsClassical1942}, $R_{\theta}$ and $R_{\theta'}$ are isomorphic if and only if $\theta=\theta'\text{ mod }\Z$ or $\theta=-\theta'\text{ mod }\Z$.\par
	The \textbf{odometers} are rank-one. These are exactly the rank-one systems without spacers (i.e.~$\sigma_{n,i}=0$), so the Rokhlin towers are partitions of the space. Such a system is isomorphic to the adding machine $S$ in the space $\prod_{n\geq 0}{\{0,1,\ldots ,q_n-1\}}$, namely the addition by $(1,0,0,0,\ldots)$ with carry over to the right, defined for every $x\in\prod_{n\geq 0}{\{0,1,\ldots ,q_n-1\}}$ by
	$$Sx=\left\{\begin{array}{ll}
		(0,\ldots,0,x_i+1,x_{i+1},\ldots)&\text{if }i\coloneq \min{\{j\geq 0\mid x_j\not=q_j-1\}}\text{ is finite}\\
		(0,0,0,\ldots)&\text{if }x=(q_0-1,q_1-1,q_2-1,\ldots)
	\end{array}\right.$$
	and it preserves the product of uniform probability measures on each finite set $\{0,1,\ldots,q_n-1\}$. Denote the cylinders of length $k$ by
	$$[x_0,\ldots ,x_{k-1}]_k\coloneq \left \{y\in\prod_{n\geq 0}{\{0,1,\ldots ,q_n-1\}}\mid y_0=x_0,\ldots ,y_{k-1}=x_{k-1}\right \}.$$
	If $S$ is the odometer on the space $\prod_{n\geq 0}{\{0,1,\ldots ,q_n-1\}}$, we can also set a partially defined map
	$$\sn_n\colon X\setminus [\bullet,\ldots,\bullet, q_{n-1}-1]_{n}\to X\setminus [\bullet,\ldots,\bullet,0]_{n}$$
	(the symbol $\bullet$ means that there is no requirement on the value at some coordinate) which is the addition by
	$$(\underbrace{0,\ldots ,0}_{n-1\text{ times}},1,0,0,\ldots)$$
	(so $S$ and $\sn_1$ coincide on $X\setminus [q_0-1]_1$). Then we have
	$$B_n=[\underbrace{0,\ldots ,0}_{n\text{ times}}]_{n},$$
	$$B_{n,i}=[\underbrace{0,\ldots ,0}_{n\text{ times}},i]_{n+1}$$
	and $B_{n,i}=\sn_{n+1}^i(B_{n,0})$ for every $0\leq i\leq q_n-1$, so it provides a scale in $B_n$. Note that it is possible to recover the odometer $S$ from these partially defined maps $\sn_n$ (see Figure~\ref{odometer}). In Section~\ref{theconstruction}, the strategy will be to build $S$ from partially defined maps $\sn_n$.
	
	\begin{figure}[ht]
		\centering
		\includegraphics[width=1\linewidth]{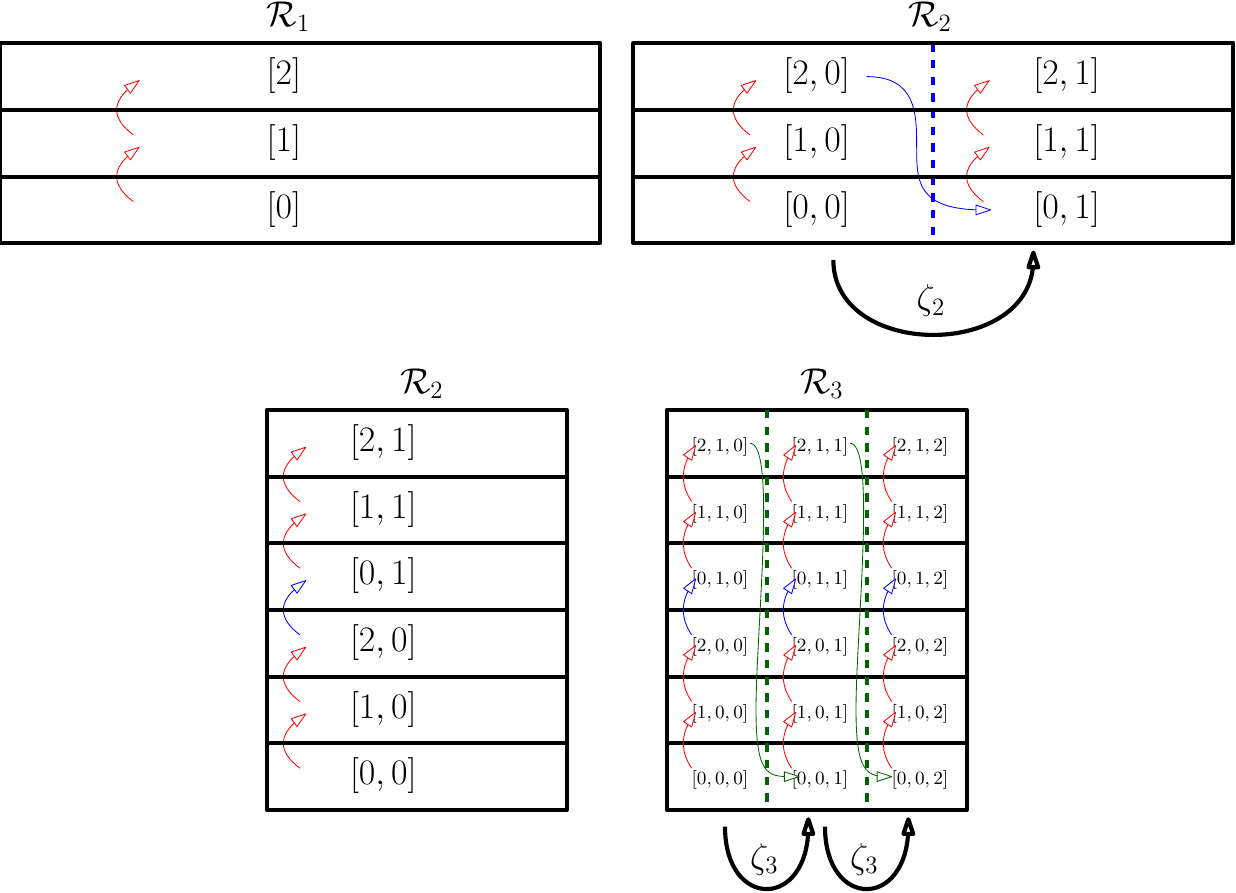}
		\caption{Example of odometer with $q_0=3$, $q_1=2$, $q_2=3$.}
		\label{odometer}
	\end{figure}
	
	In the class of odometers, the number of occurrences of every prime factors in the set $\{q_n\mid n\geq 0\}$ form a total invariant of conjugacy. As for irrational rotations, it is a consequence of the Halmos-von Neumann Theorem since odometers have discrete spectrum and their eigenvalues are given by these occurrences. In particular odometers have eigenvalues non-equal to $1$ and are not weakly mixing, moreover odometers and irrational rotations are not isomorphic. Notice that the Halmos-von Neumann Theorem implies that the conjugacy classes among ergodic systems with discrete spectrum coincide with the flip-conjugacy classes since the point spectrum of a system is a subgroup of $\T$. If every prime number has infinite multiplicity in the set $\{q_n\mid n\geq 0\}$, then the odometer is said to be \textbf{universal}. An odometer is \textbf{dyadic} if $2$ is the only prime factor.\par
	\textbf{Chacon's map} is the first example of weakly mixing system which is not strongly mixing \cite{chaconWeaklyMixingTransformations1969a} and was the starting point for the notion of rank-one systems. It is a rank-one transformation defined with cutting and spacing parameters $q_n=3$, $\sigma_{n,0}=\sigma_{n,1}=\sigma_{n,3}=0$, $\sigma_{n,2}=1$.
	
	\subsection{Flexible classes}
	
	Now we introduce classes of rank-one systems to which the main result of this paper applies. First let us consider cutting-and-stacking constructions whose spacing parameters have controlled asymptotics. Recall that $\mathcal{P}^{\N}$ is the set of sequences of cutting and spacing parameters. As introduced in Definition~\ref{defparam}, $(h_n)$, $(\sigma_n)$ and $(Z_n)$ denotes the sequences associated to a sequence in $\mathcal{P}^{\N}$: $h_n$ is the height of the $n$-th tower, $\sigma_n$ the number of new spacers at step $n$ and $Z_n$ is the maximum number of spacers between two consecutive towers, over the first $n$ steps.
	
	\begin{definition}\label{cspbsp}
		A construction by cutting and stacking with cutting and spacing parameters $(q_n, (\sigma_{n,0},\ldots,\sigma_{n,q_n}))_{n\geq 0}\in\mathcal{P}^{\N}$ is said \textbf{CSP} ("\textbf{controlled-spacing-parameter}") if there exists a constant $C>0$ such that $Z_n\leq Ch_{n}$ for all $n$. It is furthermore \textbf{BSP} ("\textbf{bounded-spacing-parameter}") if $Z_n\leq C$ and $\sigma_{n,0}=\sigma_{n,q_n}=0$ for all $n$. A rank-one system $T$ is \textbf{BSP} if it admits a BSP cutting-and-stacking construction.
	\end{definition}
	
	Odometers and Chacon's map are examples of BSP rank-one systems. Moreover BSP implies CSP. The interest in the BSP property is its stability after skipping steps in the cutting-and-stacking process, as stated in the following lemma.
	
	\begin{figure}[ht]
		\centering
		\includegraphics[width=0.9\linewidth]{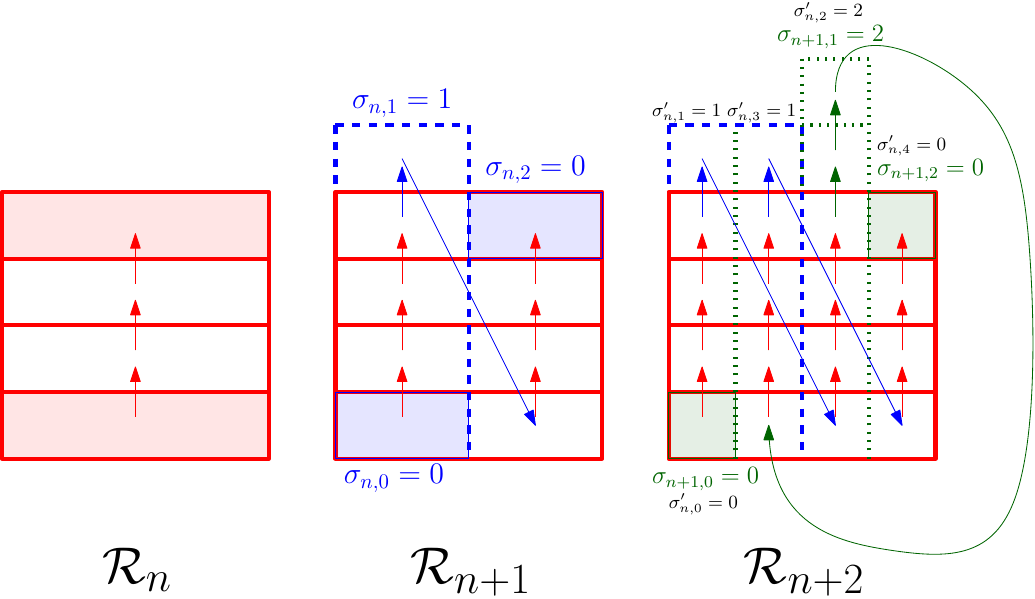}
		\caption{Illustration of the proof of Lemma~\ref{skip}, spacing parameters from $\mathcal{R}_n$ to $\mathcal{R}_{n+2}$ with $q_n=q_{n+1}=2$ (the coloured levels are the base and the roof of the towers).}
		\label{figproofskip}
	\end{figure}
	
	\begin{lemma}\label{skip}
		Given a BSP cutting-and-stacking construction, any subsequence of its Rokhlin towers still provides a BSP construction with the same constant $C$.
	\end{lemma}
	
	\begin{proof}[Proof of Lemma~\ref{skip}]
		Let $q_n$ and $\sigma_{n,i}$ be the cutting and spacing parameters of the BSP construction, $C$ the bound for the spacing parameters $\sigma_{n,i}$, $\mathcal{R}_n$ the associated towers and $\mathcal{R}_{n_k}$ a subsequence. Let $k$ be an integer and assume $n_{k+1}=n_k+2$. Denote by $q'_{n_k}$ and $\sigma'_{n_k,i}$ the new cutting and spacing parameters from $\mathcal{R}_{n_k}$ to $\mathcal{R}_{n_{k+1}}$. It is easy to show that $q'_{n_k}=q_{n_k}q_{n_k+1}$, \mbox{$\sigma'_{n_k,0}=\sigma'_{n_k,q'_{n_k}}=0$} and for every $1\leq j\leq q_{n_k+1}$, $\sigma'_{n_k,(j-1)q_{n_k}+i}$ is equal to $\sigma_{n_k,i}$ if $1\leq i\leq q_{n_k}-1$, $\sigma_{n_k+1,j}$ if $i=q_{n_k}$ (see Figure~\ref{figproofskip}). Thus the non-zero spacing parameters from $\mathcal{R}_{n_k}$ to $\mathcal{R}_{n_{k+1}}$ are of the form $\sigma_{n_k,i}$ or $\sigma_{n_k+1,i}$ and they are all bounded above by $C$. For $n_{k+1}$ bigger than $n_k+2$, the result is now clear by induction.
	\end{proof}
	
	If the parameters $\sigma_{n,q_n}$ are non-zero, then skipping steps in the cutting-and-stacking process will cause an accumulation of spacers above the last columns and the new spacing parameters will not be bounded if the subsequence is properly chosen so that the jumps $n_{k+1}-n_k$ increase quickly enough. We have the same problem for $\sigma_{n,0}$ (accumulation of spacers at the bottom of the first columns), hence the conditions $\sigma_{n,0}=\sigma_{n,q_n}=0$ in the definition of BSP.\par
	Lemma~\ref{skip} has no reason to hold for CSP construction that are not BSP. Indeed the spacing parameters from $\mathcal{R}_{n_k}$ to $\mathcal{R}_{n_{k+1}}$ have to be compared with $h_{n_k}$, the height of $\mathcal{R}_{n_k}$. The comparison is easily obtained for the spacing parameters $\sigma_{n_k,i}$, $0\leq i\leq q_{n_k}$, but for the other spacing parameters, we only know that they are bounded above by $Ch_{n_k+1}, Ch_{n_k+2}, \ldots, Ch_{n_{k+1}-1}$.
	\newline
	
	In the sequel we will see other important CSP examples by considering classes containing "nice" cutting-and-stacking constructions, meaning that we will be able to properly choose the parameters in order to have the desired quantification of the cocycles for the orbit equivalence built in Section~\ref{theconstruction}. By definition, every flexible class $\mathcal{C}$ will be associated to some subset $\fc$ of $\mathcal{P}^{*}$, which can be considered as sufficient conditions that the cutting and spacing parameters have to satisfy at each step for the underlying transformation to belong to $\mathcal{C}$. Recall that $\mathcal P^*$ denotes the set of all finite sequences of cutting and stacking parameters.
	
	\begin{definition}\label{deffc}
		A class $\mathcal{C}$ of rank-one systems is said to be \textbf{flexible} if there exists a subset $\fc$ of $\mathcal{P}^{*}$ satisfying the following properties:
		\begin{enumerate}
			\item there exists a constant $C>0$ such that for all $\left (q_n,(\sigma_{n,0},\ldots ,\sigma_{n,q_n})\right )_{n\geq 0}\in\mathcal{P}^{\N}$ satisfying the condition~\eqref{finite} in Definition~\ref{defr1}, if $\fc$ contains every projection $\left (q_k,(\sigma_{k,0},\ldots ,\sigma_{k,q_k})\right )_{0\leq k\leq n}\in\mathcal{P}^{n+1}$ for $n\geq 0$, then these parameters define a CSP construction (with the constant $C$) and the underlying rank-one transformation is in $\mathcal{C}$;
			\item there exists a cutting and spacing parameter $(q_0,(\sigma_{0,0}\ldots,\sigma_{0,q_0}))$ in $\fc$ with \mbox{$q_0\geq 3$};
			\item there is a constant $C'>0$ such that for all $n\geq 1$, if $\left (q_k,(\sigma_{k,0},\ldots ,\sigma_{k,q_k})\right )_{0\leq k\leq n-1}$ is in $\fc$, then there are infinitely many integers $q_n$ such that $\left (q_k,(\sigma_{k,0},\ldots ,\sigma_{k,q_k})\right )_{0\leq k\leq n}$ is in $\fc$ for some $\sigma_{n,0},\ldots ,\sigma_{n,q_n}$ satisfying the inequality
			$$\sigma_n\leq C'q_nh_{n-1}$$
			where $(h_{k})_{0\leq k\leq n+1}$ and $(\sigma_k)_{0\leq k\leq n}$ denote the finite sequences associated to the finite sequence $\left (q_k,(\sigma_{k,0},\ldots ,\sigma_{k,q_k})\right )_{0\leq k\leq n}$ of cutting and stacking parameters.
		\end{enumerate}
		A rank-one system $T$ is \textbf{flexible} if $\{T\}$ is a flexible class.
	\end{definition}
	
	The third point of the definition aims to recursively choose the cutting parameters (and we want them to increase quickly enough) with an asymptotic control on $(\sigma_n)_n$, while the first point guarantees that it is possible to do so for the underlying system to be in the class $\mathcal{C}$. The second point is minor, but it is required for the initialization of the recursive construction of an odometer orbit equivalent to an element of our flexible class (see Lemma~\ref{welldefined} and Remark~\ref{q0}). It also ensures that $\fc$ is not an empty set.\par
	Notice that if a construction satisfies $Z_n\leq Ch_{n-1}$ for all $n$, then it is in particular CSP and we get $\sigma_n\leq C(q_n+1)h_{n-1}\leq 2Cq_nh_{n-1}$ as in the third point of Definition~\ref{deffc}.\par
	We now give examples of flexible classes. The proof is given in Section~\ref{sectionproofprop}.
	
	\begin{proposition}\label{fc}
		\begin{enumerate}
			\item Every BSP rank-one system is flexible.
			\item For every nonempty open subset $\mathcal{V}$ of $\R$, the set $\{R_{\theta}\mid\theta\in\mathcal{V}\cap (\R\setminus\mathbb{Q})\}$ is a flexible class.
			\item For every irrational number $\theta$, the class of rank-one systems which have $e^{2i\pi\theta}$ as an eigenvalue is flexible.
			\item The class of strongly mixing rank-one systems is flexible.
		\end{enumerate}
	\end{proposition}
	
	Theorems~\ref{thetheorem},~\ref{thetheoremB},~\ref{thetheoremC} and~\ref{thetheoremD} follow from  Proposition~\ref{fc} and the following theorem which is the main result.
	
	\begin{theorem}\label{thfc}
		Let $\varphi\colon\R_+\to \R_+$ be a map satisfying $\varphi(t)\underset{t\to +\infty}{=}o\left (t^{1/3}\right )$. If $\mathcal{C}$ is a flexible class, then there exists $T$ in $\mathcal{C}$ which is $\varphi$-integrably orbit equivalent to the universal odometer.
	\end{theorem}
	
	\section{Proof of Proposition~\ref{fc}}\label{sectionproofprop}
	
	In this section we prove the four statements in Proposition~\ref{fc}.
	
	\subsection{BSP systems}\label{proofBSP}
	
	Let $T$ be a BSP rank-one system, $\mathcal{C}\coloneq \{T\}$ and $q_n$, $\sigma_{n,i}$ the parameters of a BSP construction of $T$, with a constant $C>0$. For every $n\geq 0$ and $j\geq 1$, let $\sigma^{(n,n+j)}_{0},\ldots ,\sigma^{(n,n+j)}_{q_{n}\ldots q_{n+j-1}}$ be the spacing parameters from $\mathcal{R}_{n}$ to $\mathcal{R}_{n+j}$, assuming that the steps for $\mathcal{R}_{n+1},\ldots ,\mathcal{R}_{n+j-1}$ are skipped during the construction (we then have $\sigma^{(n,n+1)}_{i}=\sigma_{n,i}$ and also $\sigma^{(n,n+j)}_{i}=0$ for $i$ equal to $0$ and $q_n\ldots q_{n+j-1}$ by Lemma~\ref{skip}). The new cutting parameters are $q^{(n,n+j)}\coloneq q_n\ldots q_{n+j-1}$ and are large enough with huge jumps $j$. Now define
	$$\fc\coloneq \left\{\left (q^{(n_k,n_{k+1})},\left (\sigma^{(n_k,n_{k+1})}_{0},\ldots ,\sigma^{(n_k,n_{k+1})}_{\left (q^{(n_k,n_{k+1})}\right )}\right )\right )_{0\leq k\leq m}\Big |\ m\geq 0,\ 0=n_0<n_1<\ldots <n_{m+1}\right\}.$$
	Using Lemma~\ref{skip}, the new spacing parameters $\sigma^{(n_k,n_{k+1})}_{j}$ are bounded by $C$ and we get
	$$\sum_{1\leq j\leq q^{(n_k,n_{k+1})}}{\sigma^{(n_k,n_{k+1})}_{j}}\leq Cq^{(n_k,n_{k+1})}.$$
	The set of parameters $\fc$ thus witnesses that $\{T\}$ is flexible.
	
	\subsection{Irrational rotations}\label{proofirra}
	
	We now consider a construction from~\cite{drillickNonrigidRankoneInfinite2023}. For every irrational number $\theta$, Drillick, Espinosa-Dominguez, Jones-Baro, Leng, Mandelshtam and Silva give an explicit cutting-and-stacking construction of a transformation $T$ which is the irrational rotation of angle $\theta$ when the construction yields a finite measure-preserving system.
	
	\paragraph{The construction in~\cite{drillickNonrigidRankoneInfinite2023}.} Let $\theta$ be an irrational number and $[q_{-1};q_0,q_1,\ldots]$ its continued fraction expansion, with $q_{-1}\coloneq \lfloor\theta\rfloor$ and positive integers $q_0,q_1,\ldots$. Let us assume that there is no integer $n$ such that $q_k=1$ for every $k\geq n$. We consider the sequence $(h_n)_{n\geq 0}$ defined by $h_{-1}\coloneq 0$, $h_0\coloneq 1$ and $h_{k+1}\coloneq q_kh_k+h_{k-1}$ for every $k\geq 0$ (the integer $h_k$ is the denominator of the $k$-th convergent of the irrational number $\theta$). Finally, for every $k\geq 0$, we set $\sigma_{k,i}=0$ for every $i\in\{0,1,\ldots,q_k-1\}$, and $\sigma_{k,q_k}=h_{k-1}$. Then, the sequence of cutting and stacking parameters $(q_k,(\sigma_{k,0},\ldots,\sigma_{k,q_k}))_{k\geq 0}$ provides a rank-one system. This system is the irrational rotation of angle $\theta$ if and only if Condition~\eqref{finite} is satisfied, and this last condition holds if and only if the series $\sum_{k\geq 0}{\frac{1}{q_kq_{k+1}}}$ converges (see Theorems~3.1 and~5.1 in~\cite{drillickNonrigidRankoneInfinite2023}).
	
	\begin{remark}
		Equivalently, we can define rank-one systems with cutting parameters potentially equal to $1$, provided that there are infinitely many cutting parameters greater than or equal to $2$, but our construction of orbit equivalence is not well-defined with this weaker assumption. Therefore, in the proof of Proposition~\ref{fc} for irrational rotations, one of the main goals is to avoid cutting parameters equal to~$1$.
	\end{remark}
	
	\begin{remark}
		It is proven in~\cite{drillickNonrigidRankoneInfinite2023} that the set of irrational numbers $\theta$ such that the associated series $\sum_{k\geq 0}{\frac{1}{q_kq_{k+1}}}$ converges has measure zero.
	\end{remark}
	
	\paragraph{Proof of Proposition~\ref{fc} for these systems.}
	
	Let $\mathcal{V}$ be a nonempty open subset of $\R$ and
	$$\mathcal{C}\coloneq \{R_{\theta}\mid\theta\in\mathcal{V}\cap(\R\setminus\mathbb{Q})\}.$$
	We now prove that $\mathcal{C}$ is a flexible class.\par
	We first use the following basic fact from the theory of continued fractions: if $A$ denotes the set of sequences $(q_i)_{i\geq -1}$ of integers such that $q_0,q_1,\ldots$ are positive, and if $A$ is equipped with the induced product topology, then the map
	$$(q_i)_{i\geq -1}\in A\mapsto [q_{-1};q_0,q_1,\ldots]\in\R\setminus\mathbb{Q},$$
	is a homeomorphism (see~\cite[Lemma 3.4]{einsiedlerErgodicTheoryView2011} for instance). We can then fix integers $n_0\geq 0$ and $Q_{-1},Q_{0},\ldots,Q_{n_0}$ (where $Q_{0},\ldots,Q_{n_0}$ are positive) such that $Q_{0}\ldots Q_{n_0}$ is greater than or equal to $3$ and the following holds: for every irrational number $\theta$, if the first coefficients of its continued fraction expansion are $Q_{-1},Q_{0},\ldots,Q_{n_0}$, then $\theta$ is in $\mathcal{V}$.\par
	We write $\bm{Q}\coloneq (Q_0,\ldots,Q_{n_0})$ and we consider the set $\tilde{\mathcal{F}}(\bm{Q})$ of finite sequences $(\tilde{q}_k,(\tilde{\sigma}_{k,0},\ldots,\tilde{\sigma}_{k,\tilde{q}_k}))_{0\leq k\leq n}$ such that $n\geq n_0$ and for all $k\in\{0,\ldots,n\}$,
	
	$$\begin{array}{ll}
		\tilde{q}_k=Q_k&\text{if }k\leq n_0,\\
		\tilde{q}_k\geq 2&\text{if }k>n_0,
	\end{array}$$
	
	and
	
	$$\begin{array}{ll}
		\tilde{\sigma}_{k,i}=0&\text{for }i\in\{0,\ldots,\tilde{q}_k-1\},\\
		\tilde{\sigma}_{k,\tilde{q}_k}=\tilde{h}_{k-1}&
	\end{array}$$
	(where $(\tilde{h}_k)_{0\leq k\leq n}$ is the associated height sequence and $\tilde{h}_{-1}\coloneq 0$). The finite sequences of $\tilde{\mathcal{F}}(\bm{Q})$ may not be finite sequences of cutting and stacking parameters in the sense of Definition~\ref{defparam}, since the integers $Q_0,\ldots,Q_{n_0}$ may be equal to $1$. Moreover, even if the integers $Q_0,\ldots,Q_{n_0}$ were greater than or equal to $2$, we could not prove that $\mathcal{C}$ is a flexible class with $\fc=\tilde{\mathcal{F}}(\bm{Q})$, since the first cutting parameters $\tilde{q}_1,\ldots,\tilde{q}_{n_0}$ cannot be chosen large enough. Notice that, although we may have $\tilde{q}_k=1$ for some $k\in\{0,1,\ldots,n_0\}$, a finite sequence $(\tilde{q}_k,(\tilde{\sigma}_{k,0},\ldots,\tilde{\sigma}_{k,q_k}))_{0\leq k\leq n}$ can still define the first $(n+1)$ steps of a cutting-and-stacking construction, and we associate to it another finite sequence $(q_k,(\sigma_{k,0},\ldots,\sigma_{k,q_k}))_{0\leq k\leq n-n_0}$ corresponding to the cutting-and-stacking construction obtained from the previous one by skipping the steps $1, 2,\ldots,n_0$. We get $h_0=\tilde{h}_0=1$, $q_0=Q_0\ldots Q_{n_0}\geq 3$ and for all $\forall k\in\{1,\ldots,n-n_0\}$,
	
	$$\begin{array}{ll}
		q_k=\tilde{q}_{n_0+k}\geq 2,&\\
		h_k=\tilde{h}_{n_0+k},&\\
		\sigma_{k,i}=\tilde{\sigma}_{n_0+k,i}=0&\text{for }i\in\{0,\ldots,q_k-1\},\\
		\sigma_{k,q_k}=\tilde{\sigma}_{n_0+k,\tilde{q}_{n_0+k}}=\tilde{h}_{n_0+k-1}=h_{k-1}&\text{if }k\geq 2.
	\end{array}$$	
	
	For $k=1$, we have $\sigma_{1,q_1}=\tilde{h}_{n_0}$, where $\tilde{h}_{n_0}$ is not equal to $h_0$. Setting $C=C'\coloneq \tilde{h}_{n_0}$ (this constant only depends on $Q_0,\ldots,Q_{n_0}$), we have $Z_1\leq Ch_1$ and $\sigma_{1}\leq C'h_0$. We immediately get the inequalities $Z_k\leq Ch_k$ and $\sigma_{k}\leq C'h_{k-1}$ for $k\in\{2,\ldots,n-n_0\}$.\par
	Let $\mathcal{F}(\bm{Q})$ be the set of finite sequences $(q_k,(\sigma_{k,0},\ldots,\sigma_{k,q_k}))_{0\leq k\leq n-n_0}$ obtained from finite sequences $(\tilde{q}_k,(\tilde{\sigma}_{k,0},\ldots,\tilde{\sigma}_{k,q_k}))_{0\leq k\leq n}\in\tilde{\mathcal{F}}(\bm{Q})$. It is now easy to check that $\mathcal{C}$ is a flexible class, with the set of parameters $\fc\coloneq \mathcal{F}(\bm{Q})$ and the constants $C$ and $C'$.
	
	\subsection{Systems with a given eigenvalue}\label{proofeigenvalue}
	
	Let $\theta$ be an irrational number and $\mathcal{C}$ the class of rank-one systems which has $\lambda\coloneq e^{2i\pi\theta}$ as an eigenvalue. In \cite{danilenkoExplicitRank1Constructions2023}, Danilenko and Vieprik present an explicit cutting-and-stacking construction of a system in $\mathcal{C}$. The parameters are chosen in the following way (see the proof of Theorem~4.1 in \cite{danilenkoExplicitRank1Constructions2023}).
	
	\paragraph{The construction of Danilenko and Vieprik.}
	
	For every $n\geq 1$, we fix a number \mbox{$j_n\in\{1,\ldots ,n\}$} such that $\delta_n\coloneq |1-\lambda^{j_n}|$ is less than $2\pi/n$. Fix $n\geq 1$, assume that $\left (q_k,(\sigma_{k,0},\ldots,\sigma_{k,q_k})\right )_{0\leq k\leq n-1}$ has already been constructed with an auxiliary condition
	\begin{equation}\label{danilenkovieprik2}
		h_{n}>\frac{n^4}{\delta_{n^2}}.
	\end{equation}
	Danilenko and Vieprik show the existence of a sequence $(\ell^{(n)}_m)_{m\geq 1}$ of positive integers less than or equal to $2\pi/\delta_{n^2}$, such that for every $m\geq 1$,
	\begin{equation}\label{danilenkovieprik1}
		|1-\lambda^{mh_{n}+(\ell^{(n)}_1+\ldots +\ell^{(n)}_{m})j_{n^2}}|<\frac{2\pi}{n^2}.
	\end{equation}
	Next, let $q_n$ be an integer large enough so that the auxiliary condition~\eqref{danilenkovieprik2} holds at the next step, namely
	$$h_{n+1}\coloneq q_nh_{n}+(\ell^{(n)}_1+\ldots +\ell^{(n)}_{q_n-1})j_{n^2}>\frac{(n+1)^4}{\delta_{(n+1)^2}}$$
	(in \cite{danilenkoExplicitRank1Constructions2023}, $q_n$ is chosen as the smallest integer satisfying the property but it is not needed, so there are infinitely many choices). Finally the spacing parameters at this step are defined by $\sigma_{n,0}=\sigma_{n,q_n}=0$ and $\sigma_{n,m}=\ell^{(n)}_mj_{n^2}$ for $1\leq m\leq q_n-1$.\par
	With these parameters satisfying \eqref{danilenkovieprik2} and \eqref{danilenkovieprik1}, $\lambda$ is an eigenvalue of the underlying rank-one system (see \cite{danilenkoExplicitRank1Constructions2023}, proof of Theorem~4.1, for details).
	
	\paragraph{Proof of Proposition~\ref{fc} for these systems.}
	Let us consider the same construction as above, but with the following auxiliary condition:
	\begin{equation}\label{danilenkovieprik2prime}
		h_{n}>\max{\left (\frac{n^4}{\delta_{n^2}},\frac{(n+1)^4}{\delta_{(n+1)^2}}\right )},
	\end{equation}
	which is stronger than the previous auxiliary condition~\eqref{danilenkovieprik2}. Note that the real numbers $\delta_i$ have been fixed before setting the parameters.\par
	The subset $\fc$ of $\mathcal{P}^*$ is defined to be the set of finite sequences $\left (q_k,(\sigma_{k,0},\ldots ,\sigma_{k,q_k})\right )_{0\leq k\leq n}$ constructed in a recursive way. Any cutting and spacing parameter $(q_0,(\sigma_{0,0},\ldots,\sigma_{q_0,0}))$ is in $\fc$, and if $\left (q_k,(\sigma_{k,0},\ldots ,\sigma_{k,q_k})\right )_{0\leq k\leq n-1}$ is in $\fc$, then so is $\left (q_k,(\sigma_{k,0},\ldots ,\sigma_{k,q_k})\right )_{0\leq k\leq n}$ for every $(q_n,\sigma_{n,0},\ldots ,\sigma_{n,q_n})$ that we can obtain at the next step, as described above but with the stronger auxiliary condition~\eqref{danilenkovieprik2prime}.\par
	Let $\bm{p}\coloneq (q_n,(\sigma_{n,0},\ldots,\sigma_{n,q_n}))_{n\geq 0}$ be a sequence of cutting and spacing parameters. If all its projections are in $\fc$, then $\bm{p}$ provides a CSP construction with $C=2\pi$. Indeed, we have $\sigma_{n,m}=\ell_{m}^{(n)}j_{n^2}\leq 2\pi n^2/\delta_{n^2}<2\pi h_n$. As mentioned above, Conditions~\eqref{danilenkovieprik2} and~\eqref{danilenkovieprik1} imply that the sequence $\bm{p}$ provides rank-one systems which have $\lambda$ as an eigenvalue.\par
	Finally, if $(q_k,(\sigma_{k,0},\ldots,\sigma_{k,q_k}))_{0\leq k\leq n-1}$ is a finite sequence in $\fc$, we can choose a large enough integer $q_n$ so that the following holds at the next step:
	$$h_{n+1}>\max{\left (\frac{(n+1)^4}{\delta_{(n+1)^2}},\frac{(n+2)^4}{\delta_{(n+2)^2}}\right )}$$
	(in particular, the new auxiliary condition~\eqref{danilenkovieprik2prime} is satisfied). We use the same spacing parameters as before, namely $\sigma_{n,m}=\ell_m^{(n)}j_{n^2}$. Using $j_{n^2}\leq n^2$ and $\ell_m^{(n)}\leq\frac{2\pi}{\delta_{n^2}}$, this gives
	$$\sigma_n=(\ell^{(n)}_1+\ldots +\ell^{(n)}_{q_n-1})j_{n^2}\leq q_n\frac{2\pi n^2}{\delta_{n^2}}\leq q_n h_{n-1},$$
	so the third point of Definition~\ref{deffc} is satisfied for $C'=1$.
	
	\subsection{Strongly mixing systems}\label{proofstrmix}
	
	Let $\mathcal{C}$ be the class of strongly mixing rank-one systems. We consider the construction of Ornstein in \cite{ornsteinRootProblemErgodic1972}. The property the parameters have to satisfy at each step is given by the following lemma (Lemma~3.2 in \cite{ornsteinRootProblemErgodic1972}), proven with a probabilistic method.
	
	\begin{lemma}\label{ornstein}
		Let $N$ and $K$ be positive integers and $\varepsilon>0$, $\alpha>0$. Then there exist integers $m>N$ and $a_1,\ldots ,a_m$ such that
		\begin{itemize}
			\item $\left |\sum_{i=j}^{j+k}{a_i}\right |\leq K$ for all $1\leq j\leq j+k\leq m$;
			\item denoting by $H(\ell,k)$ the number of $j$ such that $\displaystyle\sum_{i=j}^{j+k}{a_i}=\ell$, for $1\leq j\leq j+k\leq m$, then $H(\ell,k)<\alpha \dfrac{(m-k)}{K}$ for every $k<(1-\varepsilon)m$.
		\end{itemize}
	\end{lemma}
	
	The set of parameters $\fc$ is defined in a recursive way, as in Section~\ref{proofeigenvalue}: any cutting and spacing parameter $(q_0,(\sigma_{0,0},\ldots,\sigma_{q_0,0}))$ is in $\fc$, and from a finite sequence of parameters $\left (q_k,(\sigma_{k,0},\ldots ,\sigma_{k,q_k})\right )_{0\leq k\leq n-1}$ in $\fc$, $\left (q_k,(\sigma_{k,0},\ldots ,\sigma_{k,q_k})\right )_{0\leq k\leq n}$ is also in $\fc$ if the new parameters can be written as $q_n=m$ and $\sigma_{n,i}=a_i+h_{n-1}$ where $m,a_1,\ldots, a_m$ are integers whose existence is granted by Lemma~\ref{ornstein} with $N> 10^n$, $K=h_{n-1}$, $\varepsilon=10^{n-3}$ and $\alpha=5/4$. There are infinitely many possibilities for $q_n$ as $N$ can be arbitrarily large. It is shown in \cite{ornsteinRootProblemErgodic1972} that cutting-and-stacking constructions with these parameters give strongly mixing systems, it is clear that they are CSP with $C=2$ and the third point of Definition~\ref{deffc} is satisfied for $C'=2$.
	
	\section{From flexible classes to the universal odometer}\label{sectionconstr}
	
	The goal of this section is to prove Theorem~\ref{thfc}, namely that for every $\varphi\colon\R_+\to\R_+$ satisfying $\varphi(x)=o(t^{1/3})$, any flexible class contains a rank-one system which is $\varphi$-integrably orbit equivalent to the universal odometer.
	
	\subsection{The construction}\label{theconstruction}
	
	\paragraph{Overview of the construction.}
	
	We first present a natural adaptation to the case of rank-one system of Kerr and Li's construction of an explicit orbit equivalence between the universal odometer and any other odometers. We will then see that the quantification of the cocycles becomes more complicated due to the presence of non-zero spacing parameters.\par
	Let $T\in\aut$ be a rank-one system and consider a cutting-and-stacking construction of this transformation with the same notations $q_n,\sigma_{n,i},\sigma_n,h_n,\mathcal{R}_n,\varepsilon_n,X_n$ as in Definition~\ref{defr1}. From the sequence of Rokhlin towers $\mathcal{R}_n$, new towers $\mathcal{R}'_n$ will be built as Rokhlin towers for a new system $S$. These towers $\mathcal{R}'_n$ will have no spacers, i.e. $\sigma'_{n,i}=0$, so they will be partitions of $X$. The construction will ensure that $\mathcal{R}'_n$ increases to the $\sigma$-algebra $\A$ using the fact that it is the case for $\mathcal{R}_n$, so $S$ will be an odometer. For the odometer $S$ to be universal, we fix a sequence of prime numbers $(p_n)_{n\geq 0}$ such that every prime number appears infinitely many times, and every cutting parameter $q'_n$ will be a multiple of $p_n$.\par
	We will recursively define $S$ on subsets increasing to $X$ up to a null set. More precisely if the $n$-th tower $\mathcal{R}'_n$ has been built and its base and its height are denoted by $B'_n$ and $h'_n$, then $S$ is provisionally defined on all the levels of the tower except the highest one and maps the $i$-th level to the $(i+1)$-th one. So $\mathcal{R}'_n$ is exactly $(B'_n,S(B'_n),\ldots ,S^{h'_n-1}(B'_n))$ and $S$ is defined on $X\setminus S^{h'_n-1}(B'_{n})$. To refine $S$, i.e.~to define it on a greater set, we have to build the next tower $\mathcal{R}'_{n+1}$ and define $S$ as for $\mathcal{R}'_n$. In order to do so and according to Definition~\ref{defr1}, we have to determine a subdivision of the base $B'_n$ into $q'_n$ subsets $B'_{n,0},\ldots,B'_{n,q'_n-1}$. We find a function $\sn_{n+1}$ mapping bimeasurably each $B'_{n,i}$ to $B'_{n,i+1}$ for $0\leq i\leq q'_n-2$. On the subset $D_{n+1}\coloneq \bigsqcup_{0\leq i\leq q'_n-2}{S^{h'_n-1}(B'_{n,i})}$ of the roof $S^{h'_n-1}(B_n)$ of $\mathcal{R}'_n$, $S$ will coincide with $\sn_{n+1}S^{-h'_n}$ and will be defined on $X\setminus S^{h'_{n+1}}(B'_{n+1})=D_1\sqcup\ldots\sqcup D_{n+1}$ where $B'_{n+1}=B'_{n,0}$ is the base of the new Rokhlin tower $\mathcal{R}_{n+1}$ for $S$. To sum up, $S$ is successively defined by the finite approximations obtained from the maps $\zeta_n$. Up to conjugacy, $\sn_n$ is exactly the addition by $(\underbrace{0,\ldots ,0}_{n-1\text{ times}},1,0,0,\ldots)$ with carry over to the right (as defined in Section~\ref{cutsta}), restricted to $[0,\ldots,0]_{n-1}\setminus [0,\ldots,0,q'_{n-1}-1]_n$.\par
	The construction of the maps $\zeta_n$ is by induction on $n\geq 0$. At step $n$ we will actually define 
	\[\sn_{n+1}\colon B'_{n,0}\sqcup\ldots\sqcup B'_{n,q'_n-2}\to B'_{n,1}\sqcup\ldots\sqcup B'_{n,q'_n-1}.\]
	In order to build $\zeta_{n+1}$, a second induction on a parameter $m\geq n$ is required. Actually, $B'_{n,i}$ will be the disjoint union of the $B'_{n,i}(m)$ for $m\geq n$, and this inner recursion consists in choosing $m$-bricks to define $B'_{n,i}(m)$. By definition, the $m$\textbf{-bricks} will be the $m$\textbf{-levels} (i.e.~the levels of $\mathcal{R}_m$) explicitly chosen to constitute $B'_{n,i}(m)$. Using powers of $T$, $m$-bricks of $B'_{n,i}$ are mapped to the ones of $B'_{n,i+1}$ (there will be as many in $B_{n,i}$ as in $B'_{n,i+1}$) and this gives $\sn_{n+1}$ whose orbits are included in those of $T$, implying immediately that the orbits of $S$ satisfy the same property. The reverse inclusion between the orbits will be more difficult to prove and will be due to the choice of the bricks (see Remark~\ref{capture} after the construction).
	
	\paragraph{The construction.}
	
	$T$ is a rank-one system in $\aut$. We fix one of its cutting-and-stacking construction whose parameters are denoted as in Definition~\ref{defr1}. Let $(p_n)_{n\geq 0}$ be a sequence of prime numbers such that every prime number appears infinitely many times.\par
	In the sequel, we will assume that, given the cutting parameters of $T$, some positive integers $q'_n$ and $t_{n,m}$ that we will introduce are well-defined. In Section~\ref{towardsfc} (see Lemma~\ref{welldefined}), we will give conditions on the parameters of $T$ for these quantities (and so the construction) to be well-defined.
	
	\begin{itemize}
		\item $\bm{n=1:}$ We first build $\mathcal{R}'_1$ and $\sn_1$ by an induction over $m\geq 1$. We could denote by $\mathcal{R}'_0$ the trivial tower $(X)$ with its base $B'_0\coloneq X$. At the end of step $n=1$, $S$ is not yet defined on the roof of the tower $\mathcal{R}'_1$, i.e. on its highest level, which is a Rokhlin tower of $S$.
		\begin{itemize}
			\item $\bm{m=1:}$ Let $q'_0>0$ be the largest multiple of $p_0$ such that $q'_0\leq q_0-1$.
			\begin{manualremark}{7.1}
				At this step, we simply have to assume $q_0>p_0$ for the integer $q'_0$ to be non-zero. However, for the well definition of other quantities at other steps, the conditions on the cutting parameters of $T$ get more and more technical, this is the reason why we first assume that the parameters of $T$ are chosen so that the positive quantities are well-defined and we will then state the conditions in Lemma~\ref{welldefined} (as mentioned before the beginning of the construction).
			\end{manualremark}
			For every $0\leq i\leq q'_0-1$, we define
			$$B'_{0,i}(1)\coloneq T^i(B_{1})$$
			and
			$$\sn_1(1)\colon\bigsqcup_{0\leq i\leq q'_0-2}{B'_{0,i}(1)}\to\bigsqcup_{0\leq i\leq q'_0-2}{B'_{0,i+1}(1)}$$
			coinciding with $T$ on its domain (hence every subset $B'_{0,i}(1)$ is composed of a unique $1$-brick $T^i(B_{1})$).
			
			\item $\bm{m>1:}$ Assume that the subsets $B'_{0,i}(M)$ have been built for every $1\leq M\leq m-1$ and $0\leq i\leq q'_0-1$. Let 
			$$W_{1,m}\coloneq X\setminus\bigsqcup_{1\leq M\leq m-1}{\bigsqcup_{0\leq i\leq q'_0-1}{B'_{0,i}(M)}}$$
			be the remaining piece of $X$ at the end of the $(m-1)$-th step (we could also define $W_{1,1}\coloneq X$).
			\begin{manualremark}{7.2}
				Notice that $m$-levels are either contained in $W_{1,m}$ or disjoint from it since $X\setminus W_{1,m}$ is composed of $M$-levels for $1\leq M\leq m-1$ and the Rokhlin towers are nested. This will more generally hold true for $W_{n,m}$ with $n\geq 2$.
			\end{manualremark}
			Let $r_{1,m}$ be the number of integers $j\in\llbracket 0,h_m-1\rrbracket$ such that $T^j(B_m)\subseteq W_{1,m}$, denoted by
			$$0\leq j_{1}^{(1,m)} <j_{2}^{(1,m)} <\ldots <j_{r_{1,m}}^{(1,m)} <h_m.$$
			Let $t_{1,m}$ be the positive integer such that $q'_0t_{1,m}$ is the largest multiple of $q'_0$ such that $q'_0t_{1,m}<r_{1,m}$ (we assume that we can choose the cutting parameters of $T$ for this integer to be positive, see Lemma~\ref{welldefined}). The first $q'_0t_{1,m}$ $m$-levels contained in $W_{1,m}$ are now used as $m$-bricks, they are split in $q'_0$ groups of $t_{1,m}$ $m$-bricks of the subsets $B'_{0,i}$ in the following way and the same will be done at steps $n>1$ (the fact that the inequality $q'_0t_{1,m}<r_{1,m}$ is strict, and the way we make the $q'_0$ groups will guarantee an easy control of the cocycles, see Lemma~\ref{cocycle} used for Lemmas~\ref{Kn} and~\ref{Dnm}). For every $0\leq i\leq q'_0-1$, we define
			
			$$B'_{0,i}(m)\coloneq \bigsqcup_{0\leq t\leq t_{1,m}-1}{T^{\left (j_{i+1+tq'_0}^{(1,m)}\right )}(B_m)}$$
			and $\sn_1(m)$ coinciding with $T^{\left (j_{i+2+tq'_0}^{(1,m)}\right )-\left (j_{i+1+tq'_0}^{(1,m)}\right )}$ on $T^{\left (j_{i+1+tq'_0}^{(1,m)}\right )}(B_m)$ for every $0\leq i\leq q'_0-2$ and $0\leq t\leq t_{1,m}-1$, so that each brick $T^{\left (j_{i+1+tq'_0}^{(1,m)}\right )}(B_m)$ is mapped on another $T^{\left (j_{i+2+tq'_0}^{(1,m)}\right )}(B_m)$. Thus $\sn_1(m)$ maps each $B'_{0,i}(m)$ on $B'_{0,i+1}(m)$ and this gives
			$$\sn_1(m)\colon\bigsqcup_{0\leq i\leq q'_0-2}{B'_{0,i}(m)}\to\bigsqcup_{0\leq i\leq q'_0-2}{B'_{0,i+1}(m)}.$$
		\end{itemize}
		
		\paragraph{End of Step $n=1$:} For every $0\leq i\leq q'_0-1$, we define
		$$B'_{0,i}\coloneq \bigsqcup_{m\geq 1}{B'_{0,i}(m)}$$
		(the set of its $m$-bricks for all $m\geq 1$), $B'_1\coloneq B'_{0,0}$ and
		$$\sn_1\colon\bigsqcup_{0\leq i\leq q'_0-2}{B'_{0,i}}\to\bigsqcup_{0\leq i\leq q'_0-2}{B'_{0,i+1}}$$
		coinciding with the maps $\sn_1(m)$ on their respective domain (see Figure~\ref{fignequal1}).
	\end{itemize}
	
	The universal odometer $S$ we want to build is partially defined on $X$. More precisely we define it on the domain $D_1\coloneq \bigsqcup_{0\leq i\leq q'_0-2}{B'_{0,i}}$ of $\sn_1$ so that it coincides with $\sn_1$. This gives the first Rokhlin tower $\mathcal{R}'_1\coloneq (B'_{0,0},\ldots,B'_{0,q'_1-1})=(B'_1,S(B'_1),\ldots,S^{q'_0-1}(B'_1))$. The next step will provide us a refinement $\mathcal{R}'_2$ of the tower $\mathcal{R}'_1$, allowing us to extend partially $S$ on the highest level of the $\mathcal{R}'_1$.
	\begin{figure}[ht!]
		\centering
		\includegraphics[width=0.6\linewidth]{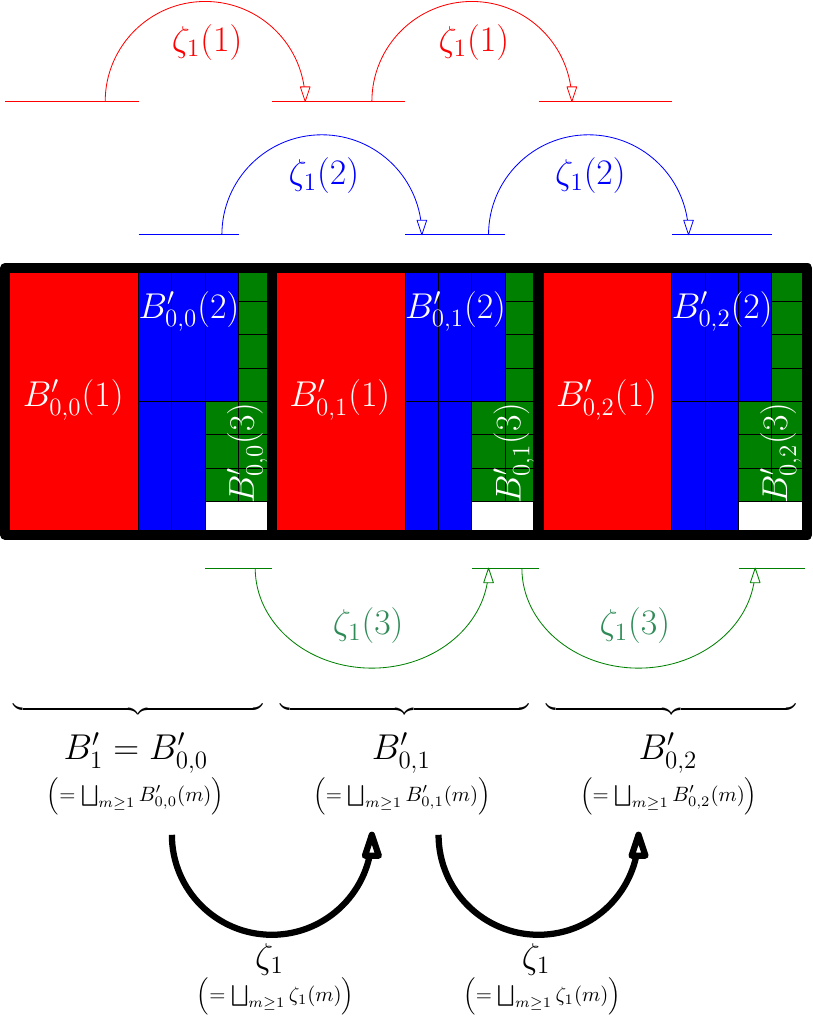}
		\caption{First step of the construction (i.e.~$n=1$).\\
			In Section~\ref{eqorbit}, we will define sets $E_{n,m}$ for every pair of integers $(n,m)$ satisfying $m\geq n\geq 1$. The set $E_{1,1}$ (resp. $E_{1,2}$; $E_{1,3}$) is the union of the red areas (resp. red and blue areas; red, blue and green areas).}
		\label{fignequal1}
	\end{figure}
	\begin{itemize}[resume]
		\item $\bm{n>1:}$ Assume that steps $1,\ldots,n-1$ have been achieved. There are nested towers $\mathcal{R}'_1,\ldots,\mathcal{R}'_{n-1}$. The $k$-th tower $\mathcal{R}'_k$ has $h'_k\coloneq q'_0\ldots q'_{k-1}$ levels and its base $B'_k$ is partitioned in $q'_k$ levels $B'_{k,0},\ldots,B'_{k,q'_k-1}$. These levels belong to $\mathcal{R}'_{k+1}$, whose base is $B'_{k+1}\coloneq B'_{k,0}$, with $\sn_{k+1}$ mapping $B'_{k,i}$ to $B'_{k,i+1}$. The map $S$ is defined on $D_1\sqcup\ldots\sqcup D_{n-1}$ using the maps $\sn_1,\ldots,\sn_{n-1}$, where $D_1\sqcup\ldots\sqcup D_{n-1}$ corresponds to the union of all the levels of $\mathcal{R}'_{n-1}$ except its roof.\par
		The map $S$ is not yet defined on the roof of $\mathcal{R}'_{n-1}$. By partitioning $B'_{n-1}$ in subsets $B'_{n-1,0},\ldots,B'_{n-1,q'_{n-1}-1}$, we will define $\mathcal{R}'_n$ which refines $\mathcal{R}'_{n-1}$ and a function $\sn_n$ mapping $B'_{n-1,i}$ to $B'_{n-1,i+1}$. The extension of $S$ will be defined on all the levels of $\mathcal{R}'_n$, except its roof (which is contained in the one of $\mathcal{R}'_{n-1}$). We will construct the subsets $B'_{n-1,i}$ as was done for the subsets $B'_{0,i}$, except that we only use the "material" in $B'_{n-1}$ to form the $m$-bricks of each $B'_{n-1,i}$. In order to do so, notice that the base $B'_{n-1}$ is exactly $B'_{n-2,0}$ (the first subset in the subdivision of $B'_{n-2}$) which is the disjoint union of subsets of the form $B'_{n-2,0}(m)$ for $m\geq n-1$. Moreover for all $n-1\leq M\leq m$, every $m$-level is contained in an $M$-level, we will then pick the new $m$-bricks in $B'_{n-2,0}(n), \ldots, B'_{n-2,0}(m)$. This motivates the definition of each set $W_{n,m}$ (the set of the remaining material to form $m$-bricks with an incremented integer $m$). We now discuss separately the following cases.
		\begin{itemize}
			\item $\bm{m=n:}$ Set
			\begin{equation}\label{Wnn}
				W_{n,n}\coloneq B'_{n-2,0}(n-1)\sqcup B'_{n-2,0}(n)
			\end{equation}
			and let $r_{n,n}$ be the number of integers $j\in\llbracket 0,h_n-1\rrbracket$ such that $T^j(B_n)\subseteq W_{n,n}$ (note that we could have defined $r_{1,1}=q_0$), denoted by
			$$0\leq j_{1}^{(n,n)} <j_{2}^{(n,n)} <\ldots <j_{r_{n,n}}^{(n,n)} <h_n.$$
			Let $q'_{n-1}$ be the largest multiple of $p_{n-1}$ such that $q'_{n-1}< r_{n,n}$ (we assume that we can choose the cutting parameters of $T$ for this integer to be positive, see Lemma~\ref{welldefined}). We then define for every \mbox{$0\leq i\leq q'_{n-1}-1$},
			$$B'_{n-1,i}(n)\coloneq T^{\left (j_{i+1}^{(n,n)}\right )}(B_n),$$
			meaning that among the $n$-levels in $W_{n,n}$, the $n$-bricks at step $(n,n)$ are exactly the first $q'_{n-1}$ ones (and set $t_{n,n}=1$ for consistency later on). Let
			$$\sn_n(n)\colon\bigsqcup_{0\leq i\leq q'_{n-1}-2}{B'_{n-1,i}(n)}\to\bigsqcup_{0\leq i\leq q'_{n-1}-2}{B'_{n-1,i+1}(n)}$$
			be the map coinciding with $T^{\left (j_{i+2}^{(n,n)}\right )-\left (j_{i+1}^{(n,n)}\right )}$ on each $B'_{n-1,i}(n)$, so that $B'_{n-1,i}(n)$ is mapped to $B'_{n-1,i+1}(n)$.
			
			\item $\bm{m>n:}$ Set
			\begin{equation}\label{Wnm}
				W_{n,m}\coloneq \left (\bigsqcup_{n-1\leq M\leq m}{B'_{n-2,0}(M)}\right )\setminus\left (\bigsqcup_{n\leq M\leq m-1}{\bigsqcup_{0\leq i\leq q'_{n-1}-1}{B'_{n-1,i}(M)}}\right )
			\end{equation}
			and let $r_{n,m}$ be the number of integers $j\in\llbracket 0,h_m-1\rrbracket$ such that $T^j(B_m)\subseteq W_{n,m}$, denoted by
			$$0\leq j_{1}^{(n,m)} <j_{2}^{(n,m)} <\ldots <j_{r_{n,m}}^{(n,m)} <h_m.$$
			Let $t_{n,m}$ be the positive integer such that $q'_{n-1}t_{n,m}$ is the largest multiple of $q'_{n-1}$ such that $q'_{n-1}t_{n,m}<r_{n,m}$ (we assume that we can choose the cutting parameters of $T$ for this integer to be positive, see Lemma~\ref{welldefined}). The first $q'_{n-1}t_{n,m}$ $m$-levels contained in $W_{n,m}$ are now used as $m$-bricks at step $(n,m)$, they are split in $q'_{n-1}$ groups of $t_{n,m}$ $m$-bricks of the subsets $B'_{n-1,i}$ in the following way. For every $0\leq i\leq q'_{n-1}-1$, we define
			$$B'_{n-1,i}(m)\coloneq \bigsqcup_{0\leq t\leq t_{n,m}-1}{T^{\left (j_{i+1+tq'_{n-1}}^{(n,m)}\right )}(B_m)}$$
			and $\sn_n(m)$ coinciding with $T^{\left (j_{i+2+tq'_{n-1}}^{(n,m)}\right )-\left (j_{i+1+tq'_{n-1}}^{(n,m)}\right )}$ on $T^{\left (j_{i+1+tq'_{n-1}}^{(n,m)}\right )}(B_m)$ for every $0\leq i\leq q'_{n-1}-2$, $0\leq t\leq t_{n,m}-1$, so that each $m$-brick $T^{\left (j_{i+1+tq'_{n-1}}^{(n,m)}\right )}(B_m)$ is mapped on another $T^{\left (j_{i+2+tq'_{n-1}}^{(n,m)}\right )}(B_m)$. Thus $\sn_n(m)$ maps each $B'_{n-1,i}(m)$ on $B'_{n-1,i+1}(m)$ and this gives
			$$\sn_n(m)\colon\bigsqcup_{0\leq i\leq q'_{n-1}-2}{B'_{n-1,i}(m)}\to\bigsqcup_{0\leq i\leq q'_{n-1}-2}{B'_{n-1,i+1}(m)}.$$
			
		\end{itemize}
		
		\paragraph{End of Step $n$:} We define for every $0\leq i\leq q'_{n-1}-1$,
		$$B'_{n-1,i}\coloneq \bigsqcup_{m\geq n}{B'_{n-1,i}(m)}$$
		(the set of its $m$-bricks for $m\geq n$), $B'_n=B'_{n-1,0}$ and
		$$\sn_n\colon\bigsqcup_{0\leq i\leq q'_{n-1}-2}{B'_{n-1,i}}\to\bigsqcup_{0\leq i\leq q'_{n-1}-2}{B'_{n-1,i+1}}$$
		coinciding with the maps $\sn_n(m)$ on their respective domain (see Figure~\ref{fign} for step $n=2$, after the first step illustrated in Figure~\ref{fignequal1}).\par
		As the base $B'_{n-1}$ of $\mathcal{R}'_{n-1}$ is partitioned in $B'_{n-1,0}\sqcup\ldots\sqcup B'_{n-1,q'_{n-1}-1}$, its highest level $S^{h'_{n-1}-1}(B'_{n-1})$ is partitioned in $S^{h'_{n-1}-1}(B'_{n-1,0})\sqcup\ldots\sqcup S^{h'_{n-1}-1}(B'_{n-1,q'_{n-1}-1})$. The map $S$ is extended in the following way. On
		$$D_n\coloneq S^{h'_{n-1}-1}(B'_{n-1,0})\sqcup\ldots\sqcup S^{h'_{n-1}-1}(B'_{n-1,q'_{n-1}-2}),$$
		it coincides with $\sn_nS^{-(h'_{n-1}-1)}$. So $S$ maps $S^{h'_{n-1}-1}(B'_{n-1,i})$ on $B'_{n-1,i+1}$. This gives a Rokhlin tower $\mathcal{R}'_n$ for $S$, nested in the previous one, of base $B'_n\coloneq B'_{n-1,0}$ and height $h'_{n}\coloneq q'_0\ldots q'_{n-1}$. Now $S$ is defined on $(D_1\sqcup\ldots\sqcup D_{n-1})\sqcup D_n$. The set $D_n$ consists in the levels of $\mathcal{R}_n$, except the highest one, which are contained in the highest level of $\mathcal{R}'_{n-1}$.
	\end{itemize}
	
	\begin{figure}[ht]
		\centering
		\includegraphics[width=0.9\linewidth]{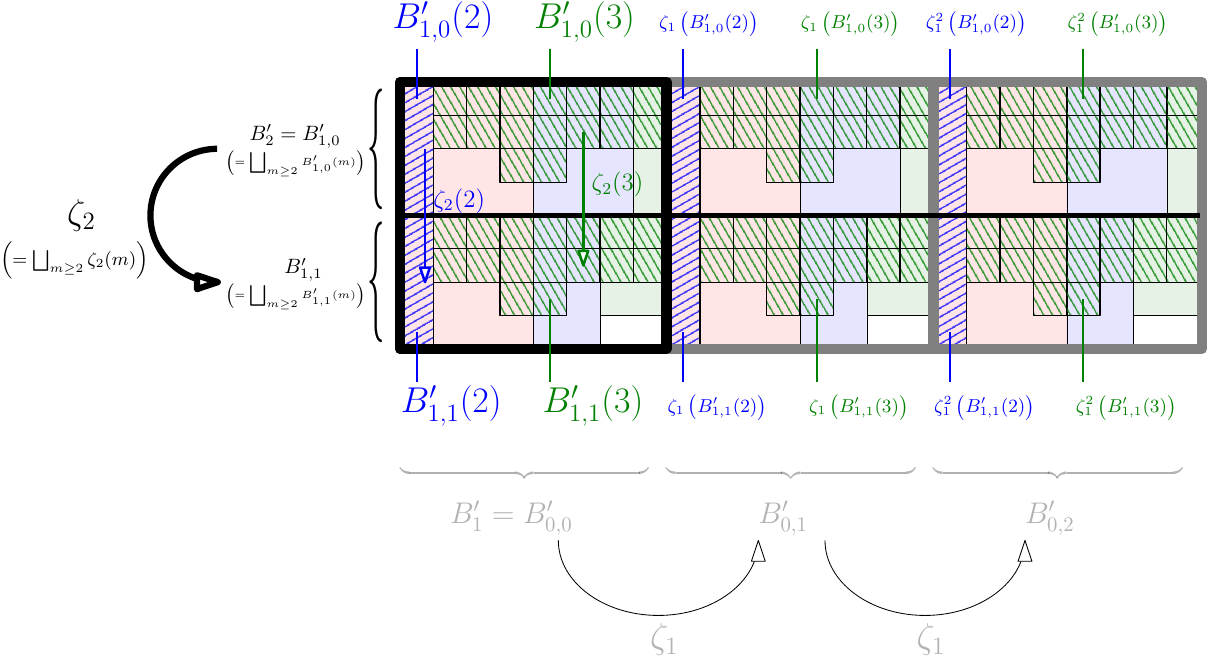}
		\caption{From the first step (illustrated in Figure~\ref{fignequal1}) to the second one.\\
			In $B'_1$, we inductively build $B'_{1,i}(2), B'_{1,i}(2), B'_{1,i}(3), \ldots$ for every $0\leq i\leq q'_1-1$ (in this example, we have $q'_1=2$). Each set $B'_{1,i}(2)$ is composed of a unique $2$-level in $B'_{0,0}(1)\sqcup B'_{0,0}(2)$ (i.e.~in pale red and pale blue areas). Then each set $B'_{1,i}(3)$ is composed of $3$-levels in $B'_{0,0}(1)\sqcup B'_{0,0}(2)\sqcup B'_{0,0}(3)$ (i.e.~in pale red, pale blue and pale green areas) and so on. The structure that we build in $B'_1=B'_{0,0}$ can be translated in $B'_{0,1}$ and $B'_{0,2}$ using the map $\zeta_1$.\\
			In Section~\ref{eqorbit}, we will define sets $E_{n,m}$ for every pair of integers $(n,m)$ satisfying $m\geq n\geq 1$. The set $E_{2,1}$ (resp. $E_{2,2}$) is the union of the areas hatched in blue (resp. in blue or green).}
		\label{fign}
	\end{figure}
	
	\begin{remark}\label{capture}
		Notice that the inclusion of the $S$-orbits in the $T$-orbits is easy since $S$ is defined from maps $\sn_n(m)$ which are "piecewise powers of $T$".\par
		The reverse inclusion will follow from the fact that we have $t_{n,n}=1$ for every $n\geq 1$ (at step $(n,n)$ we form groups of only one $n$-level). Indeed, uniqueness implies that these chosen blocks are linked by $\sn_n(n)$ and hence clearly by $S$ (on the contrary, an $m$-level, for $m>n$, of $B'_{n-1,i}$ is mapped by $\sn_n(m)$ to only one of the $t_{n,m}$ $m$-levels of $B'_{n-1,i+1}$, and not to the other). Thus ensuring that the unique $n$-brick of each $B'_{n-1,i}$ is a large part of it enables the system $S$ to capture most of the $T$-orbits.
	\end{remark}
	
	\subsection{First properties of this construction}\label{firstprop}
	
	Recall that we consider a cutting-and-stacking construction of $T$ with the same notations as in Definition~\ref{defr1} ($q_n$, $\sigma_{n,i}$, $\mathcal{R}_n$, $X_n$, $\varepsilon_n$, ...), and the sequences $(h_n)$, $(\sigma_{n})$ and $(Z_n)$ associated to the cutting and spacing parameters, and the notations $q'_n$, $\mathcal{R}'_n$, ... refer to the construction of $S$.\par
	We state some important properties preparing for further results in Section~\ref{eqorbit}. Many of them enable us to only take into account the combinatorics behind a cutting-and-stacking construction.
	We assume that all the "largest multiples" (for every $n<m$, the largest multiple $q'_{n-1}$ of $p_{n-1}$ such that $q'_{n-1}<r_{n,n}$, and the largest multiple $q'_{n-1}t_{n,m}$ of $q'_{n-1}$ such that $q'_{n-1}t_{n,m}<r_{n,m}$) are non-zero. In Section~\ref{towardsfc} (see Lemma~\ref{welldefined}), we will see how to choose the parameters for the construction to be well-defined.
	
	\begin{lemma}\label{partition}
		Every tower $\mathcal{R}'_n$ is a partition of $X$.
	\end{lemma}
	
	\begin{proof}[Proof of Lemma~\ref{partition}]
		Let $n\geq 1$. The levels of $\mathcal{R}'_n$ are pairwise disjoint by the definition of $(W_{n,m})_{m\geq n}$. It remains to show that $\mathcal{R}'_n$ covers the whole space. Recall that $X_n$ denotes the subset covered by the tower $\mathcal{R}_n$, and $\varepsilon_n$ the measure of its complement.\par
		The result holds for $n=1$ since $\mu(W_{1,m})\underset{m\to +\infty}{\to}0$. Indeed $W_{1,m+1}\cap X_{m}$ is the union of the $m$-levels which are not chosen at step $(1,m)$. By the definition of $t_{1,m}$, there are at most $q'_0$. So we have $W_{1,m+1}\leq \varepsilon_m+q'_0\mu(B_m)\to 0$.\par
		For $n>1$, it suffices to show that the levels $B'_{n-1,0},\ldots,B'_{n-1,q'_{n-1}-1}$ of $\mathcal{R}'_n$ form a partition of the base $B'_{n-1}$ of $\mathcal{R}'_{n-1}$. We have to show that the measure of
		$$\tilde{W}_{n,m}\coloneq B'_{n-1}\setminus\left (\bigsqcup_{n\leq M\leq m-1}{\bigsqcup_{0\leq i\leq q'_{n-1}-1}{B'_{n-1,i}(M)}}\right )$$
		tends $0$ as $m\to +\infty$. But this set $\tilde{W}_{n,m}$ is the disjoint union of $\bigsqcup_{M\geq m+1}{B'_{n-2,0}(M)}$ and $W_{n,m}$. It is clear that
		$$\mu\left (\bigsqcup_{M\geq m+1}{B'_{n-2,0}(M)}\right )\underset{m\to +\infty}{\to}0,$$
		since $\mu$ is a finite measure. The set $W_{n,m}$ is obtained from $W_{n,m-1}$ by adding $B'_{n-2,0}(m)$ and removing $q'_{n-1}t_{n,m-1}$ $(m-1)$-levels. Thus we have $\mu\left (W_{n,m}\right )\underset{m\to +\infty}{\to}0$ by the definition of $(t_{n,m})_{m\geq n}$. Hence we have $\mu\left (\tilde{W}_{n,m}\right )\underset{m\to +\infty}{\to}0$ and we are done.
	\end{proof}
	
	As a consequence, if $(\mathcal{R}'_n)_n$ increases to the $\sigma$-algebra $\A$ (this will be proved in Corollary~\ref{odouniv}), then $S$ is a rank-one system without spacer, so this is an odometer.
	
	\begin{lemma}\label{explicit}
		Let $n\geq 1$. On the base $B'_n$ of the $n$-th $S$-Rokhlin tower $\mathcal{R}'_n$, $S$ is defined as follows. For every $0\leq i\leq h'_n-1$, we have
		$$S^i=\sn_1^{i_0}\ldots \sn_n^{i_{n-1}}\text{ on }B'_n$$
		with \mbox{$i_{0}\in\llbracket 0,q'_{0}-1\rrbracket,\ldots,i_{n-1}\in\llbracket 0,q'_{n-1}-1\rrbracket$} such that \mbox{$i=\sum\limits_{\ell=0}^{n-1}{q'_0\ldots q'_{\ell -1}i_{\ell}}=\sum\limits_{\ell=0}^{n-1}{h'_{\ell}i_{\ell}}$}.
	\end{lemma}
	
	\begin{proof}[Proof of Lemma~\ref{explicit}]
		By induction over $n\geq 1$. It is clear for $n=1$ since $S$ coincides with $\sn_1$ on the levels of $\mathcal{R}'_1$ except its roof. Assume that the result holds for $n\geq 1$. The tower $\mathcal{R}'_n$ is divided in $q'_n$ subcolumns whose levels are exactly the ones of $\mathcal{R}'_{n+1}$, and the $i_n$-th subcolumn ($0\leq i_n\leq h'_{n}-1$) is the $S$-Rokhlin tower of height $h'_n$ and base $B'_{n,i_n}$. Let $0\leq i\leq h'_{n+1}-1$. By the equality $B'_{n+1}=B'_{n,0}$ and by the definition of $S$ from $\sn_{n+1}$ (at the end of step $n+1$ of the construction), $S^i=S^j\sn_{n+1}^{i_n}$ on $B'_{n+1}$ for non-negative integers $i_n$ and $j$ such that $i=i_nh'_n+j$ and $j<h'_n$. The set $\sn_{n+1}^{i_n}(B'_{n,0})$ is equal to $B'_{n,i_n}$, so this is a subset of $B'_{n}$, hence the result by the induction hypothesis.
	\end{proof}
	
	Therefore the subset $D_n$ defined in the construction can be written as follows:
	\begin{equation}\label{explicit2}
		\begin{array}{lcl}
			D_n&=&\displaystyle \sn_1^{q'_0-1}\ldots \sn_{n-1}^{q'_{n-2}-1}\left (\bigsqcup_{0\leq i_n\leq q'_{n-1}-2}{B'_{n-1,i_n}}\right )\\
			&=&\displaystyle \sn_1^{q'_0-1}\ldots \sn_{n-1}^{q'_{n-2}-1}\left (\bigsqcup_{0\leq i_n\leq q'_{n-1}-2}{\sn_n^{i_n}(B'_{n-1,0})}\right )
		\end{array}
	\end{equation}
	and $S$ coincides with $\sn_n\sn_{n-1}^{-(q'_{n-2}-1)}\ldots \sn_1^{-(q'_0-1)}$ on $D_n$.\newline
	
	By the \textbf{cocycle} of $\sn_n(m)$, we mean the integer-valued map defined on the domain of $\sn_n(m)$ and which maps $x$ to the unique integer $k$ satisfying $\sn_n(m)x=T^kx$.
	
	\begin{lemma}\label{cocycle}
		The cocycle of $\sn_n(m)$ is positive and bounded above by $h_{m-1}+Z_{m-1}$.
	\end{lemma}
	
	\begin{proof}[Proof of Lemma~\ref{cocycle}]
		By definition, for fixed integers $0\leq i\leq q'_{n-1}-2$ and $0\leq t\leq t_{n,m}-1$, the cocycle on $\mathbb{B}\coloneq T^{\left (j_{i+1+tq'_{n-1}}^{(n,m)}\right )}(B_m)$ takes the value $\Delta j\coloneq j_{i+2+tq'_{n-1}}^{(n,m)}-j_{i+1+tq'_{n-1}}^{(n,m)}$. Let us recall that the integers
		$$0\leq j_{1}^{(n,m)} <j_{2}^{(n,m)} <\ldots <j_{r_{n,m}}^{(n,m)} <h_m$$
		are the set of indices $j\in\llbracket 0,h_m-1\rrbracket$ such that $T^j(B_m)\subseteq W_{n,m}$. Thus $\Delta j$ is obviously positive. Let us fix an $(m-1)$-level $\mathbb{B}^{\star}$ which is not chosen at step $(n,m-1)$, so it is contained in $W_{n,m}$. If $m$ is equal to $n$, we can choose $\mathbb{B}^{\star}=B'_{n-2,0}(n-1)$. For $m>n$, the existence of $\mathbb{B}^{\star}$ is granted by the fact that we have $q'_{n-1}t_{n,m-1}<r_{n,m-1}$. We write it as $\mathbb{B}^{\star}=T^{k_0}(B_{m-1})$, where $k_0$ is an integer in $\llbracket 0, h_{m-1}-1\rrbracket$.\par
		By definition, $\Delta j$ is the least positive integer $j$ such that $T^j(\mathbb{B})$ is in $W_{n,m}$. Moreover the $m$-levels of $\mathbb{B}^{\star}$ are in $W_{n,m}$. Therefore the consecutive $m$-levels $T(\mathbb{B}),\ldots ,T^{\Delta j-1}(\mathbb{B})$ are not in $\mathbb{B}^{\star}$.
		\paragraph{First case.} In the tower $\mathcal{R}_m$, assume that the $m$-levels $T(\mathbb{B}),\ldots ,T^{\Delta j-1}(\mathbb{B})$ are before $T^{k_0}(B_{m-1,0})$, i.e.~before the first $m$-level of $\mathbb{B}^{\star}$. Therefore the enumeration $\mathbb{B},\ldots ,T^{\Delta j}(\mathbb{B})$ is included in the enumeration
		$$\Sigma_{m-1,0,1},\ldots,\Sigma_{m-1,0,\sigma_{m-1,0}},B_{m-1,0},\ldots,T^{k_0}(B_{m-1,0}),$$
		implying that $\Delta j\leq \sigma_{m-1,0}+k_0\leq Z_{m-1}+h_{m-1}$.
		\paragraph{Second case.} Now assume that $T(\mathbb{B}),\ldots ,T^{\Delta j-1}(\mathbb{B})$ are after $T^{k_0}(B_{m-1,q_{m-1}-1})$, i.e.~after the last $m$-level of $\mathbb{B}^{\star}$. Therefore the enumeration $\mathbb{B},\ldots ,T^{\Delta j}(\mathbb{B})$ is included in the enumeration
		$$T^{k_0}(B_{m-1,q_{m-1}-1}),\ldots,T^{h_{m-1}-1}(B_{m-1,q_{m-1}-1}),\Sigma_{m-1,q_{m-1},1},\ldots,\Sigma_{m-1,q_{m-1},\sigma_{m-1,q_{m-1}}},$$
		and we get $\Delta j\leq (h_{m-1}-k_0-1)+\sigma_{m-1,q_{m-1}}\leq h_{m-1}+Z_{m-1}$.
		\paragraph{Third case.} Finally if $T(\mathbb{B}),\ldots ,T^{\Delta j-1}(\mathbb{B})$ are between $T^{k_0}(B_{m-1,i})$ and $T^{k_0}(B_{m-1,i+1})$ for some $0\leq i\leq q_{m-1}-2$, i.e.~between two consecutive $m$-levels of $\mathbb{B}^{\star}$, then the enumeration $\mathbb{B},\ldots ,T^{\Delta j}(\mathbb{B})$ is included in the enumeration
		$$T^{k_0}(B_{m-1,i}),\ldots,T^{h_{m-1}-1}(B_{m-1,i}),\Sigma_{m-1,i,1},\ldots,\Sigma_{m-1,i,\sigma_{m-1,i}},B_{m-1,i+1},\ldots,T^{k_0}(B_{m-1,i+1}),$$
		this gives $\Delta j\leq (h_{m-1}-k_0-1)+\sigma_{m-1,i}+(k_0+1)\leq h_{m-1}+Z_{m-1}$.
	\end{proof}
	
	\begin{lemma}\label{relbrick}
		An $m$-brick at step $n$ is included in an $M$-brick at step $n-1$ for some \mbox{$n-1\leq M\leq m$}.
	\end{lemma}
	
	\begin{proof}[Proof of Lemma~\ref{relbrick}]
		This follows directly from the definition of $W_{n,m}$ in the construction (see Section~\ref{theconstruction}). Indeed the "$(M)$" in "$B'_{n-2,0}(M)$" means that we only consider the $M$-bricks, at step $n-1$, composing $B'_{n-2,0}$.
	\end{proof}
	
	We now present a combinatorial description of the construction.
	
	\begin{lemma}\label{relrec}
		The quantities $r_{n,m}, q_n, q'_n, t_{n,m}, \sigma_{n}$ satisfy the following recurrence relation:
		$$\begin{array}{rl}
			
			&\ \ \ \ \displaystyle t_{0,1}\coloneq 0;\\
			
			&\\
			
			\displaystyle \text{for }m\geq 2,&\ \ \ \ \displaystyle t_{0,m}\coloneq \sigma_{m-1};\\
			
			&\\
			
			\displaystyle \text{for }m=n\geq 1,&\left\{
			\begin{array}{l}
				\displaystyle r_{n,n}=q_{n-1}+t_{n-1,n},\\
				\displaystyle q'_{n-1}=\left \lfloor \frac{r_{n,n}-1}{p_{n-1}}\right \rfloor p_{n-1},\\
				\displaystyle t_{n,n}=1;
			\end{array}
			\right.\\
			
			&\\
			
			\displaystyle \text{for }m>n\geq 1,&\left\{
			\begin{array}{l}
				\displaystyle r_{n,m}=q_{m-1}(r_{n,m-1}-q'_{n-1}t_{n,m-1})+t_{n-1,m},\\
				\displaystyle t_{n,m}=\left \lfloor \frac{r_{n,m}-1}{q'_{n-1}}\right \rfloor.
			\end{array}\right.
			
		\end{array}$$
	\end{lemma}
	
	During the construction, some integers have been defined for consistency ($r_{1,1}\coloneq q_0$, $t_{n,n}\coloneq 1$). Note that in this lemma, we also define the integers $t_{n,m}$ for $n=0$. This enables us to extend the relations
	$$r_{n,n}=q_{n-1}+t_{n-1,n}\text{ and }r_{n,m}=q_{m-1}(r_{n,m-1}-q'_{n-1}t_{n,m-1})+t_{n-1,m}$$
	for $n=1$.
	
	\begin{proof}[Proof of Lemma~\ref{relrec}]
		Case $n=1$. For $m=1$, the $r_{1,1}$ $1$-levels potentially chosen to be $1$-bricks are exactly the levels of $\mathcal{R}_1$, so we have $r_{1,1}=q_0+t_{0,1}$ since $t_{0,1}\coloneq 0$. We choose $q'_0$ of them, where $q'_0$ is the largest multiple of $p_0$ such that $q'_0< r_{1,1}$, so $q'_0$ is equal to $\lfloor (r_{1,1}-1)/p_0\rfloor p_0$. Finally $q'_0$ is obviously equal to $q'_0t_{1,1}$ since $t_{1,1}\coloneq 1$. For $m>1$, there are $r_{n,m}$ $m$-levels in $W_{1,m}$: some of them are in the $r_{1,m-1}-q'_{0}t_{1,m-1}$ $(m-1)$-levels which are not chosen at step $(1,m-1)$ and the other are the spacers from $\mathcal{R}_{m-1}$ to $\mathcal{R}_{m}$. So we have
		$$r_{1,m}=q_{m-1}(r_{1,m-1}-q'_{0}t_{1,m-1})+\sigma_{m-1}$$
		and we set $t_{0,m}\coloneq \sigma_{m-1}$. We choose $q'_0t_{1,m}$ of them as $m$-bricks, where $q'_0t_{1,m}$ is the largest multiple of $q'_0$ such that $q'_0t_{1,m}<r_{1,m}$, i.e.~$t_{1,m}\coloneq \lfloor (r_{1,m}-1)/q'_0\rfloor$.\par
		Case $n>1$. For $m=n$, there are $r_{n,n}$ $n$-levels in $W_{n,n}=B'_{n-2,0}(n-1)\sqcup B'_{n-2,0}(n)$. First, since we have $t_{n-1,n-1}=1$, the set $B'_{n-2,0}(n-1)$ is an $(n-1)$-brick at step $n-1$ and it contains $q_{n-1}$ $n$-levels. Secondly $B'_{n-2,0}(n)$ is the union of $t_{n-1,n}$ $n$-bricks. Hence we have $r_{n,n}=q_{n-1}+t_{n-1,n}$. By definition, $q'_{n-1}$ is equal to $\lfloor (r_{n,n}-1)/p_{n-1}\rfloor p_{n-1}$ and obviously to $q'_{n-1}t_{n,n}$ with $t_{n,n}\coloneq 1$. For $m>n$, there are $r_{n,m}$ $m$-levels in $W_{n,m}$. This set is composed of
		$$\left (\bigsqcup_{n-1\leq M\leq m-1}{B'_{n-2,0}(M)}\right )\setminus\left (\bigsqcup_{n\leq M\leq m-1}{\bigsqcup_{0\leq i\leq q'_{n-1}-1}{B'_{n-1,i}(M)}}\right )$$
		and
		$$B'_{n-2,0}(m).$$
		The first one is the union of the $r_{n,m-1}-q'_{n-1}t_{n,m-1}$ $(m-1)$-levels which are not chosen at step $(n,m-1)$, and the second one is built at step $(n-1,m)$ from its $t_{n-1,m}$ $m$-bricks. So we have
		$$r_{n,m}=q_{m-1}(r_{n,m-1}-q'_{n-1}t_{n,m-1})+t_{n-1,m}.$$
		We choose $q'_{n-1}t_{n,m}$ of these $m$-levels as $m$-bricks at this step, where $q'_{n-1}t_{n,m}$ is the largest multiple of $q'_{n-1}$ such that $q'_{n-1}t_{n,m}< r_{n,m}$, i.e.~$t_{n,m}\coloneq \lfloor (r_{n,m}-1)/q'_{n-1}\rfloor$.
	\end{proof}
	
	It will be more convenient to use the following slight modification of Lemma~\ref{relrec}:
	\begin{equation}\label{relrec2}
		\begin{array}{rl}
			
			&\ \ \ \ \displaystyle t_{0,1}=0;\\
			
			&\\
			
			\displaystyle \text{for }m\geq 2,&\ \ \ \ \displaystyle t_{0,m}=\sigma_{m-1};\\
			
			&\\
			
			\displaystyle \text{for }m=n\geq 1,&\left\{
			\begin{array}{l}
				\displaystyle r_{n,n}=q_{n-1}+t_{n-1,n},\\
				\displaystyle q'_{n-1}\leq r_{n,n}-1,\\
				\displaystyle t_{n,n}=1;
			\end{array}
			\right.\\
			
			&\\
			
			\displaystyle \text{for }m>n\geq 1,&\left\{
			\begin{array}{l}
				\displaystyle r_{n,m}\leq q_{m-1}q'_{n-1}+t_{n-1,m},\\
				\displaystyle t_{n,m}\leq\frac{r_{n,m}-1}{q'_{n-1}}.
			\end{array}\right.
			
		\end{array}
	\end{equation}
	
	This is a consequence of the inequalities $\lfloor x\rfloor\leq x$ and $r_{n,m-1}-q'_{n-1}t_{n,m-1}\leq q'_{n-1}$ (by the definition of $t_{n,m-1}$).\newline
	
	As the strategy will be to recursively choose large enough cutting parameters $q_n$ for $T$, we would like to understand the asymptotic behaviour of $q'_n$ as $q_n$ increases. Then the goal is to find bounds for $q'_n/q_n$.
	
	\begin{lemma}\label{asympt}
		For every $n\geq 0$, we have
		$$q'_{n}\geq q_{n}-(1+p_{n}).$$
	\end{lemma}
	
	\begin{proof}[Proof of Lemma~\ref{asympt}]
		Using the equalities $q'_{n}=\left \lfloor \frac{r_{n+1,n+1}-1}{p_{n}}\right \rfloor p_{n}$ and $r_{n+1,n+1}=q_{n}+t_{n,n+1}$ in Lemma~\ref{relrec}, where the integer $t_{n,n+1}$ is non-negative, we get
		$$q'_{n}\geq\left (\frac{r_{n+1,n+1}-1}{p_{n}}-1\right )p_{n}\geq q_{n}-1-p_{n}$$
		and we are done.
	\end{proof}
	
	We have found a lower bound for $q'_n/q_n$ (up to some additional term $-(1+p_n)$). Let us find an upper bound. 
	
	\begin{lemma}\label{preq'_n/q_nleq}
		For every $n\geq 1$, we have
		$$q'_{n}\leq 3q_{n}+\frac{\sigma_{n}}{q'_{0}\ldots q'_{n-1}}.$$
	\end{lemma}
	
	With an asymptotic control on $\sigma_n$, using flexible classes, we will be able to get $q'_n\leq 4q_n$ (see Lemma~\ref{boundusingcondfc}).
	
	\begin{proof}[Proof of Lemma~\ref{preq'_n/q_nleq}]
		By induction over $i\in\llbracket 0,n-1\rrbracket$ (with $n\geq 1$) and using \eqref{relrec2}, we show that
		$$q'_{n}\leq q_{n}\left (2+\sum_{j=1}^{i}{\prod_{k=1}^{j}{\frac{1}{q'_{n-k}}}}\right )+t_{n-1-i,n+1}\prod_{k=1}^{i+1}{\frac{1}{q'_{n-k}}}.$$
		
		For $i=0$, we have $q'_{n}< r_{n+1,n+1}= q_{n}+t_{n,n+1}$ and
		$$t_{n,n+1}\leq \frac{r_{n,n+1}-1}{q'_{n-1}}\leq\frac{1}{q'_{n-1}}\left (q_{n}q'_{n-1}+t_{n-1,n+1}\right )=q_{n}+\frac{t_{n-1,n+1}}{q'_{n-1}},$$
		so we get $q'_{n}\leq 2q_{n}+\frac{t_{n-1,n+1}}{q'_{n-1}}$. For $0\leq i\leq n-2$, we have
		$$t_{n-1-i,n+1}\leq\frac{r_{n-1-i,n+1}-1}{q'_{n-2-i}}\leq\frac{1}{q'_{n-2-i}}\left (q_{n}q'_{n-2-i}+t_{n-2-i,n+1}\right )=q_{n}+\frac{t_{n-2-i,n+1}}{q'_{n-2-i}}.$$
		
		If the result holds true for $i$, we get
		$$\begin{array}{lcl}
			\displaystyle q'_{n}&\leq &\displaystyle q_{n}\left (2+\sum_{j=1}^{i}{\prod_{k=1}^{j}{\frac{1}{q'_{n-k}}}}\right )+\left (q_{n}+\frac{t_{n-2-i,n+1}}{q'_{n-2-i}}\right )\prod_{k=1}^{i+1}{\frac{1}{q'_{n-k}}}\\
			&=&\displaystyle q_{n}\left (2+\sum_{j=1}^{i+1}{\prod_{k=1}^{j}{\frac{1}{q'_{n-k}}}}\right )+t_{n-1-(i+1),n}\prod_{k=1}^{i+2}{\frac{1}{q'_{n-k}}},
		\end{array}$$
		so the result is also true for $i+1$.
		
		Taking $i=n-1$, this gives the lemma since $q'_{\ell}\geq 2$ for every integer $\ell\geq 1$.
	\end{proof}
	
	\subsection{Towards flexible classes}\label{towardsfc}
	
	We now explain why flexible classes fit in this construction.\par
	First a condition for the construction to be well-defined needs an inductive choice of the cutting parameters $(q_n)_{n\geq 0}$ of $T$ (see Lemma~\ref{dependence}). Secondly, a control on the spacing parameters will imply useful asymptotic controls for the quantification of the cocycles (see Lemma~\ref{boundusingcondfc}). Note that, in the proof of Theorem~\ref{thfc} (see Section~\ref{proofth}), we will need other estimates to quantify the cocycles. It will be possible, again using the definition of a flexible class, to inductively build large enough cutting parameters in order to have these estimates.\par
	If the parameters are chosen according to a set $\mathcal{F}_{\mathcal{C}}\subseteq\mathcal{P}^*$ associated to a flexible class $\mathcal{C}$, the underlying rank-one system has the desired property, i.e.~it is in $\mathcal{C}$, and is orbit equivalent to the universal odometer, with some quantification guaranteed by the control of the spacing parameters and by the fact that the cutting parameters $q_n$ have been recursively chosen and large enough.
	
	\begin{lemma}\label{dependence}
		Let $T$ be a rank-one system with cutting and spacing parameters
		$$(q_n,(\sigma_{n,0},\ldots,\sigma_{n,q_n}))_{n\geq 0}$$
		such that the construction in Section~\ref{theconstruction} is well-defined. Then, for every $n\in\N$, $q'_n$ only depends on $(q_k,(\sigma_{k,0},\ldots,\sigma_{k,q_k}))_{0\leq k\leq n}$.
	\end{lemma}
	
	\begin{proof}[Proof of Lemma~\ref{dependence}]
		This directly follows from Lemma~\ref{relrec}.
	\end{proof}
	
	Then the main novelty in this paper is to build the rank-one system $T$ while we are building the universal odometer $S$. Once $(q'_0,\ldots,q'_n)$ has been built from $(q_k,(\sigma_{k,0},\ldots,\sigma_{k,q_k}))_{0\leq k\leq n}$, we are free to choose $(q_{n+1},(\sigma_{n+1,0},\ldots,\sigma_{n+1,q_{n+1}}))$ for the definition of $q'_{n+1}$. The recursive definition of the cutting parameters is one of the main ideas behind the definition of a flexible class, and it allows to find cutting parameters satisfying some assumptions, for example the assumptions of the following lemma.
	
	\begin{lemma}\label{welldefined}
		Assume that for every $n\in\N$,
		\begin{equation}\label{crit1}
			q_n>\max{(p_n,q'_0,\ldots,q'_{n-1})}.
		\end{equation}
		Then the construction is well-defined, i.e.~all the "largest multiples" are non-zero (that is, the largest multiple $q'_{n-1}$ of $p_{n-1}$ such that $q'_{n-1}<r_{n,n}$, and the largest multiple $q'_{n-1}t_{n,m}$ of $q'_{n-1}$ such that $q'_{n-1}t_{n,m}<r_{n,m}$).
	\end{lemma}
	
	\begin{remark}\label{q0}
		Without loss of generality, we can assume that $p_0$ is equal to $2$. Therefore the assumption of Lemma~\ref{welldefined} for $n=0$ requires that $q_0$ is greater than $2$, which explains the second item of Definition~\ref{deffc}.
	\end{remark}
	
	\begin{proof}[Proof of Lemma~\ref{welldefined}]
		First, let us prove this result at step $n=1$ of the outer recursion. At step $m=1$ of the inner recursion, $q_0$ is greater than $p_0$, so $q'_0$ (the largest multiple of $p_0$ such that $q'_0\leq q_0-1$) is positive. For a step $m>1$, notice that there exists an \mbox{$(m-1)$-level} which is not chosen at the previous step (as we have $r_{1,m-1}$ \mbox{$(m-1)$-levels} in $W_{1,m-1}$ and we choose $q'_0t_{1,m-1}$ of them, with $q'_0t_{1,m-1}<r_{1,m-1}$) so its $q_{m-1}$ $m$-levels are in $W_{1,m}$ and this gives $r_{1,m}\geq q_{m-1}$. Therefore we have $r_{1,m}>q'_0$ and $t_{1,m}$ is non-zero.\par
		Now consider a step $n>1$ of the outer recursion. For $m=n$, $B'_{n-2,0}(n-1)$ is an $(n-1)$-level in $W_{n,n}$, so we have $r_{n,n}\geq q_{n-1}>p_{n-1}$, hence the positivity of $q'_{n-1}$. For $m>n$, we have $r_{n,m}\geq q_{m-1}$ (same argument as for $n=1$), this implies $r_{n,m}>q'_{n-1}$ and $t_{n,m}$ is positive.
	\end{proof}
	
	The next lemma refines the estimate given by Lemma~\ref{preq'_n/q_nleq}, with assumptions which will be satisfied in the context of flexible classes.
	
	\begin{lemma}\label{boundusingcondfc}
		Let $(q_n,(\sigma_{n,0},\ldots,\sigma_{n,q_n}))_{n\geq 0}$ be the parameters of a CSP construction of $T$ with associated constant $C>0$. Assume that there exists a constant $C'>0$ such that
		$$\forall n\geq 1,\ \sigma_n\leq C'q_nh_{n-1}\text{ and }q_{n}-(1+p_{n})\geq C'h_{n}$$
		(for instance, if the third point of Definition~\ref{deffc} holds and if, given $(q_k,(\sigma_{0,k},\ldots ,\sigma_{q_k,k}))_{0\leq k\leq n-1}$, $q_n$ is chosen large enough).
		Then we get the following bound:
		$$\forall n\in\N,\ \frac{q'_n}{q_n}\leq 4.$$
	\end{lemma}
	
	\begin{proof}[Proof of Lemma~\ref{boundusingcondfc}]
		For $n=0$, this is a consequence of the inequality $q'_0\leq q_0-1$. Now let us prove the result for $n\geq 1$. Using Lemma~\ref{preq'_n/q_nleq}, it suffices to get
		$$\forall n\geq 1,\ \frac{\sigma_n}{q'_0\ldots q'_{n-1}}\leq q_n.$$
		But we have
		$$\sigma_n\leq C'q_n h_{n-1}\leq q_n\left (q_{n-1}-(1+p_{n-1})\right ),$$
		and the right hand side is bounded above by $q_nq'_{n-1}$ (by Lemma~\ref{asympt}), so the result follows from the inequality $q'_{n-1}\leq q'_0\ldots q'_{n-1}$.
	\end{proof}
	
	\subsection{Equality of the orbits, universal odometer and quantitative control of the cocycles}\label{eqorbit}
	
	Recall the notations for the construction of $T$ by cutting and stacking, $(q_n)_n$ and $(\sigma_{n,i})_{n,i}$ are respectively the cutting and spacing parameters. The tower $\mathcal{R}_n$ is the $n$-th $T$-Rokhlin tower, its height is $h_n$, it covers the subset $X_n$ of $X$, $\varepsilon_n$ is the measure of its complement, $Z_n$ is the maximum of the spacing parameters over the first $n$ steps and $M_0$ is the measure of the unique $0$-level $B_0$.\par
	We use similar notations $q'_n$, $h'_n$ and $\mathcal{R}'_n$ for $S$. We also set
	$$H'_n\coloneq h'_1+\ldots +h'_n$$
	for all $n\geq 1$, and $H'_0\coloneq 0$.\par
	The construction is assumed to be well-defined, considering a cutting-and-stacking definition of $T$ with parameters satisfying the criterion~\eqref{crit1} (see Lemma~\ref{welldefined}). Since $S$ is piecewise given by powers of $T$, the $S$-orbits are included in the $T$-orbits. It remains to show the reverse inclusion, to prove that $(\mathcal{R}'_n)_{n\geq 0}$ is increasing to the $\sigma$-algebra $\A$ and to quantify the cocycles.\par
	As in \cite{kerrEntropyVirtualAbelianness2024}, we set
	$$E_{n,m}\coloneq \bigsqcup_{i=0}^{h'_{n}-1}{S^i}\left (\bigsqcup_{n\leq M\leq m}{B'_{n-1,0}(M)}\right )=\bigsqcup_{\substack{0\leq i_0\leq q'_0-1\\ \vdots\\0\leq i_{n-1}\leq q'_{n-1}-1}}{\sn_{1}^{i_0}\ldots \sn_{n}^{i_{n-1}}}\left (\bigsqcup_{n\leq M\leq m}{B'_{n-1,0}(M)}\right )$$
	and
	$$K_n\coloneq \bigsqcup_{\substack{0< i\leq h_n-1\\T^{i-1}(B_n)\sqcup T^{i}(B_n)\subseteq E_{n,n}}}{T^i(B_n)}.$$
	Since $B'_{n-1,0}$ is exactly the base $B'_n$ of $\mathcal{R}'_n$, the subsets $S^i(B'_{n-1,0})$, for $0\leq i\leq h'_{n}-1$, are exactly the levels of $\mathcal{R}'_n$ which is a partition of $X$. So the motivation behind the definition of $E_{n,m}$ is first to approximate $B'_{n-1,0}$ by its $M$-bricks for $n\leq M\leq m$, and then the set $E_{n,m}$ is actually the union of the $M$-bricks, for $n\leq M\leq m$, of step $n$ of the outer recursion, and their translates by $S$ in the other levels in $\mathcal{R}'_n$ (the sets $E_{1,1}$, $E_{1,2}$, $E_{1,3}$, $E_{2,1}$ and $E_{2,2}$ are illustrated in Figures~\ref{fignequal1} and~\ref{fign}). We get a better approximation of $X$ as $m$ increases and notice that $E_{n,m}$ is a subset of $X_m$ since every $M$-brick, for $n\leq M\leq m$, is a union of $m$-levels. Finally the sets $K_n$, for $n\geq 1$, are introduced in order to show that the system $S$ captures the $T$-orbits (recall Remark~\ref{capture}).
	
	\begin{lemma}\label{Enm}
		The following holds:
		$$\mu\left (X_m\setminus E_{n,m}\right )\leq\left\{\begin{array}{ll}
			\displaystyle\frac{H'_n}{h_m}&\text{for }n<m\\
			&\\
			\displaystyle\frac{H'_{n-1}+p_{n-1}h'_{n-1}}{h_n}&\text{for }n=m
		\end{array}\right..$$
	\end{lemma}
	
	\begin{proof}[Proof of Lemma~\ref{Enm}]
		We prove the inclusions
		$$E_{n,m}\subset E_{n-1,m}\subset\ldots\subset E_{2,m}\subset E_{1,m}\subset X_m$$
		and we bound the measures of $X_m\setminus E_{1,m}$ and each set $E_{k,m}\setminus E_{k-1,m}$. The result follows from the decomposition
		\begin{equation}\label{decompo}
			X_m\setminus E_{n,m}=\left (X_m\setminus E_{1,m}\right )\sqcup\bigsqcup_{2\leq k\leq n}{\left (E_{k-1,m}\setminus E_{k,m}\right )}
		\end{equation}
		and $\sigma$-additivity of $\mu$.\par
		The set $E_{1,m}$ is composed of $m$-levels, so it is contained in $X_m$. If $m=1$, then $X_m\setminus E_{1,m}$ is the disjoint union of $r_{1,1}-q'_0$ $1$-levels (see step $(1,1)$ of the construction). If $m>1$, then $X_m\setminus E_{1,m}$ is the disjoint union of $r_{1,m}-q'_0t_{1,m}$ $m$-levels (see step $(1,m)$ of the construction). By definition of $q'_0$ (if $m=1$) or $t_{1,m}$ (if $m>1$), we thus have
		$$\mu\left (X_m\setminus E_{1,m}\right )\leq\left\{\begin{array}{cl}
			\displaystyle\frac{p_0}{h_m}&\text{if }m=1\\
			&\\
			\displaystyle\frac{h'_1}{h_m}&\text{if }m>1
		\end{array}\right.$$
		(recall that $h'_1=q'_0$).\par    
		Let $k\in\llbracket 2,n\rrbracket$. The function $\sn_k$ has been built in order to map each $M$-brick ($M\geq k$) at step $k$ to another. But such a brick is contained in an $M'$-brick ($k-1\leq M'\leq M$) from the previous step $k-1$ (see Lemma~\ref{relbrick}). We then have
		$$\bigsqcup_{0\leq i_{k-1}\leq q'_{k-1}-1}{\sn_{k}^{i_{k-1}}\left (\bigsqcup_{k\leq M\leq m}{B'_{k-1,0}(M)}\right )}\subseteq \bigsqcup_{k-1\leq M\leq m}{B'_{k-2,0}(M)}.$$
		Applying $\sn_{1}^{i_0}\ldots \sn_{k-1}^{i_{k-2}}$ and considering the union over $i_0,\ldots,i_{k-2}$, we get the inclusion $E_{k,m}\subseteq E_{k-1,m}$ and the equality
		$$\begin{array}{l}
			\displaystyle E_{k-1,m}\setminus E_{k,m}=\\
			\displaystyle \bigsqcup_{\substack{0\leq i_0\leq q'_0-1\\ \vdots\\0\leq i_{k-2}\leq q'_{k-2}-1}}{\sn_{1}^{i_0}\ldots \sn_{k-1}^{i_{k-2}}\left (\underbrace{\left (\bigsqcup_{k-1\leq M\leq m}{B'_{k-2,0}(M)}\right )\setminus\left (\bigsqcup_{k\leq M\leq m}{\bigsqcup_{0\leq i_{k-1}\leq q'_{k-1}-1}{B'_{k-1,i_k}(M)}}\right )}_{=:\ [\ast]}\right )}.
		\end{array}$$
		So the measure of $E_{k-1,m}\setminus E_{k,m}$ is $q'_0\ldots q'_{k-2}\mu\left ([\ast]\right )=h'_{k-1}\mu\left ([\ast]\right )$ by $T$-invariance. The set $[\ast]$ is obtained from $W_{k,m}$ (see~\eqref{Wnn} and~\eqref{Wnm}) by removing the $m$-bricks that have been chosen at step $(k,m)$. If $m=k$, then $[\ast]$ is the disjoint union of $r_{k,k}-q'_{k-1}$ $m$-levels (see step $(k,k)$ of the construction). If $m>k$, then $[\ast]$ is the disjoint union of $r_{k,m}-q'_{k-1}t_{k,m}$ $m$-levels (see step $(k,m)$ of the construction). By definition of $q'_{k-1}$ (if $m=k$) or $t_{k,m}$ (if $m>k$), we thus have
		$$\mu\left ([\ast]\right )\leq\left\{\begin{array}{cl}
			\displaystyle\frac{p_{k-1}}{h_m}&\text{if }m=k\\
			&\\
			\displaystyle\frac{q'_{k-1}}{h_m}&\text{if }m>k
		\end{array}\right.$$
		and
		$$\mu\left (E_{k-1,m}\setminus E_{k,m}\right )\leq\left\{\begin{array}{cl}
			\displaystyle\frac{h'_{k-1}p_{k-1}}{h_m}&\text{if }m=k\\
			&\\
			\displaystyle\frac{h'_{k}}{h_m}&\text{if }m>k
		\end{array}\right..$$
		Using~\eqref{decompo} and $\sigma$-additivity of $\mu$, we get the following inequalities. If $m>n$, we get
		$$\mu\left (X_m\setminus E_{n,m}\right )=\mu\left (X_m\setminus E_{1,m}\right )+\sum_{2\leq k\leq n}{\mu\left (E_{k-1,m}\setminus E_{k,m}\right )}\leq\sum_{1\leq k\leq n}{\frac{h'_k}{h_m}}=\frac{H'_n}{h_m}.$$
		If $m=n$, we get
		$$\begin{array}{lcl}
			\displaystyle\mu\left (X_m\setminus E_{n,m}\right )&=&\displaystyle\left (\mu\left (X_m\setminus E_{1,m}\right )+\sum_{2\leq k\leq m-1}{\mu\left (E_{k-1,m}\setminus E_{k,m}\right )}\right )+\mu\left (E_{n-1,n}\setminus E_{n,n}\right )\\
			&\leq&\displaystyle\sum_{1\leq k\leq m-1}{\frac{h'_k}{h_m}}+\frac{p_{n-1}h'_{n-1}}{h_n}\\
			&=&\displaystyle\frac{H'_{n-1}}{h_n}+\frac{p_{n-1}h'_{n-1}}{h_m}
		\end{array}$$
		and we are done.
	\end{proof}
	
	The quantity $H'_{n-1}+p_{n-1}h'_{n-1}$ only depends on $q'_1,\ldots,q'_{n-2}$ which only depend on $(q_i,(\sigma_{i,j})_{0\leq j\leq q_i})_{0\leq i\leq n-2}$ (see Lemma~\ref{dependence}), and $h_n$ is larger than $q_1\ldots q_{n-1}/M_0$ with $q_{n-1}$ appearing at step $n-1$. Then the strategy will be to recursively choose the cutting parameters $q_{n-1}$ so that
	\begin{equation}\label{crit3}
		\frac{H'_{n-1}+p_{n-1}h'_{n-1}}{h_n}\underset{n\to +\infty}{\to}0.
	\end{equation}
	As $\mu(X_n)\underset{n\to +\infty}{\to}1$, this gives $\mu(E_{n,n})\underset{n\to +\infty}{\to}1$ by Lemma~\ref{Enm}.
	
	\begin{corollary}\label{odouniv}
		If $\mu(E_{n,n})\underset{n\to +\infty}{\to}1$, then $S$ is the universal odometer.
	\end{corollary}
	
	\begin{proof}[Proof of Corollary~\ref{odouniv}]
		By the definition of $q'_n$ at step $(n,n)$ and by choice of the sequence $(p_n)$, every prime number appears infinitely many time as a prime factor among the integers $q'_0, q'_1, q'_2 ,\ldots$. If $S$ is an odometer, then it is clearly universal. It remains to show that $(\mathcal{R}'_n)_{n\in\N}$ increases to the $\sigma$-algebra $\A$. Then $S$ is a rank-one system with zero spacing parameters by Lemma~\ref{partition}, so this is an odometer.\par
		Consider a subsequence $(n_k)_{k\geq 0}$ such that the series $\sum_{k\geq 0}{\mu((E_{n_k,n_k})^c)}$ is convergent. By the Borel-Cantelli lemma, the set $X_0\coloneq \bigcup_{j\geq 0}{\bigcap_{k\geq j}{E_{n_k,n_k}}}$ is of full measure. Let $x,y\in X_0$. Assume that they belong to the same level of $\mathcal{R}'_n$ for every $n$ larger that some threshold $N_0$. The goal is to show that $x$ and $y$ are equal, so that $(\mathcal{R}'_n)_{n\in\N}$ separates the points of a set of full measure and hence it increases to $\A$.\par
		By the definition of $X_0$, there exists an infinite subset $I$ of $\N$, bounded below by $N_0$, such that $E_{n,n}$ contains $x$ and $y$ for every $n\in I$. Let us fix an integer $n\in I$. By the definition of $E_{n,n}$, $x$ is in some $S^i(B'_{n-1,0}(n))$ and $y$ in some $S^j(B'_{n-1,0}(n))$, for $0\leq i,j\leq q'_0\ldots q'_{n-1}-1$. But $x$ and $y$ are in the same level of $\mathcal{R}'_n$, furthermore $S^i(B'_{n-1,0}(n))$ is included in the level $S^i(B'_{n})$ and $S^j(B'_{n-1,0}(n))$ in the level $S^j(B'_{n})$, so we have $i=j$. Moreover, since we have $t_{n,n}=1$, all the sets $S^k(B'_{n-1,0}(n))$ are $n$-levels, i.e.~levels of the $n$-th $T$-Rokhlin tower $\mathcal{R}_n$, so $x$ and $y$ are in the same $n$-level. This holds for every $n\in I$, so for infinitely many $n$. Moreover $(\mathcal{R}_n)_{n\in\N}$ separates the points up to a null set, since $T$ is rank-one, hence the result.
	\end{proof}
	
	\begin{lemma}\label{measureKn}
		For every $n\in\N$, we have
		$$\mu(K_n)\geq \mu(X_n)-\mu\left (B_{n}\right )-2\mu\left (X_n\setminus E_{n,n}\right ).$$
		Moreover, $\mu\left (K_n\right )\underset{n\to +\infty}{\to}1$ if $\mu(E_{n,n})\underset{n\to +\infty}{\to}1$.
	\end{lemma}
	
	\begin{proof}[Proof of Lemma~\ref{measureKn}]
		The set $K_n$ is equal to $\left (E_{n,n}\setminus B_n\right )\ \setminus\ T\left (X_n\setminus E_{n,n}\right )$, so we get
		$$\begin{array}{lcl}
			\mu\left (K_n\right )&\geq&\mu\left (E_{n,n}\setminus B_n\right )-\mu\left (T\left (X_n\setminus E_{n,n}\right )\right )\\
			&\geq&\mu\left (E_{n,n}\right )-\mu\left (B_{n}\right )-\mu\left (X_n\setminus E_{n,n}\right )\\
			&=&\mu(X_n)-\mu\left (B_{n}\right )-2\mu\left (X_n\setminus E_{n,n}\right ).
		\end{array}$$
		The second result follows from the fact that $\mu(X_{n})\underset{n\to +\infty}{\to}1$ and $\mu(B_{n})\underset{n\to +\infty}{\to}0$.
	\end{proof}
	
	\begin{lemma}\label{Kn}
		For every $x\in K_n$, there exists $k\in\Z$ such that
		$$|k|\leq 4(h_{n-1}+Z_{n-1})(h'_{n-1})^2$$
		and $T^{-1}x=S^kx$.
	\end{lemma}
	
	\begin{proof}[Proof of Lemma~\ref{Kn}]
		Let $x\in K_n$. By the definition of $K_n$, the points $x$ and $T^{-1}x$ are in $E_{n,n}$ and there exists $1\leq i\leq h_n-1$ such that $x\in T^i(B_n)$. Writing $E_{n,n}$ this way:
		$$E_{n,n}=\bigsqcup_{\substack{0\leq i\leq h'_{n-1}-1\\0\leq i_{n-1}\leq q'_{n-1}-1}}{S^i\sn_n^{i_{n-1}}\left (B'_{n-1,0}(n)\right )}=\bigsqcup_{\substack{0\leq i\leq h'_{n-1}-1\\0\leq i_{n-1}\leq q'_{n-1}-1}}{S^i\left (B'_{n-1,i_{n-1}}(n)\right )},$$
		it is clear that there exist $0\leq k_0,k_1\leq h'_{n-1}-1$ such that $y\coloneq S^{-k_0}x$ and $z\coloneq S^{-k_1}T^{-1}x$ are in $\bigsqcup_{0\leq i_{n-1}\leq q'_{n-1}-1}{B'_{n-1,i_{n-1}}(n)}$.\par
		We first show that we can write $y=\sn_n^{k_2}z$ for some $k_2$, using the fact that $\sn_n$ connects the $n$-bricks of step $(n,n)$ of the construction (since $t_{n,n}=1$). Secondly $\sn_n$ can be written as a power of $S$ and the equality $y=S^{k_3}z$ holds for some $k_3$ that we will be able to bound by Lemma~\ref{cocycle}. Finally the result follows from the bound for each integer $k_0,k_1,k_3$.
		\paragraph{Step 1: Finding $k_2$ such that $y=\sn_n^{k_2}z$.} Using Lemma~\ref{explicit}, we can write
		$$x=\sn_{1}^{i_0}\ldots \sn_{n-1}^{i_{n-2}}y\text{ and }T^{-1}x=\sn_{1}^{j_0}\ldots \sn_{n-1}^{j_{n-2}}z$$
		for some integers $0\leq i_0,j_0\leq q'_0-1, \ldots, 0\leq i_{n-2},j_{n-2}\leq q'_{n-2}-1$, and there exist $0\leq i_{n-1},j_{n-1}\leq q'_{n-1}-1$ such that
		$$y\in B'_{n-1,i_{n-1}}(n)\text{ and }z\in B'_{n-1,j_{n-1}}(n).$$
		More precisely, by Lemma~\ref{relbrick} and the fact that $y$ and $z$ are in $n$-bricks at step $(n,n)$, we have
		$$x=\sn_{1}(M_1)^{i_0}\ldots \sn_{n-1}(M_{n-1})^{i_{n-2}}y\text{ and }T^{-1}x=\sn_{1}(L_1)^{j_0}\ldots \sn_{n-1}(L_{n-1})^{j_{n-2}}z$$
		with $k\leq L_k,M_k\leq n$ for every $1\leq k\leq n-1$. By construction, $T$ and the maps $\sn_k(m)$, for $1\leq k\leq n-1$ and $k\leq m\leq n$, satisfy the following property: for every $n$-level $T^k(B_n)$, with $0\leq k\leq h_n-1$, contained in the domain of the map, if it is mapped to another $n$-level $T^{k+\ell}(B_n)$, with $0\leq k+\ell\leq h_n-1$, then the application coincides with $T^{\ell}$ on $T^k(B_n)$. In other word it consists in going up or down $|\ell|$ floors in the tower $\mathcal{R}_n$, without going above its roof or below its base. Therefore, from $B'_{n-1,i_{n-1}}(n)$ to $B'_{n-1,j_{n-1}}(n)$, the map
		$$\tilde{S}\coloneq \left (\sn_{1}(L_1)^{j_0}\ldots \sn_{n-1}(L_{n-1})^{j_{n-2}}\right )^{-1}T^{-1}\sn_{1}(M_1)^{i_0}\ldots \sn_{n-1}(M_{n-1})^{i_{n-2}}$$
		consists in successively going up or down in the tower, so this is a power of $T$ given by the difference between the floor of $B'_{n-1,i_{n-1}}(n)$ and the one of $B'_{n-1,j_{n-1}}(n)$. The map $\sn_n^{j_{n-1}-i_{n-1}}$ also satisfies this property, thus $\sn_n^{j_{n-1}-i_{n-1}}$ and $\tilde{S}$ coincide on $B'_{n-1,i_{n-1}}(n)$ and $y=\sn_n^{k_2}z$ with $k_2\coloneq j_{n-1}-i_{n-1}$.
		\paragraph{Step 2: Finding $k_3$ such that $y=S^{k_3}z$.} Using the Lemma~\ref{explicit} and the equality $\sn_n^i(B'_n)=B'_{n-1,i}$, we have $S^{h'_{n-1}\left (j_{n-1}-i_{n-1}\right )}y=z$, we set $k_3\coloneq h'_{n-1}\left (j_{n-1}-i_{n-1}\right )$ and it remains to find a bound for $j_{n-1}-i_{n-1}$. We need to get more information on the power of $T$, denoted by $T^{\ell}$, which coincides with $\tilde{S}$ on $B'_{n-1,i_{n-1}}(n)$. By Lemma~\ref{cocycle} and the definition of $\tilde{S}$, we get
		$$\begin{array}{lcl}
			|\ell |&\leq &(h_{n-1}+Z_{n-1})(i_0+\ldots +i_{n-2}) +1 +(h_{n-1}+Z_{n-1})(j_0+\ldots +j_{n-2})\\
			&\leq &2(h_{n-1}+Z_{n-1})(q'_0+\ldots +q'_{n-2})+1\\
			&\leq &3(h_{n-1}+Z_{n-1})(q'_0+\ldots +q'_{n-2})
		\end{array}$$
		where "$+1$" comes from "$T^{-1}$" in the expression of $\tilde{S}$ and has been bounded by $(h_{n-1}+Z_{n-1})(q'_0+\ldots +q'_{n-2})$. The sum $q'_0+\ldots +q'_{n-2}$ is less than the product $q'_0\ldots q'_{n-2}=h'_{n-1}$, this gives
		$$|\ell|\leq 3(h_{n-1}+Z_{n-1})h'_{n-1}.$$
		Since $\sn_n$ has a positive cocycle (by Lemma~\ref{cocycle}), the equality $\sn_n^{\left (j_{n-1}-i_{n-1}\right )}=T^{\ell}$ implies $|\ell|\geq |j_{n-1}-i_{n-1}|$. Therefore we find the bound
		$$|k_3|\leq 3(h_{n-1}+Z_{n-1})(h'_{n-1})^2.$$
		\paragraph{Step 3: Bounding the integer $k$ such that $T^{-1}x=S^{k}x$.} By the definition of $k_0$, $k_1$ and $k_3$, $T^{-1}x$ is equal to $S^kx$ with $k\coloneq k_1-k_3-k_0$ which is thus bounded as follows:
		$$\begin{array}{lcl}
			\displaystyle |k|&\leq &\displaystyle |k_0|+|k_1|+|k_3|\\
			&\leq &\displaystyle 2(h'_{n-1}-1)+3(h_{n-1}+Z_{n-1})(h'_{n-1})^2\\
			&\leq &\displaystyle 4(h_{n-1}+Z_{n-1})(h'_{n-1})^2,
		\end{array}$$
		hence the result.
	\end{proof}
	
	\begin{corollary}\label{orbiteq}
		If $\mu(E_{n,n})\underset{n\to +\infty}{\to}1$, then $T$ and $S$ have the same orbits.
	\end{corollary}
	
	\begin{proof}[Proof of Corollary~\ref{orbiteq}]
		It is clear that the $S$-orbits are contained in the $T$-orbits. By Lemma~\ref{measureKn}, $\bigcup_{n\in\N}{K_n}$ is of full measure, so the reverse inclusion follows from Lemma~\ref{Kn}.
	\end{proof}
	
	\begin{remark}\label{oe}
		Corollary~\ref{orbiteq} holds for every rank-one system $T$. Indeed skipping steps in the cutting-and-stacking process of $T$ recursively increases the cutting parameters $q_n$, it enables us to get criteria~\eqref{crit1} and \eqref{crit3} (the first one implies that the construction in Section~\ref{theconstruction} is well-defined, the second one that $\mu(E_{n,n})\to 1$).\par
		However the quantification of the cocycles will not necessarily hold for all the rank-one systems, since we will need to control the quantities $Z_n$ depending on the spacing parameters (see Section~\ref{proofth}).\par
		Note that by Dye's theorem, it was already known that every rank-one system is orbit equivalent to the universal odometer, but the proof of this theorem does not provide an explicit orbit equivalence, thus preventing us from quantifying the cocycles.
	\end{remark}
	Now the goal is to control the cocycle $c_S$. The equalities~\eqref{explicit2} in Section~\ref{firstprop} and the decomposition of $B_{n-1,i}$ in bricks motivate the following definition:
	\begin{equation}\label{exprDnm}
		\begin{array}{lll}
			\forall m\geq n\geq 1, &D_n(m)&\coloneq \sn_1^{q'_0-1}\ldots \sn_{n-1}^{q'_{n-2}-1}\left (\bigsqcup_{0\leq i_n\leq q'_{n-1}-2}{B'_{n-1,i_n}(m)}\right )\\
			&&=\sn_1^{q'_0-1}\ldots \sn_{n-1}^{q'_{n-2}-1}\left (\bigsqcup_{0\leq i_n\leq q'_{n-1}-2}{\sn_n^{i_n}(m)(B'_{n-1,0})}\right ).
		\end{array}
	\end{equation}
	It is the union of all the translates of the $m$-bricks at step $(n,m)$ composing $D_n$. Note that $S$ coincides with $\sn_n(m)\sn_{n-1}^{-(q'_{n-2}-1)}\ldots \sn_1^{-(q'_0-1)}$ on $D_n(m)$ (since it coincides with $\sn_n\sn_{n-1}^{-(q'_{n-2}-1)}\ldots \sn_1^{-(q'_0-1)}$ and $\sn_n$ coincides with $\sn_n(m)$ on the $m$-bricks at step $n$). The partition of $D_n$ into such subsets $D_n(m)$, for $m\geq n$, gives a fine control of the cocycle $c_S$.
	
	\begin{lemma}\label{Dnm}
		For $1\leq n<m$, $D_{n}(m)$ is contained in $X_m\setminus E_{n,m-1}$ and we have
		$$\mu(D_{n}(m))\leq \left\{\begin{array}{ll}
			\displaystyle\varepsilon_{m-1}-\varepsilon_m+\frac{H'_n}{h_{m-1}}&\text{if }m>n+1\\
			\displaystyle\varepsilon_{m-1}-\varepsilon_{m}+\frac{H'_{n-1}+p_{n-1}h'_{n-1}}{h_{n}}&\text{if }m=n+1
		\end{array}\right..$$
		For all $n\geq 1$, we have
		$$\mu(D_{n}(n))\leq\frac{q'_{n-1}}{h_n}.$$
		Moreover for every $x\in D_{n}(m)$,
		$$|c_S(x)|\leq (h_{m-1}+Z_{m-1})h'_{n-1}.$$
	\end{lemma}
	
	\begin{proof}[Proof of Lemma~\ref{Dnm}]
		For $1\leq n <m$, $D_{n}(m)$ is composed of translates of the $m$-bricks used at step $(n,m)$, so it is disjoint from the translates of the $M$-bricks used at step $(n,M)$ for $n\leq M\leq m-1$, hence the inclusion $D_{n,m}\subseteq X_m\setminus E_{n,m-1}$. The bound for $\mu(D_{n}(m))$ follows from the decomposition $X_m\setminus E_{n,m-1}=(X_m\setminus X_{m-1})\sqcup (X_{m-1}\setminus E_{n,m-1})$ and Lemma~\ref{Enm}.\par
		For $n\geq 1$, by the definition of $D_n(n)$ and the $\sn_i$-invariance of the measure, we get
		$$\mu(D_{n}(n))=(q'_{n-1}-1)\mu\left (B'_{n-1,0}(n)\right )\leq q'_{n-1}\mu\left (B'_{n-1,0}(n)\right ),$$
		hence the result, since $B'_{n-1,0}(n)$ is an $n$-level, so it has measure less than $1/h_n$.\par
		For the cocycle $c_S$, we first decompose $D_n(m)$ in the following way:
		$$D_n(m)=\bigsqcup_{\ell}{\underbrace{\sn_1^{q'_0-1}\ldots \sn_{n-1}^{q'_{n-2}-1}(\beta_{\ell}))}_{=:D^{\ell}}}$$
		where $(\beta_{\ell})_{\ell}$ is the family of $m$-bricks, at step $(n,m)$, which constitute the subset $\bigsqcup_{0\leq i_n\leq q'_{n-1}-2}{B'_{n-1,i_n}(m)}$. For a fixed $\ell$, by Lemma~\ref{relbrick} there exist $1\leq L_1\leq m,\ldots ,$ $n-1\leq L_{n-1}\leq m$ such that
		$$D^{\ell}=\sn_1^{q'_0-1}(L_1)\ldots \sn_{n-1}^{q'_{n-2}-1}(L_{n-1})\left (\beta_{\ell}\right )$$
		and, on this subset, $S$ coincides with $\sn_n(m)\sn_{n-1}^{-(q'_{n-2}-1)}(L_{n-1})\ldots \sn_1^{-(q'_0-1)}(L_1)$. Then using Lemma~\ref{cocycle}, we get
		$$\begin{array}{lcl}
			|(c_S)_{|D^{\ell}}|&\leq &(h_{m-1}+Z_{m-1})((q'_{0}-1)+\ldots +(q'_{n-2}-1)+1)\\
			&\leq &(h_{m-1}+Z_{m-1})q'_{0}\ldots q'_{n-2}\\
			&=&(h_{m-1}+Z_{m-1})h'_{n-1},
		\end{array}$$
		hence the result.
	\end{proof}
	
	\subsection{Proof of Theorem~\ref{thfc}}\label{proofth}
	
	Let $T$ be a rank-one system whose parameters satisfy the criteria~\eqref{crit1} and \eqref{crit3}. The first one ensures that the construction is well-defined (Lemma~\ref{welldefined}), the second one implies $\mu(E_{n,n})\to 1$ (Lemma~\ref{Enm}), so we have an orbit equivalence between $T$ and $S$ (Lemma~\ref{orbiteq}). We can then define the cocycles $c_T,c_S\colon X\to\Z$ by
	$$\forall x\in X,\ Tx=S^{c_T(x)}x\text{ and }Sx=T^{c_S(x)}x.$$
	In Lemmas~\ref{Kn} and~\ref{Dnm}, we obtained bounds for the cocycles on precise subsets covering $X$: $(K_n)_n$ for $c_T$, $(D_n(m))_{n,m}$ for $c_S$. This will provide a bound for the $\varphi$-integral of each cocycle. But first, we need to change $\varphi$ via the following lemma inspired by Lemma~2.12 in \cite{carderiBelinskayaTheoremOptimal2023}. Without loss of generality, $\varphi$ has the properties given by the lemma and this will simplify the bound for each $\varphi$-integral.
	
	\begin{lemma}\label{metriccompatible}
		Let $0<\alpha\leq 1$ and $\varphi\colon\R_+\to\R_+$ satisfying $\varphi(t)=o(t^{\alpha})$. Then there exists $\Phi\colon\R_+\to\R_+$ with the following properties:
		\begin{itemize}
			\item $\Phi$ is increasing;
			\item $\Phi$ is subadditive: $\forall t,s\in\R_+,\ \Phi(t+s)\leq\Phi(t)+\Phi(s)$;
			\item $\Phi(t)=o\left (t^{\alpha}\right )$;
			\item $\varphi(t)=O\left (\Phi(t)\right )$.
		\end{itemize}
	\end{lemma}
	
	\begin{proof}[Proof of Lemma~\ref{metriccompatible}]
		Set
		$$\begin{array}{llccl}
			\displaystyle \theta&\colon &\displaystyle \R_+^*&\displaystyle \to&\displaystyle \R_+\\
			&&\displaystyle t&\displaystyle \mapsto &\displaystyle \min{\left (1,\ \sup_{s\geq t}{\frac{\varphi(s)+1}{s}}\right )}
		\end{array}$$
		and
		$$\begin{array}{llccl}
			\displaystyle \Phi&\colon&\displaystyle \R_+&\displaystyle \to&\displaystyle \R_+\\
			&&\displaystyle t&\displaystyle \mapsto &\displaystyle \int_0^t{\theta(s)\mathrm{d}s}
		\end{array}.$$
		The map $\theta$ is positive-valued and non-increasing, so $\Phi$ is an increasing and subadditive function satisfying $\Phi(t)\geq t\theta(t)$ for every $t\in\R_+$. The assumption $\varphi(t)=o(t^{\alpha})$ implies that \mbox{$\theta(t)=\sup_{s\geq t}{\frac{\varphi(s)+1}{s}}$} for $t>0$ large enough, so we have
		$$\Phi(t)\geq t\theta(t)\geq t\sup_{s\geq t}{\frac{\varphi(s)+1}{s}}\geq\varphi(t).$$
		Finally, for a fixed $\varepsilon>0$, there exists $t_0>0$ such that $\varphi(s)\leq \varepsilon s^{\alpha}$ for every $s\geq t_0$. For every $t\geq t_0$, this gives 
		$$\sup_{s\geq t}{\frac{\varphi(s)+1}{s}}\leq\sup_{s\geq t}{\left (\frac{\varepsilon}{s^{1-\alpha}}+\frac{1}{s}\right )}=\frac{\varepsilon}{t^{1-\alpha}}+\frac{1}{t}$$
		and for every $t\geq t_0$, we have $$\int_{t_0}^t{\theta(s)\mathrm{d}s}\leq\int_{t_0}^t{\left (\frac{\varepsilon}{s^{1-\alpha}}+\frac{1}{s}\right )\mathrm{d}s}=\frac{\varepsilon}{\alpha}t^{\alpha}+\ln{t} - \frac{\varepsilon}{\alpha}t_0^{\alpha}-\ln{t_0},$$
		hence $\Phi(t)=o(t^\alpha)$.
	\end{proof}
	
	\begin{lemma}\label{quantcT}
		Assume that criteria~\eqref{crit1} (in Lemma~\ref{welldefined}) and \eqref{crit3} (after Lemma~\ref{Enm}) are satisfied. Let $\varphi\colon\R_+\to\R_+$ be an increasing and subadditive map. Then, setting
		$$\begin{array}{l}
			\displaystyle\Delta(n)\coloneq \left (1+2(H'_{n}+p_{n}h'_{n})\right )(h'_{n})^2\left (\frac{\varphi(h_{n+1}^3)}{h_{n+1}}+\frac{\varphi(Z_{n+1}h_{n+1}^2)}{h_{n+1}}\right );\\
			\displaystyle\Delta_{\varepsilon}(n)\coloneq \varepsilon_{n+1}(h'_{n})^2\left (\varphi(h_{n+1}^3)+\varphi(Z_{n+1}h_{n+1}^2)\right ),
		\end{array}$$
		we have the following bound:
		\begin{equation}\label{boundint1}
			\int_{X}{\varphi(|c_T(x)|)\mathrm{d}\mu}\leq\varphi(4(h_0+Z_0)(h'_0)^2)+4\sum_{n=0}^{+\infty}{\Delta(n)}+4\sum_{n=0}^{+\infty}{\Delta_{\varepsilon}(n)}.
		\end{equation}
	\end{lemma}
	
	\begin{proof}[Proof of Lemma~\ref{quantcT}]
		Motivated by Lemma~\ref{Kn}, we will rather quantify the cocycle $c_{T^{-1}}$ defined on $X$ (up to a null set) by
		$$T^{-1}x=S^{c_{T^{-1}}(x)}x.$$
		It is equivalent to quantifying $c_T$ since we have
		$$\forall x\in X,\ c_{T^{-1}}(x)=-c_T(T^{-1}x)$$
		and $\mu$ is $T$-invariant.\par
		Let $(K'_n)_{n>0}$ be the partition of $X$ inductively defined by
		$$\left\{\begin{array}{l}
			K'_1\coloneq K_1,\\
			\forall n>0,\ K'_{n+1}\coloneq K_{n+1}\setminus (K_1\cup\ldots\cup K_n).
		\end{array}\right.$$
		
		The subsets $K'_n$ are pairwise disjoint and cover the whole space since we have
		$$K'_1\cup\ldots\cup K'_n=K_1\cup\ldots\cup K_n$$
		and \mbox{$\mu(K_{n})\to 1$} (using Lemma~\ref{measureKn}). By the fact that $K_n$ is included in $X_n$, and by Lemmas~\ref{measureKn} and~\ref{Enm}, we have
		$$\begin{array}{lcl}
			\displaystyle\mu(K'_{n+1})&\leq&\displaystyle\mu(X\setminus K_n)\\
			&=&\displaystyle\mu(X\setminus X_n)+\mu(X_n\setminus K_n)\\
			&\leq&\displaystyle\varepsilon_n+\mu\left (B_{n}\right )+2\mu\left (X_n\setminus E_{n,n}\right )\\
			&\leq&\displaystyle\varepsilon_n+\frac{1+2(H'_{n-1}+p_{n-1}h'_{n-1})}{h_n}.
		\end{array}$$
		Since $K'_{n+1}$ is contained in $K_{n+1}$, Lemma~\ref{Kn} implies
		$$\forall x\in K'_{n+1},\ |c_{T^{-1}}(x)|\leq 4(h_n +Z_n)(h'_n)^2.$$
		We then get
		$$\begin{array}{lcl}
			\displaystyle \int_{X}{\varphi(|c_T(x)|)\mathrm{d}\mu}&=&\displaystyle \int_{X}{\varphi(|c_{T^{-1}}(x)|)\mathrm{d}\mu}
			\\&=&\displaystyle \sum_{n=0}^{+\infty}{\int_{K'_{n+1}}{\varphi(|c_{T^{-1}}(x)|)\mathrm{d}\mu}}\\
			&\leq &\displaystyle \sum_{n=0}^{+\infty}{\mu(K'_{n+1})\varphi(4(h_n+Z_n)(h'_n)^2)}\\
			&\leq &\displaystyle \varphi(4(h_0+Z_0)(h'_0)^2)\\
			&&\displaystyle +\sum_{n=1}^{+\infty}{\left (\varepsilon_n+\frac{1+2(H'_{n-1}+p_{n-1}h'_{n-1})}{h_n}\right )\varphi(4(h_n +Z_n)(h'_n)^2)}.
		\end{array}$$
		
		Now we use the assumptions on $\varphi$ to simplify the previous bound. We have $h'_n=h'_{n-1}q'_{n-1}\leq h'_{n-1}h_n$ (by construction we have $q'_{n-1}\leq r_{n,n}\leq h_n$). By monotonicity and subadditivity, this yields
		$$\begin{array}{lcl}
			\displaystyle \varphi(4(h_n+Z_n)(h'_n)^2)&\leq &\displaystyle \varphi(4(h_n+Z_n)(h'_{n-1}h_n)^2)\\
			&\leq &\displaystyle 4(h'_{n-1})^2\left (\varphi(h_n^3)+\varphi(Z_nh_n^2)\right )
		\end{array}$$
		and we get the bound~\eqref{boundint1}.
	\end{proof}
	
	\begin{lemma}\label{quantcS}
		Assume that criteria~\eqref{crit1} (in Lemma~\ref{welldefined}) and \eqref{crit3} (after Lemma~\ref{Enm}) are satisfied and that the following holds:
		$$\forall n\geq 0,\ \frac{q'_n}{q_n}\leq 4$$
		(this is an assumption that we will be able to get by Lemma~\ref{boundusingcondfc}, using flexible classes). Let $\varphi\colon\R_+\to\R_+$ be an increasing and subadditive map. Then, setting
		$$\begin{array}{l}
			\displaystyle\Gamma_1(n)\coloneq \displaystyle 4h'_{n}\left (\frac{\varphi(h_{n+1}^2)}{h_{n+1}}+\frac{\varphi(Z_{n+1}h_{n+1})}{h_{n+1}}\right );\\
			\displaystyle\Gamma_2(n)\coloneq \displaystyle\left (H'_{n}+p_{n}h'_{n}\right )h'_{n}\left (\frac{\varphi(h_{n+1})}{h_{n+1}}+\frac{\varphi(Z_{n+1})}{h_{n+1}}\right );\\
			\displaystyle\Gamma_3(n,m)\coloneq \displaystyle H'_nh'_{n-1}\left (\frac{\varphi(h_{m})}{h_{m}}+\frac{\varphi(Z_{m})}{h_{m}}\right );\\
			\displaystyle\Gamma_{\varepsilon}(n,m)\coloneq \displaystyle \varepsilon_{m} h'_{n}(\varphi(h_{m})+\varphi(Z_{m})),\\
		\end{array}$$
		we have the following bound:
		
		\begin{equation}\label{boundint2}
			\begin{array}{lcl}
				\displaystyle \int_X{\varphi(|c_S|)\mathrm{d}\mu}
				&\leq&\displaystyle \mu(D_1(1))\varphi((h_0+Z_0)h'_0)\\
				&&\displaystyle +\sum_{n\geq 0}{\Gamma_1(n)}+\sum_{n\geq 0}{\Gamma_2(n)}+\sum_{n\geq 1}{\sum_{m\geq n+1}{\Gamma_3(n,m)}}\\
				&&\displaystyle +\sum_{n\geq 0}{\sum_{m\geq n+1}{\Gamma_{\varepsilon}(n,m)}}
			\end{array}.
		\end{equation}
	\end{lemma}
	
	\begin{proof}[Proof of Lemma~\ref{quantcS}]
		By Lemma~\ref{Dnm}, for each subset $D_n(m)$, we found a bound for the cocycle $c_S$ on it, we then get
		
		$$\begin{array}{lcl}
			
			\displaystyle \int_X{\varphi(|c_S|)\mathrm{d}\mu}&=&\displaystyle \sum_{m\geq n\geq 1}{\int_{D_{n}(m)}{\varphi(|c_S|)\mathrm{d}\mu}}\\
			&\leq &\displaystyle \sum_{m\geq n\geq 1}{\mu(D_{n}(m))\varphi((h_{m-1}+Z_{m-1})h'_{n-1})}\\
			&\leq
			&\displaystyle \mu(D_1(1))\varphi((h_0+Z_0)h'_0)\\
			&&\displaystyle +\sum_{n\geq 2}{\gamma_1(n)}+\sum_{n\geq 1}{\gamma_2(n)}+\sum_{n\geq 1}{\sum_{m\geq n+2}{\gamma_3(n,m)}}
		\end{array}$$
		where
		$$\begin{array}{l}
			\displaystyle\gamma_1(n)\coloneq \mu(D_{n}(n))\varphi((h_{n-1}+Z_{n-1})h'_{n-1}),\\
			\\
			\displaystyle\gamma_2(n)\coloneq \mu(D_{n}(n+1))\varphi((h_{n}+Z_n)h'_{n-1}),\\
			\\
			\displaystyle\gamma_3(n,m)\coloneq \mu(D_{n}(m))\varphi((h_{m-1}+Z_{m-1})h'_{n-1}).
		\end{array}$$
		Lemma~\ref{Dnm} also yields a bound for the measure of each set $D_n(m)$, this implies:
		$$\begin{array}{l}
			\displaystyle\gamma_1(n)\leq\frac{q'_{n-1}}{h_n}\varphi((h_{n-1}+Z_{n-1})h'_{n-1}),\\
			\\
			\displaystyle\gamma_2(n)\leq\left (\varepsilon_{n}+\frac{H'_{n-1}+p_{n-1}h'_{n-1}}{h_{n}}\right )\varphi((h_{n}+Z_n)h'_{n-1}),\\
			\\
			\displaystyle\gamma_3(n,m)\leq\left (\varepsilon_{m-1}+\frac{H'_n}{h_{m-1}}\right )\varphi((h_{m-1}+Z_{m-1})h'_{n-1}).
		\end{array}$$
		For all $n\geq 2$, note that we have
		$$\begin{array}{lcl}
			\varphi((h_{n-1}+Z_{n-1})h'_{n-1})&\leq &\varphi((h_{n-1}+Z_{n-1})h'_{n-2}h_{n-1})\\
			&\leq&h'_{n-2}\left (\varphi(h_{n-1}^2)+\varphi(Z_{n-1}h_{n-1})\right )
		\end{array}$$
		and
		$$\frac{q'_{n-1}}{h_n}\leq \frac{q'_{n-1}}{h_{n-1}q_{n-1}}\leq\frac{4}{h_{n-1}},$$
		so wet get
		$$\gamma_1(n)\leq 4h'_{n-2}\left (\frac{\varphi(h_{n-1}^2)}{h_{n-1}}+\frac{\varphi(Z_{n-1}h_{n-1})}{h_{n-1}}\right )=\Gamma_1(n-2).$$
		For $\gamma_2(n)$ and $\gamma_3(n,m)$, note that we have
		$$\forall n\geq 1,\forall m\geq n+1,\ \varphi((h_{m-1}+Z_{m-1})h'_{n-1})\leq h'_{n-1}(\varphi(h_{m-1})+\varphi(Z_{m-1})),$$
		so we get
		$$\begin{array}{lcl}
			\displaystyle \gamma_2(n)&\leq&\displaystyle \left (\varepsilon_{n}+\frac{H'_{n-1}+p_{n-1}h'_{n-1}}{h_{n}}\right ) h'_{n-1}(\varphi(h_{n})+\varphi(Z_{n}))\\
			&=&\displaystyle \varepsilon_{n} h'_{n-1}(\varphi(h_{n})+\varphi(Z_{n}))+\left (H'_{n-1}+p_{n-1}h'_{n-1}\right ) h'_{n-1}\left (\frac{\varphi(h_{n})}{h_{n}}+\frac{\varphi(Z_{n})}{h_{n}}\right )\\
			&=&\Gamma_{\varepsilon}(n-1,n)+\Gamma_2(n-1)
		\end{array}$$
		and
		$$\begin{array}{lcl}
			\displaystyle \gamma_3(n,m)&\leq&\displaystyle \left (\varepsilon_{m-1}+\frac{H'_n}{h_{m-1}}\right ) h'_{n-1}(\varphi(h_{m-1})+\varphi(Z_{m-1}))\\
			&=&\displaystyle \varepsilon_{m-1} h'_{n-1}(\varphi(h_{m-1})+\varphi(Z_{m-1}))+H'_nh'_{n-1}\left (\frac{\varphi(h_{m-1})}{h_{m-1}}+\frac{\varphi(Z_{m-1})}{h_{m-1}}\right )\\
			&=&\Gamma_{\varepsilon}(n-1,m-1)+\Gamma_3(n,m-1)
		\end{array}.$$
		The bound~\eqref{boundint2} now follows immediately.
	\end{proof}

	\begin{proof}[Proof of Theorem~\ref{thfc}]
		Let $\mathcal{C}$ be a flexible class and $\varphi\colon\R_+\to\R_+$ a map satisfying \mbox{$\varphi(t)\underset{t\to +\infty}{=}o\left (t^{1/3}\right )$}. If $\Phi\colon\R_+\to\R_+$ is another map satisfying $\varphi(t)=O\left (\Phi(t)\right )$, then $\Phi$-integrability implies $\varphi$-integrability. Therefore, without loss of generality, we assume that $\varphi$ satisfies the assumptions of Lemma~\ref{metriccompatible}, i.e.~$\varphi$ is increasing and subadditive.\par
		Using the definition of a flexible class, we will build $T$ with large enough and inductively chosen cutting parameters $q_n$. Let $\fc$ be an associated set of parameters, and fix the associated constants $C$ and $C'$ given in Definition~\ref{deffc}. First choose any cutting and spacing parameter $(q_0,(\sigma_{0,0},\ldots,\sigma_{0,q_0}))$ in $\fc$ such that $q_0\geq 3$. Without loss of generality, we assume $p_0=2$ and we get $q_0>p_0$, as required in the assumption of Lemma~\ref{welldefined} for $n=0$. For a fixed $n\geq 1$, assume that $(q_k,(\sigma_{k,0},\ldots,\sigma_{k,q_k}))_{0\leq k\leq n-1}$ has already been determined in $\fc$, this immediately gives $q'_0,\ldots,q'_{n-1}$ (see Lemma~\ref{dependence}). The goal is to find the next parameters with $q_n$ large enough. Consider $\kappa_n>0$ such that for every $t\geq h_{n}\kappa_n$ the following hold:
		\begin{equation}\label{critpr1}
			\kappa_n>\max{(p_n,q'_0,\ldots,q'_{n-1})};
		\end{equation}
		\begin{equation}\label{critpr2}
			\frac{H'_{n}+p_{n}h'_{n}}{t}\leq\frac{1}{n}.
		\end{equation}
		
		The assumption $\varphi(t)=o(t^{1/3})$ also implies the following inequations for a large enough $\kappa_n$:
		
		\begin{equation}\label{critpr4}
			\left(1+2(H'_{n}+p_{n}h'_{n})\right )(h'_{n})^2\left (\frac{\varphi(t^3)}{t}+\frac{\varphi(Ct^3)}{t}\right )\leq\frac{1}{2^n};
		\end{equation}
		\begin{equation}\label{critpr4bis}
			(h'_{n})^2\left (\varphi(t^3)+\varphi(Ct^3)\right )\leq\frac{t}{2^nq_0\ldots q_{n-1}};
		\end{equation}
		\begin{equation}\label{critpr5}
			4h'_{n}\left (\frac{\varphi(t^2)}{t}+\frac{\varphi(Ct^2)}{t}\right )\leq\frac{1}{2^n};
		\end{equation}
		\begin{equation}\label{critpr6}
			\left (H'_{n}+p_{n}h'_{n}\right )h'_{n}\left (\frac{\varphi(t)}{t}+\frac{\varphi(Ct)}{t}\right )\leq\frac{1}{2^{n+1}};
		\end{equation}
		\begin{equation}\label{critpr7}
			\forall 1\leq\ell\leq n,\ H'_{\ell}h'_{\ell-1}\left (\frac{\varphi(t)}{t}+\frac{\varphi(Ct)}{t}\right )\leq\frac{1}{2^{n+2}};
		\end{equation}
		\begin{equation}\label{critpr8}
			\forall 0\leq\ell\leq n,\ h'_{\ell}\left (\varphi(t)+\varphi(Ct)\right )\leq\frac{t}{2^nq_0\ldots q_{n-1}},
		\end{equation}
		for every $t\geq h_n\kappa_n$. With Inequations~\eqref{critpr4}, ~\eqref{critpr4bis}, ~\eqref{critpr5}, ~\eqref{critpr6}, ~\eqref{critpr7} and~\eqref{critpr8}, we will respectively find bounds for the quantities $\Delta(n)$, $\Delta_{\varepsilon}(n)$, $\Gamma_1(n)$, $\Gamma_2(n)$, $\Gamma_3(n,m)$ and $\Gamma_{\varepsilon}(n,m)$ (see Lemmas~\ref{quantcT} and~\ref{quantcS}).\par
		We then set a new cutting parameter $q_n\geq \kappa_n$ large enough with associated spacing parameters $\sigma_{n,0},\ldots ,\sigma_{n,q_n}$ so that $(q_k,(\sigma_{k,0},\ldots ,\sigma_{k,q_k}))_{0\leq k\leq n}\in\fc$, $\sigma_n\leq C'q_nh_{n-1}$ and the following additional assumptions are satisfied:
		\begin{equation}\label{critpr3}
			q_n-(1+p_n)\geq C'h_n
		\end{equation}
		and
		\begin{equation}\label{critpr10}
			\forall 0\leq k\leq n-1,\ q_n\geq C'2^{n-k}q_k.
		\end{equation}
		
		Let $(h_n)$, $(\sigma_n)$ and $(Z_n)$ be the sequences associated to $\bm{p}\coloneq (q_n,(\sigma_{n,0},\ldots,\sigma_{n,q_n}))_{n\geq 0}\in\mathcal{P}^{\N}$ (as described in Definition~\ref{defparam}), $(h'_n)$ the height sequence of the cutting sequence $(q'_n)_{n\geq 0}$ for the universal odometer that we build. We first check that the underlying system is finite measure-preserving, i.e.~the condition~\eqref{finite} in Definition~\ref{defr1} is satisfied. But we have
		$$\frac{\sigma_n}{h_{n+1}}\leq\frac{C'q_nh_{n-1}}{q_nq_{n-1}h_{n-1}}=\frac{C'}{q_{n-1}},$$
		so the summability easily follows from Inequality~\eqref{critpr10}. The underlying system preserves a probability measure, so it is rank-one. Moreover it belongs to $\mathcal{C}$ by the definition of a flexible class.\par
		Inequality~\eqref{critpr1} ensures that the criterion~\eqref{crit1} holds and that the construction in Section~\ref{theconstruction} is well-defined (see Lemma~\ref{welldefined}). Using $h_{n+1}\geq h_nq_n$, the limit in \eqref{crit3} is a consequence of Inequality~\eqref{critpr2} and implies $\mu(E_{n,n})\to 1$. Inequality~\eqref{critpr3} implies
		$$\forall n\in\N,\ \frac{q'_n}{q_n}\leq 4$$
		(see Lemma~\ref{boundusingcondfc}).\par
		Then Lemmas~\ref{quantcT} and~\ref{quantcS} imply that the bounds \eqref{boundint1} for the $\varphi$-integral of $c_T$ and \eqref{boundint2} for the $\varphi$-integral of $c_S$ hold. It remains to prove that these bounds are finite, namely that the series
		$$\sum_{n\geq 0}{\Delta(n)},\ \sum_{n\geq 0}{\Delta_{\varepsilon}(n)},\ \sum_{n\geq 0}{\Gamma_1(n)},\ \sum_{n\geq 0}{\Gamma_2(n)},\ \sum_{n\geq 1}{\sum_{m\geq n+1}{\Gamma_3(n,m)}}\text{ and } \sum_{n\geq 0}{\sum_{m\geq n+1}{\Gamma_{\varepsilon}(n,m)}}$$
		converge.\par
		Using the monotonicity of $\varphi$ and the inequalities $Z_{n+1}\leq Ch_{n+1}$ and~\eqref{critpr4} for $t=h_{n+1}$ (which is greater or equal to $h_n\kappa_n$), we get $\Delta(n)\leq\frac{1}{2^n}$, so the series $\sum_{n\geq 0}{\Delta(n)}$ converges. It is also straightforward to see that the series $\sum_{n\geq 0}{\Gamma_1(n)}$ and $\sum_{n\geq 0}{\Gamma_2(n)}$ are convergent, using Inequalities~\eqref{critpr5} and~\eqref{critpr6}. Inequality~\eqref{critpr7} implies $\Gamma_{3}(n,m)\leq\frac{1}{2^{m+1}}$, so we get
		$$\sum_{m\geq n+1}{\Gamma_3(n,m)}\leq\frac{1}{2^{n+1}}$$
		for every $n\geq 0$, and the series $\sum_{n\geq 1}{\sum_{m\geq n+1}{\Gamma_3(n,m)}}$ converges.\par
		For the other series $\sum_{n\geq 0}{\Delta_{\varepsilon}(n)}$ and $\sum_{n\geq 0}{\sum_{m\geq n+1}{\Gamma_{\varepsilon}(n,m)}}$, we have to control the sequence $(\varepsilon_n)$ (recall that $\varepsilon_n\coloneq \mu((X_n)^c)$). Denote by $M_0$ the measure of $B_0$ (the unique level of the $T$-Rokhlin tower $\mathcal{R}_0$). For every $n\geq 1$, we have
		$$\varepsilon_n=\sum_{k\geq n}{\frac{M_0}{q_0\ldots q_k}\sigma_k}\leq \sum_{k\geq n}{\frac{M_0C'h_{k-1}}{q_0\ldots q_{k-1}}}\leq\sum_{k\geq n}{\frac{C'}{q_{k-1}}}\leq \frac{1}{q_{n-1}}\sum_{k\geq n}{\frac{1}{2^{k-n}}}\leq\frac{2}{q_{n-1}},$$
		using Lemma~\ref{etagerg1bis} and Inequation~\eqref{critpr10}.\par
		Given $n\geq 0$, Inequation~\eqref{critpr8} and Lemma~\ref{etagerg1bis} imply
		$$(h'_{n})^2(\varphi(h_{n+1}^3)+\varphi(Z_{n+1}h_{n+1}^2))\leq\frac{h_{n+1}}{2^{n}q_0\ldots q_{n-1}}\leq\frac{q_{n}}{2^{n}M_0}.$$
		Combining this with the inequality $\varepsilon_{n+1}\leq 2/q_{n}$, we then get
		$$\Delta_{\varepsilon}(n)=\varepsilon_{n+1}(h'_{n})^2(\varphi(h_{n+1}^3)+\varphi(Z_{n+1}h_{n+1}^2))\leq\frac{1}{2^{n-1}M_0},$$
		so the series $\sum_{n\geq 0}{\Delta_{\varepsilon}(n)}$ converges.\par
		For fixed integers $n\geq 0$ and $m\geq n+1$, Inequation~\eqref{critpr8} and Lemma~\ref{etagerg1bis} imply
		$$h'_{n}(\varphi(h_{m})+\varphi(Z_{m}))\leq\frac{h_{m}}{2^{m-1}q_0\ldots q_{m-2}}\leq\frac{q_{m-1}}{2^{m-1}M_0}.$$
		Combining this with the inequality $\varepsilon_m\leq 2/q_{m-1}$, we then get
		$$\Gamma_{\varepsilon}(n,m)=\varepsilon_{m}h'_{n}(\varphi(h_{m})+\varphi(Z_{m}))\leq\frac{1}{2^{m-2}M_0}.$$
		This gives
		$$\sum_{m\geq n+1}{\Gamma_{\varepsilon}(n,m)}\leq\frac{1}{2^{n-2}M_0}$$
		for every $n\geq 0$, so the series $\sum_{n\geq 0}{\sum_{m\geq n+1}{\Gamma_{\varepsilon}(n,m)}}$ converges.\par
		Therefore the cocycles $c_T$ and $c_S$ are $\varphi$-integrable as wanted, which concludes the proof.
	\end{proof}
	
	\begin{remark}
		For $\varphi$-integrability of $c_S$, we only need to control quantities of the form $\varphi(u^2)/u$ and $\varphi(u)/u$ ($\varphi(u^3)/u$ does not appear). Therefore Theorem~\ref{thfc} can be stated with a stronger quantification on the cocycle $c_S$, namely $\psi$-integrability with $\psi(t)=o(t^{1/2})$ (it suffices to replace $t^{1/3}$ by $t^{1/2}$ in Inequation~\eqref{critpr8}).
	\end{remark}
	
	We are now able to prove Theorem~\ref{uncountable}.
	
	\begin{proof}[Proof of Theorem~\ref{uncountable}]
		Let $\varphi\colon\R_+\to\R_+$ be a map satisfying $\varphi(t)\underset{t\to +\infty}{=}o\left (t^{1/3}\right )$. By Lemma~\ref{metriccompatible}, we may and do assume that $\varphi$ is increasing and subadditive.\par
		Given a flexible class $\mathcal{C}$, an associated set of parameters $\fc$ and constants $C$ and $C'$, the last proof shows that we can choose the parameters in the following way. First, we choose any cutting and spacing parameter $(q_0,(\sigma_{0,0},\ldots,\sigma_{0,q_0}))$ in $\fc$, with $q_0\geq 3$. Then, if $\bm{p_n}\coloneq (q_k,(\sigma_{k,0},\ldots,\sigma_{0,q_k}))_{0\leq k\leq n-1}$ has been set, there exists a constant depending on $\varphi$, $\fc$, $C$, $C'$ and $\bm{p_n}$, denoted by $\mathrm{K}_{\varphi}(\fc, C, C',\bm{p_n})$, such that Conditions~\eqref{critpr1}, \eqref{critpr2}, \eqref{critpr4}, \eqref{critpr4bis}, \eqref{critpr5}, \eqref{critpr6}, \eqref{critpr7}, \eqref{critpr8}, \eqref{critpr3} and~\eqref{critpr10} hold for every $q_n\geq \mathrm{K}_{\varphi}(\fc, C, C',\bm{p_n})$, and it remains to find such an integer $q_n$ and spacing parameters $\sigma_{n,0},\ldots,\sigma_{n,q_n}$ such that $\bm{p_{n+1}}\coloneq (q_k,(\sigma_{k,0},\ldots,\sigma_{0,q_k}))_{0\leq k\leq n}$ is in $\fc$ and the inequality $\sigma_n\leq C'q_nh_{n-1}$ holds.\par
		Let $\bm{Q}\coloneq (Q_{-1},\ldots,Q_{n_0})$ be a sequence of integers, where $n_0,Q_0,\ldots,Q_{n_0}$ are positive and $Q_0\ldots Q_{n_0}\geq 3$, and let us consider the set of parameters $\mathcal{F}(\bm{Q})$ built in Section~\ref{proofirra}, and the associated constants $C_{\bm{Q}}$ and $C'_{\bm{Q}}$. In this case, the spacing parameters $\sigma_{k,i}$ at step $k$ are equal to $0$ or $h_{k-1}$, so they are determined by the previous cutting parameters. Moreover, the first cutting parameter $q_0$ is equal to $Q_0\ldots Q_{n_0}$. Therefore, for every finite sequence $\bm{p_n}\coloneq (q_k,(\sigma_{k,0},\ldots,\sigma_{0,q_k}))_{0\leq k\leq n-1}$ in $\mathcal{F}(\bm{Q})$, we write $\mathrm{K}_{\varphi}(\bm{Q},q_1,\ldots,q_{n-1})$ instead of $\mathrm{K}_{\varphi}(\mathcal{F}(\bm{Q}), C_{\bm{Q}}, C'_{\bm{Q}},\bm{p_n})$.\par
		Recall that $A$ denotes the set of sequences $(q_{i})_{i\geq -1}$ of integers such that $q_0,q_1,\ldots$ are positive. To every sequence $\bm{\varepsilon}=(\varepsilon_i)_{i\geq 0}\in\{0,1\}^{\N}$, we associate a sequence $q(\bm{\varepsilon})\in A$ inductively defined by:
		$$\begin{array}{rl}
			&q(\bm{\varepsilon})_0=q_0,\\
			\forall i\geq 0,&q(\bm{\varepsilon})_{i+1}=\mathrm{K}_{\varphi}(\bm{Q},q(\bm{\varepsilon})_1,\ldots,q(\bm{\varepsilon})_{i})+\varepsilon_{i}
		\end{array}.$$
		Every sequence $\bm{\varepsilon}=(\varepsilon_i)_{i\geq 0}\in\{0,1\}^{\N}$ provides a sequence of parameters in $\mathcal{F}(\bm{Q})$, whose cutting parameters are $q(\bm{\varepsilon})_0,q(\bm{\varepsilon})_1,\ldots$, and which gives rise to the irrational rotation of angle $\theta(\bm{\varepsilon})\coloneq [Q_{-1},\ldots,Q_{n_0},q(\bm{\varepsilon})_1,q(\bm{\varepsilon})_2,\ldots]$.\par
		Let us now consider a nonempty open subset $\mathcal{V}$ of $\R$ and a finite sequence $\bm{Q}$ so that $\theta(\bm{\varepsilon})$ is in $\mathcal{V}$ for every $\bm{\varepsilon}\in\{0,1\}^{\N}$. We get that the set of irrational numbers $\theta$ in $\mathcal{V}$ such that the irrational rotation of angle $\theta$ is $\varphi$-integrably orbit equivalent to the universal odometer contains the set $\{q(\bm{\varepsilon})\mid\bm{\varepsilon}\in\{0,1\}^{\N}\}$, so it is uncountable using the facts that the map $\bm{\varepsilon}\in\{0,1\}^{\N}\mapsto q(\bm{\varepsilon})\in A$ is injective and the continued fraction expansion is unique for every irrational number.
	\end{proof}
	
	\bibliographystyle{alphaurl}
	\bibliography{biblio}

\end{document}